%% file: main_arxiv.tex
\pdfoutput=1

\documentclass[12pt]{alt2024}  
\usepackage{fullpage}

\usepackage{amsmath,amsfonts,amssymb,thmtools, mathtools, mathrsfs}

\usepackage{dsfont} 

\usepackage{url}					%
\hypersetup{hidelinks,linktocpage,colorlinks=true, allcolors=blue, linktoc = all}

\usepackage{enumitem}

\usepackage{xargs}
\usepackage{xfrac} %

\usepackage[bb=boondox]{mathalfa}

\usepackage{hhline}

\usepackage{multirow}

\let\epsilon\varepsilon

\makeatother

\usepackage{tikz}			%

\usepackage[capitalise,nameinlink,noabbrev]{cleveref}  %
\crefname{equation}{}{}

\newtheorem{theorem}{Theorem}
\newtheorem{lemma}[theorem]{Lemma}
\newtheorem{proposition}[theorem]{Proposition}
\newtheorem{remark}[theorem]{Remark}

\newtheorem{definition}[theorem]{Definition}

\newtheorem{fact}[theorem]{Fact}

\crefname{enumi}{Property}{Properties} %

\DeclareMathOperator{\vol}{vol}

\DeclareMathOperator*{\argmax}{arg\,max}
\DeclareMathOperator*{\argmin}{arg\,min}

\newcommandx{\E}[2][1=, 2=, usedefault]{\mathbb{E}_{#2}[ #1 ]} 
 
\newcommandx{\El}[2][1=, 2=, usedefault]{\mathbb{E}_{#2}\left[ #1 \right]} 
\newcommand{\R}{\mathbb{R}} 
\newcommand{\Z}{\mathbb{Z}} 
 
\newcommand{\X}{\mathcal{X}} 
\newcommand{\pa}[1]{\left( #1\right)} 
\newcommand{\innp}[1]{\langle #1 \rangle}
\newcommand{\norm}[1]{\| #1 \|} 
\newcommand{\norml}[1]{\left\| #1 \right\|} 
\newcommand{\abs}[1]{| #1 |}

\newcommand{\defi}{\stackrel{\mathrm{\scriptscriptstyle def}}{=}}

\title[Non-Euclidean High-Order Smooth Convex Optimization]{Non-Euclidean High-Order Smooth Convex Optimization}

\altauthor{%
 \Name{Juan Pablo Contreras$^\ast$} \Email{\href{mailto:jcontrere@uc.cl}{jcontrere@uc.cl}}\\
 \addr Institute for Mathematical and Computational Engineering \\ Pontificia Universidad Católica de Chile
 \AND
 \Name{Cristóbal Guzmán$^\ast$} \Email{\href{mailto:crguzmanp@mat.uc.cl}{crguzmanp@mat.uc.cl}}\\
 \addr Institute for Mathematical and Computational Engineering\\ Faculty of Mathematics and School of Engineering \\ Pontificia Universidad Católica de Chile
 \AND
 \Name{David Martínez-Rubio$^\ast$} \Email{\href{mailto:dmrubio@ing.uc3m.es}{dmrubio@ing.uc3m.es}}\\
 \addr Signal Theory and Communications Department, \\
 Universidad Carlos III de Madrid, Spain
}

\newcommand\blfootnote[1]{%
\begingroup
\renewcommand\thefootnote{}\footnote{#1}%
\addtocounter{footnote}{-1}%
\endgroup
}

\input{algorithm_config.tex}
\input{todo_config.tex} %
\input{bibliography_config.tex}
\input{definitions.tex}

\begin{document}

\maketitle

\blfootnote{$^\ast$Equal contribution.}

\blfootnote{\color{darkgray}Most of the non-local notations in this work have a link to their definitions, using \href{https://damaru2.github.io/general/notations_with_links/}{this code}, such as ${\protect\hyperlink{def:qth_taylor_expansion}{\color{darkgray}f_q(y;x)}}$, which links to where this notation is defined as the ${\protect\hyperlink{def:order_of_derivative_for_smoothness}{\color{darkgray}q}}$-th order Taylor expansion of $f$ around $x$.}

\begin{abstract}
    We develop algorithms for the optimization of convex objectives that have H\"older continuous $q$-th derivatives by using a $q$-th order oracle, for any $q \geq 1$. Our algorithms work for general norms under mild conditions, including the $\ell_p$-settings for $1\leq p\leq \infty$.
    We can also optimize structured functions that allow for inexactly implementing a non-Euclidean ball optimization oracle. We do this by developing a non-Euclidean inexact accelerated proximal point method that makes use of an \textit{inexact uniformly convex regularizer}. 
    We show a lower bound for general norms that demonstrates our algorithms are nearly optimal in high-dimensions in the black-box oracle model for $\ell_p$-settings and all $q \geq 1$, even in randomized and parallel settings. This new lower bound, when applied to the first-order smooth case, 
    resolves an open question in parallel convex optimization.
\end{abstract}

\section{Introduction}

In optimization, objectives with high-order smoothness offer the possibility of faster convergence rates, at the expense of computation of higher-order derivatives. Recently, this area of research has gained significant interest, both due to the discovery of acceleration techniques for high-order methods, and also due to the active development of tensor methods which are the working horse for the subroutines required in this context \citep{ahmadi2023higher,cartis2023second,zhu2023cubicquartic,zhu2024global}. Despite the substantial activity in this field, there has been scarce  investigation of the role of these ideas for non-Euclidean norms (more precisely, norms that are not Hilbertian). Given the proved advantages of exploiting non-Euclidean structure in various applications, see e.g.~\citep{BenTal:2001,Nesterov:2005Smoothing, Nemirovski:2004, Sherman:2017}, we consider this as a major gap in the current optimization toolbox.

In this work, we study the optimization of a general convex $\q$-times differentiable function $f$ whose $\q$-th derivative is $(\L, \nu)$-H\"older continuous with respect to a norm $\norm{\cdot}$, that is,
\begin{equation} \label{eq:holder_cont}
    \norm{\nabla^{\q} f(x) - \nabla^{\q} f(y)}_{\ast} \leq \L \norm{x -y}^{\nu} \text{ for all } x, y \in \Rd,
\end{equation}
where $\q \in \Z_+$, $\nu \in (0, 1]$, and where the norm of a multilinear operator $F: \R^{d \otimes \q} \to \R$ like $F = \nabla^{\q} f(x) - \nabla^{\q} f(y)$ is defined as $ \norm{F}_{\ast} \defi \max_{\norm{v}\le 1}|F[v]^{\otimes \q}|$.
In this case, we say $f$ is $\q$-th order $(\L, \nu)$-H\"older smooth with respect to $\norm{\cdot}$. We make use of an oracle that returns all derivatives of $f$ at a point up to order $\q$. For the case of $p$-norms, we specialize our results and characterize the optimal oracle complexity, up to logarithmic factors.  We also study the optimization of convex functions with a reduction to inexact $\p$-norm $\rho$-ball optimization oracles. That is, using an oracle to approximately minimizing the function in balls of a fixed radius $\rho$ with respect to a $\p$-norm, we minimize the function globally. The oracle can be implemented fast for some functions with structure. More concretely, our contributions can be summarized as in the following.

\subsection{Our Contributions}

\paragraph{Upper Bounds} 

We develop a general \textit{non-Euclidean} inexact accelerated proximal point method and apply it for the optimization of $\q$-th order H\"older-smooth convex functions, and of structured functions for which we can implement a ball optimization oracle. The algorithm makes use of an \textit{inexact uniformly convex regularizer}, a property that we introduce that is key to solve several cases, in particular the $\ell_{\p}$ setting for $\p \geq \q + \nu$. We also develop an inexact unaccelerated proximal point method, that achieves near optimality for the case $\p=\infty$, not covered by the accelerated method.

Each iteration of our algorithms only requires one call to the $\q$-th order, or ball optimization oracle. When the square of the norm considered is strongly-convex with respect to itself, we establish convexity of the regularized Taylor subproblems appearing at each iteration of the high-order smooth convex case. In the $q$-th order $(\L, \nu)$-H\"older smooth convex setting with respect to $\norm{\cdot}_{\p}$, the near-optimal convergence rates that we establish for achieving an $\newtarget{def:accuracy}{\epsilon}$-minimizer are
\[
    \bigotildepl{\q+\nu, \p}{\left(\frac{\L R_{\p}^{\q+\nu}}{\epsilon}\right)^{\tfrac{\m}{(\m+1)(\q+\nu)-\m}}} \text{ if } \p \in [1, \infty), \text{ and } \bigopl{\q+\nu}{\left(\frac{\L R_\infty^{\q+\nu}}{\epsilon}\right)^{\tfrac{1}{\q+\nu-1}}}  \text{ if } \p = \infty,
\] 
where $\m \defi \max\{2, \p\}$, $R_{\p} \defi \norm{x_0-\xast}_{\p}$ is the initial distance to a minimizer $\xast$, and where log factors only appear for $\p=1$. Similarly for $\rho$-ball optimization oracles, which can be thought as the case $\q \to \infty$, we achieve rates $\bigotildep{\m}{(R_{\p}/\rho)^{\frac{\m}{\m+1}}}$ and $\bigotilde{R_{\infty}/\rho}$ for $\p\in[1, \infty)$ and $\p = \infty$.

\paragraph{Lower Bounds}  
As is customary in convex optimization, we study the suboptimality of our algorithms in the black-box oracle model \citep{Nemirovski:1983} and provide a lower bound for convex high-order H\"older smooth functions in high dimensions, with respect to a general norm, even for randomized and parallel settings with access to a local oracle, implying near-optimality of our algorithms for \(\ell_p\) settings. Our approach constructs lower bounds by composing a non-Euclidean randomized smoothing with a hard Lipschitz instance with respect to an arbitrary norm, built as the maximum of softmax-like functions applied to an increasing sequence of linear functions. A key technical innovation is proving that any piecewise linear function can be smoothed while preserving its norm-dependent Lipschitz properties, unlike previous techniques restricted to the \(\ell_2\) setting.  

Another contribution is the analysis of a non-Euclidean randomized smoothing operator that can be iterated to obtain high-order smooth functions from a Lipschitz function. By leveraging the divergence theorem, we establish a smoothing technique that applies seamlessly to all norms, and in particular to all $\ell_p$ settings, \( p \in [1, \infty] \), whereas previous lower bounds via smoothing techniques only worked for \( p \geq 2 \) and required intricate reductions via high-dimensional embeddings to the \( p = \infty \) case to handle \( p \in [1,2) \), even for simpler cases, like those applying to deterministic algorithms in the first-order smooth case, see \cref{sec:related_work}. Our approach offers a unified treatment and, to the best of our knowledge, is the first to address the case of order \( \q \geq 2 \) in this setting. Moreover, our results strengthen existing first-order lower bounds, establishing a nearly optimal \( \bigomegatilde{\epsilon^{-1/2}} \) rate for first-order parallel smooth convex optimization in the \( \ell_1 \) setting, thereby improving upon previous work, cf. \citep{diakonikolas2020lower}.

\subsection{Related Work}\label{sec:related_work}

    We note that \cite{baes2009estimate} was the first work to develop (unaccelerated) general high-order methods under convexity and high-order smoothness, defined with respect to the Euclidean norm. 
    While it is enough to approximate a critical point of the proximal subproblems appearing in \citep{baes2009estimate}, \citet{nesterov2021implementable} showed that by choosing the right regularization parameter, the subproblems become convex. Although convexity is not required in order to find an approximate critical point in a tractable way in several optimization contexts, it usually enables to solve such a problem faster, cf.  \citep{carmon2021lower}.
    Previously, \citet{monteiro2013accelerated} developed a general accelerated inexact proximal point algorithm, for which they achieved near optimal second-order oracle complexity for convex functions with a Lipschitz Hessian with respect to the Euclidean norm. Building on this framework, three works \citep{gasnikov2019optimal,bubeck2019near,jiang2019optimal} independently achieved near optimal $\q$-th order oracle complexity for high-order Euclidean smooth convex optimization.  Later \citet{kovalev2022first,carmon2022optimal} concurrently achieved optimal $\q$-th order oracle complexity, up to constants, improving over previous solutions by logarithmic factors, via two very different techniques.
    \citet{song2019unified} studied the problem for functions with $\p$-norm regularity, but they only solved the case where $\p \leq \q+1$, where $\q$ is the degree of the high-order oracle. Besides, each iteration of their algorithm requires solving two regularized Taylor expansions of the function with different regularization functions and a binary search. 
    \citet{adil2022optimal} designed and algorithm for the setting of high-order non-Euclidean smooth monotone variational inequalities with strongly-convex regularizers. \citet{carmon2020acceleration} introduced the optimization framework with Euclidean ball optimization oracles, and this technique has enabled the design of algorithms in several different settings \citep{carmon2021thinking,carmon2022distributionally,martinez2020global,carmon2024whole}.

    Regarding lower bounds, \citet{arjevani2019oracle} showed a lower bound for deterministic algorithms for convex functions with Lipschitz $\q$-th derivatives with respect to the Euclidean norm, by providing a hard function in the form of a $(\q+1)$-degree polynomial. Independently, \citet{agarwal2018lower} developed some suboptimal lower bounds by an interesting technique consisting of compounding randomized smoothing by repeated convolution of a hard convex Lipschitz instance resulting in a function with Lipschitz high-order derivatives. In this spirit and inspired by them, \citet{garg2021nearoptimal} developed a nearly optimal lower bound via applying randomized smoothing to a construction similar to the classical Lipschitz instance consisting of a maximum of linear functions, but using the maximum of a variant of these functions via applying several softmax. They achieve, up to logarithmic factors, the lower bound in \cite{arjevani2019oracle}, but they also provide lower bounds for parallel randomized algorithms, and for quantum algorithms. In the non-Euclidean setting, existing lower bounds typically rely on inf-convolution smoothing for \( p \geq 2 \), whereas for \( p \in [1,2) \), they use high-dimensional embeddings since an inf-convolution smoothing kernel is known to be unattainable in this regime without incurring polynomial dependence on the dimension, as implicitly shown in \cite[Example 5.1]{aspremont2018optimal}. These techniques have been applied to establish lower bounds for deterministic sequential methods \citep{Nemirovski:1983,guzman2015lower} and parallel randomized methods \citep{diakonikolas2020lower}.

\paragraph{Concurrent independent work.} We note that the concurrent work \citep{adil2024convex}, independently showed convergence of an analogous accelerated non-Euclidean (exact) proximal algorithm. As opposed to them, we also introduced the notion of inexact uniformly-convex regularizers, proved convergence when we use them, even when we have an inexact implementation of the proximal oracles, and we show our subproblems are convex for several cases. \citet{adil2024convex} also apply their framework to the optimization of non-Euclidean high-order smooth convex functions by exactly solving a regularized Taylor expansion of the function. However, we studied the more general $\q$-order $\nu$-H\"older smooth case with respect to a $\p$-norm and established the optimal or near-optimal convergence, by inexactly solving a regularized Taylor expansion, for all cases $\p \geq 1$, $\q \geq 1$, $\nu \in (0, 1]$, where the smooth case corresponds to $\nu =1$. The high-order smooth convex optimization analysis in \citep{adil2024convex} is limited to $\p \geq 2$ and $\q + 1 \geq \p$. On the other hand, they studied the application of this framework to $\p$-norm regression.

\section{Preliminaries and Groundwork}\label{sec:preliminaries_and_groundwork}

Throughout, we consider a finite-dimensional normed space $(\R^{\newtarget{def:dimension}{\dim}},\|\cdot\|)$ with an inner product $\langle \cdot,\cdot\rangle$ that, importantly, does not necessarily induce the norm. Most of our proofs work for general norms, although we sometimes specialize to the case $\|\cdot\|=\|\cdot\|_{\p}$, where $1\leq \p\leq\infty$.

\paragraph{Notation.} In this work, we often use functions that are regular with respect to $\newtarget{def:p_from_p_norm}{\p}$-norms such as $\newtarget{def:order_of_derivative_for_smoothness}{\q}$-th order $(\newtarget{def:hoelder_smoothness_constant}{\L}, \newtarget{def:hoelder_smoothness_exponent}{\nu})$-H\"older smoothness, and we use regularizers that are, possibly $\newtarget{def:delta_inexact}{\deltainexact}$-inexact $(\newtarget{def:regularizers_unif_cvx_constant}{\muUnif}, \newtarget{def:exponent_of_regularizers_unif_convexity}{\r})$-uniformly convex. We reserve the letters $\p, \q, \r, \deltainexact, \muUnif, \L, \nu$ for this. We always use $\newtarget{def:m_max_of_p_and_2}{\m} \defi \max\{2, \p\}$. We denote $\newtarget{def:event_indicator_function}{\eventindicator{A}}$ the event indicator that is $1$ if $A$ holds true and $0$ otherwise.
We denote $\newtarget{def:qth_taylor_expansion}{\taylorf[\q][f]}(y;x) \defi  \sum_{i=0}^{\q} \frac{1}{i!} \nabla^i f(x)[y-x]^{\otimes i}$ the $\q$-th order Taylor expansion of $f$ at $y$ around $x$. We use $\newtarget{def:xast}{\xast}$ for a minimizer of a function when is clear from context and it exists. We use $\newtarget{def:big_O_p}{\bigop{\p}{\cdot}}$ and $\newtarget{def:big_O_tilde}{\bigotilde{\cdot}}$ as the big-$O$ notation omitting, respectively, factors depending on $\p$ and logarithmic factors.  Given a differentiable function $\psi$, we denote the Bregman divergence of $\psi$ at $x, y$ by $\newtarget{def:bregman}{\breg[\psi]}(x,y) \defi \psi(x) - \psi(y) - \innp{\nabla \psi(y), x-y}$.
\begin{definition}[Young's conjugate number] \label{def:youngs_conjugate_number}
    Given $\p \in [1, \infty]$, we define its Young's conjugate as $\newtarget{def:youngs_conjugate_number}{\dualnumber{p}} \defi (1 - 1/\p)^{-1}$ so that $\frac{1}{\p} + \frac{1}{\dualnumber{p}} = 1$. For $\p = 1$ it is $\dualnumber{p} = \infty$ and vice versa. It is well known that the dual norm of $\|\cdot\|_{\p}$ is $\|\cdot\|_{\dualnumber{p}}$.
\end{definition}

\begin{definition}[Enlarged subdifferential] Given a function $f:\Rd \to \R$ and $\gamma \geq 0$, we define the $\gamma$-enlarged subdifferential of $f$ as
    \[
        \newtarget{def:epsilon_subdifferential}{\subdiffeps[\gamma]}f(y) \defi \setsuch{g\in\Rd}{f(z) \geq f(y) + \innp{g, z-y} - \gamma} \ \text{ for all } z\in\Rd.
    \] 
    We say any $g\in\subdiffeps[\gamma]f(y)$ is a $\gamma$-enlarged subgradient of $f$ at $y$.
\end{definition}

\begin{definition}[Non-Euclidean Moreau envelope]\label{def:moreau_env_and_prox}
    Given a norm $\norm{\cdot}$, and 
    a parameter $\lambda\geq 0$, define the Moreau envelope of a convex, proper, and closed function $f:\Rd \to \R \cup \{+\infty\}$ as
    \begin{equation}\label{def:non_eucl_moreau_envelope}
        \newtarget{def:non_euclidean_moreau_envelope}{\M[\lambda]}(x) \defi \min_{y\in\Rd}\{ f(y) + \frac{1}{2\lambda}\norm{x-y}^2\},
    \end{equation}
    where for $\lambda = 0$ we define $\M[0](x) \defi f(x)$. Similarly, we define $\newtarget{def:set_of_all_prox_points}{\Prox}_\lambda(x) \defi \argmin_{y\in\Rd} \{ f(y) + \frac{1}{2\lambda}\norm{x-y}^2\}$ and $\newtarget{def:proximal_operator}{\prox[\lambda]}(x)\in \Prox_\lambda(x)$ to be an arbitrary element.  We omit subindices if $\lambda$ is clear from context. 
\end{definition}

We now present some properties of this envelope. The proof can be found in \cref{app:proofs_of_preliminaries}.

\begin{proposition}[Envelope properties]\linktoproof{prop:properties_of_M}\label{prop:properties_of_M}
    Using \cref{def:moreau_env_and_prox} and letting $\xast$ be a minimizer of $f$, the following holds:
    \begin{enumerate}
        \item\label{properties_of_M:1} If $\norm{\cdot} = \norm{\cdot}_{\p}$, for $\p \in (1, \infty)$, $\Prox_\lambda(x)$ contains a single element. This may not be the case for $\p =1$ or $\p=\infty$.
        \item \label{properties_of_M:2}$\M[\lambda](x)$ is convex.
        \item\label{properties_of_M:3}  $f(\prox[\lambda](x)) \leq \M[\lambda](x) \leq f(x)$. In particular, $f(\xast) = \M[\lambda](\xast)$.
        \item \label{properties_of_M:4} Let $\newtarget{def:subdiff_of_squared_norm}{\subdiffsqnorm{x}}(y) \defi \partial_x \frac{\norm{x-y}^2}{2\lambda}$ be the subdifferential of $\frac{\norm{\cdot - y}^2}{2\lambda}$ at x. Then 
            $\partial \M[\lambda](x)=\operatorname{conv}\{\subdiffsqnorm{x}(z): z\in\Prox_\lambda(x)\}$ and there is $ g \in \subdiffsqnorm{x}(\prox[\lambda](x))$ such that $g \in \partial f(\prox[\lambda](x))$. 
            \item \label{properties_of_M:5}For all $y \in \Rd$ and $g \in \subdiffsqnorm{x}(y)$, it is $\lambda\innp{g, y-x} =  \norm{x-y}^2 = \lambda^2\norm{g}_\ast^2$. In particular, for any $g \in \subdiffsqnorm{x}(\prox[\lambda](x)) \subseteq \partial \M[\lambda](x)$ we have $\norm{g}_\ast = \frac{1}{\lambda} \norm{x-\prox[\lambda](x)}$.
            \item \label{properties_of_M:6} For any $\lambda_1 > 0$, $\lambda_2 \geq 0$, we have the following descent condition:
            \[
                \M[\lambda_1](x) - \M[\lambda_2](\prox[\lambda_1](x)) \geq \frac{1}{2\lambda_1} \norm{x-\prox[\lambda_1](x)}^2.
            \] 
        \end{enumerate}
    \end{proposition}

    Given a function class $\mathcal{F}$ and a set $\X$, a local oracle is a functional, mapping $(f,x) \mapsto \mathcal{O}_f(x)$ to a vector space, such that when queried with the same point $x \in \X$ for two different functions $f, g \in \mathcal{F}$ that are equal in a neighborhood of $x$, it returns the same answer \citep{Nemirovski:1983,Nemirovski:1995}. An example of such an oracle that we use for our upper bounds is a $\q$-th order oracle, for $\q \in \Z_{+}$. Given the family $\mathcal{F}$ of functions that are $\q$-times differentiable, the $\q$-th order oracle is defined as $\mathcal{O}_f(x) = (f(x), \nabla f(x), \dots,\nabla^{\q} f(x))$. The main problem we study is the optimization of high-order H\"older-smooth functions convex functions by making use of a $\q$-th order oracle. 
    Similarly to the definition of H\"older smoothness, we say a function is $\L$-Lipschitz with respect to $\norm{\cdot}_{\p}$ if $\abs{f(x) - f(y)} \leq \L\norm{x-y}_{\p}$. For a convex function that has $(\L, 1)$-H\"older continuous first derivative with respect to some norm, we simply say that the function is $\L$-smooth with respect to that norm.
Our algorithms make use of regularizers with a new property that we introduce below, which is key to fully solve all cases of high-order smooth convex optimization.

\begin{definition}[Inexact uniform convexity]
    Given $\muUnif, \sigma, \delta >0$, a differentiable function $\psi$ is said to be $\delta$-inexact $(\muUnif, \sigma)$-uniformly convex with respect to a norm $\norm{\cdot}$, in a convex set $\X$, if for all $x, y \in \X$ we have 
\[
    \breg[\psi](x, y) \geq \frac{\muUnif}{\sigma}\norm{x-y}^\sigma - \deltainexact.
\] 
When $\delta = 0$ and $\sigma \geq 2$, we recover the classical notion of uniform convexity.
\end{definition}
Exact uniform convexity implies the inexact property with respect to smaller exponents.
\begin{lemma}\label{lemma:inexact_unif_convex_from_unif_convex}\linktoproof{lemma:inexact_unif_convex_from_unif_convex}
    Let $\psi$ be a function that is $(1, \sigma)$-uniformly convex, $\sigma \geq 2$. If $0 < s < \sigma$, then $\psi$ is also $(a^{\frac{\sigma^2}{s(\sigma-s)}} \frac{\sigma-s}{s\sigma})$-inexact $(a^{\frac{\sigma}{s}} , s)$-uniformly convex for any $a > 0$.
\end{lemma}
Note that although $(\muUnif, \sigma)$-uniform convexity is a property that requires $\sigma\geq 2$, our definition of $\deltainexact$-inexact $(\muUnif, s)$-uniform convexity, and the example provided in the previous lemma, allows for any $s>0$. Our algorithms work with inexact uniformly convex regularizers for $s > 1$. In particular, we will use the following regularizers for simplicity, but we note that our accelerated method works for any norm, given that we provide an inexact uniformly convex regularizer. Our unaccelerated method works for any norm.
A proof of the following well-known fact can be found in \cref{app:proofs_of_preliminaries}.

\begin{fact}[Regularizers' properties]\label{lemma:regularity_of_regularizers}\linktoproof{lemma:regularity_of_regularizers}
    If $\p \geq 2$, the regularizer $\psi(x) = \frac{1}{\p}\norm{x-x_0}_{\p}^{\p}$, is $(2^{2-\p}, \m)$-uniformly convex  regularizer  in $\Rd$ with respect to $\norm{\cdot}_{\p}$ , and if $\p \leq 2$,  $\psi(x) = \frac{1}{2(\p-1)}\norm{x-x_0}_{\p}^2$ is $(1, \m)$-uniformly convex in $\Rd$ with respect to $\norm{\cdot}_{\p}$, where $\m \defi \max\{2, \p\}$.
\end{fact}

\section{Accelerated Inexact Proximal Point with an Inexact Uniformly Convex Regularizer} \label{sec:acc_inexact_PP_unif_cvx}
We study an accelerated optimization method that interacts with a function $f$ via a non-Euclidean inexact proximal oracle, in the spirit of \citep{monteiro2013accelerated}. The algorithm approximately optimizes the non-Euclidean Moreau envelope convolving with respect to a power of the norm being considered, instead of with respect to the more traditional choice of a strongly-convex or other types of functions, see e.g. \citep{teboulle2018simplified}. Explained from the point of view of linear coupling \citep{allenzhu2017linear}, the intuition of the analysis is that this choice makes the gradient norm of the Moreau envelope approximation satisfy some crucial property analogous to \cref{properties_of_M:5}, that makes the regret of the mirror descent algorithm in Line \ref{line:mirror_descent_step} be small enough. On the other hand, this Moreau envelope is not smooth in general, but still applying the oracle of Line \ref{line:criterion_unif_convex}, we obtain enough descent to compensate for the aforementioned regret and the approximation error. 

\paragraph{Inexact Proximal Oracle} \emph{ Given a function $f$, the oracle $y_k, v_k \gets \newtarget{def:inexact_proximal_oracle}{\iproxoracle[r]}(x_k, \lambda_k)$  returns an inexact proximal point $y_k$ of the proximal problem $\min_y\{f(y) + \frac{1}{\r\lambda_k}\norm{y-x_k}^{\r}\}$, and an enlarged subgradient $v_k \in \subdiffeps[\oldepsilon_k] f(y_k)$. Given $\sigma, \sigma' \in [0, 1/2)$, a norm $\norm{\cdot}$, and exponent $\r$, the requirement on the oracle is
\begin{equation}\label{eq:inexact_prox_oracle_properties}
    \norm{v_k -\hat{v}_k}_\ast \leq \frac{\sigma}{\lambda_k}\norm{x_k-y_k}^{\r-1} \text{ for some }\hat{v}_k \in \subgrad_{y}(-\frac{1}{\r\lambda_k}\norm{y-x_k}^{\r})(y_k), \text{ and } \oldepsilon_k\leq \frac{\sigma'}{\lambda_k}\norm{x_k-y_k}^{\r}.
\end{equation}
}

\begin{algorithm}[ht!]
        \caption{Non-Euclidean Accelerated Inexact Proximal Point with Inexact Uniformly Convex Regularizer}
    \label{alg:non_euclidean_accelerated_proximal_point_unif_convex}
\begin{algorithmic}[1] 
    \REQUIRE Convex function $f$. Regularizer $\psi$ that is a $\deltainexact$-inexact $(\muUnif,\r)$-uniformly convex function wrt a norm $\norm{\cdot}$, and $\r > 1$. Inexactness constants $\sigma, \sigma'$ and proximal parameters $\lambda_k > 0$.
    \vspace{0.1cm}
    \hrule
    \vspace{0.1cm}
    \State $\zk[0] \gets \yk[0] \gets \xk[0]$; \quad $\Ak[0] \gets 0$; \quad $C \gets \frac{\muUnif}{2}\left(\frac{\dualnumber{r}(1-\sigma-\sigma')}{1+\sigma^{\dualnumber{r}}}\right)^{\r-1}$ 
    \FOR {$k = 1 \textbf{ to } T$}
        \State $\newtarget{def:integral_of_steps}{\Ak[k]} = \ak[k] + \Ak[k-1]$
        \State $\newtarget{def:discrete_step}{\ak[k]} = (C^{\r-1}\Ak[k][\r-1] \lambda_k)^{1/\r}$ \Comment{$\r$-degree equation on $a_k > 0$.}
        \State $\xk \gets \frac{\Ak[k-1]}{\Ak[k]}\yk[k-1] + \frac{\ak}{\Ak[k]}\zk[k-1] $
        \State $y_k, v_k \gets \iproxoracle[r](x_k, \lambda_k)$ \Comment{Oracle satisfying \cref{eq:inexact_prox_oracle_properties}}  \label{line:criterion_unif_convex} %

        \State $\zk \gets \argmin_{z\in\Rd} \{\sum_{i=1}^{k}\ak[i]\innp{ v_i, z} + \breg(z, \xk[0]) \}$  \label{line:mirror_descent_step} %
    \ENDFOR
    \State \textbf{return} $\yk[T]$.
\end{algorithmic}
\end{algorithm}

It is straightforward to check that an exact solution of the proximal problem satisfies the properties in \cref{eq:inexact_prox_oracle_properties} for $\sigma = \sigma' = 0$. 
We also have the following, by \cref{prop:properties_of_M}, \cref{properties_of_M:5} and $\hat{v}_k \in\partial(-\frac{1}{\lambda_k \r}\norm{y-x_k}^{\r})(y_k) = \norm{y_k-x_k}^{\r-2}\partial(-\frac{1}{2\lambda_k }\norm{y-x_k}^2)(y_k)$:
\begin{equation}\label{eq:property_subgrads_of_norm_to_the_r}
    \norm{\hat{v}_k}_\ast = \frac{1}{\lambda_k}\norm{x_k - y_k}^{\r-1} \text{ and } \innp{\hat{v}_k, y_k-x_k} = -\frac{1}{\lambda_k} \norm{x_k - y_k}^{\r}.
\end{equation}
Making use of this oracle, we show the following convergence rate. In \cref{sec:application_to_high_order_methods}, we discuss how to implement such an oracle in different settings.

\begin{theorem}\label{thm:convergence_of_accelerated_not_adaptive_algorithm}\linktoproof{thm:convergence_of_accelerated_not_adaptive_algorithm}
    Let $f:\Rd \to \R$ be a convex function and let $\newtarget{def:acc_sec_regularizer}{\psi}$ be a $\deltainexact$-inexact $(\muUnif, \r)$-uniformly convex regularizer with respect to a norm $\norm{\cdot}$, for $\r>1$. Given some constants $\sigma, \sigma'$ and proximal parameters $\lambda_i> 0$, the iterates $y_t$ of \cref{alg:non_euclidean_accelerated_proximal_point_unif_convex} satisfy for any $u \in \Rd$:
    \[
        f(y_t)-f(u) = \bigopl{\r}{\frac{\breg[\psi](u, x_0) + \deltainexact t}{\muUnif\left(\sum_{k=1}^{t} \lambda_k^{1/\r}\right)^{\r}}}.
    \]
    In particular, it holds for a minimizer $u = \xast$ of $f$, if it exists.
\end{theorem}

\subsection{Adaptive version} \label{sec:adaptivity_unif_cvx}
    In some cases of the high-order smooth convex setting, we neither have full control nor prior knowledge over the value of $\lambda_k$ for which we can implement the oracle $\iproxoracle[r](x_k, \lambda_k)$: this value becomes an output rather than an oracle input, since $\lambda_k$ turns out to be a function of the traveled distance $\norm{x_k - y_k}$. This causes an implicit mutual dependence between $y_k$ and $\lambda_k$. For this reason, inspired by the adaptive Euclidean analysis in \citep{carmon2022optimal}, we develop a generalized adaptive algorithm that uses a guess for the proximal parameter, and comes with stronger guarantees.

\paragraph{Generalized Inexact Proximal Oracle} \emph{ 
Given a function $f$, the oracle $\yktilde[k], v_k, \lambda_k \gets \newtarget{def:generalized_inexact_proximal_oracle}{\iproxoraclehat[r]}(x_k, \widehat{\lambda}_k)$ returns a proximal parameter $\lambda_k$, an inexact proximal point $y_k$ of the proximal problem $\min_y\{f(y) + \frac{1}{\r\lambda_k}\norm{y-x_k}^{\r}\}$, and an enlarged subgradient $v_k \in \subdiffeps[\oldepsilon_k] f(y_k)$, possibly using $\widehat{\lambda}_k$ for these estimations.  Given $\sigma, \sigma' \in [0, 1/2)$, a norm $\norm{\cdot}$, and exponent $\r$, the output satisfies
\begin{equation}\label{eq:inexact_general_prox_oracle_properties}
    \norm{v_k -\hat{v}_k}_\ast \leq \frac{\sigma}{\lambda_k}\norm{x_k-\yktilde[k]}^{\r-1} \text{ for some }\hat{v}_k \in \subgrad_{y}(-\frac{1}{\r\lambda_k}\norm{y-x_k}^{\r})(\yktilde[k]), \text{ and } \oldepsilon_k \leq \frac{\sigma'}{\lambda_k}\norm{x_k-\yktilde[k]}^{\r}.
\end{equation}
}
Note that for a convex function $f$, the points in $\argmin\{f(y) + \frac{1}{\r\lambda_k}\norm{y-x_k}^{\r} \}$ satisfy the properties above. 
Then, we obtain the following theorem for \cref{alg:adaptive_non_euclidean_accelerated_proximal_point_unif_convex}. %

\begin{algorithm}[ht!]
        \caption{Non-Euclidean Adaptive Accelerated Proximal Point with Uniformly Convex Regularizer}
    \label{alg:adaptive_non_euclidean_accelerated_proximal_point_unif_convex}
\begin{algorithmic}[1] 
    \REQUIRE Convex function $f:\Rd\to\R$. Regularizer $\psi$ that is $(1,\r)$-uniformly convex function wrt a norm $\norm{\cdot}$. Initial $\widehat{\lambda}_0$. Adjustment constant factor $\alpha > 1$. Inexactness constants $\sigma, \sigma'$.
    \vspace{0.1cm}
    \hrule
    \vspace{0.1cm}
    \State $\zk[0] \gets \yk[0] \gets \xk[0]$; \quad $\Ak[0] \gets 0$; \quad $C \gets \frac{1}{2}\left(\frac{\dualnumber{r}(1-\sigma-\sigma')}{1+\sigma^{\dualnumber{r}}}\right)^{\r-1}$ 
    \State $\tilde{y}_1, v_1, \lambda_1 \gets \iproxoraclehat[r] (x_0, \widehat{\lambda}_0) $; \quad $\widehat{\lambda}_1 \gets \lambda_1$ \label{line:initializing_widehat_lambda}
    \FOR {$k = 1 \textbf{ to } T$}
        \State $\widehat{A}_k = \widehat{a}_k + \Ak[k-1]$; \ \  $\widehat{a}_k = (C\widehat{A}_k^{\r-1} \widehat{\lambda}_k)^{1/\r}$ \Comment{$\r$-degree equation on $\widehat{a}_k > 0$.}
        \State $\xk \gets \frac{A_{k-1}}{\widehat{A}_{k}} y_{k-1} + \frac{\widehat{a}_k}{\widehat{A}_{k}} z_{k-1}$
        \State \textbf{if } $k>1$ \textbf{then} $\newtarget{def:y_k_tilde}{\yktilde[k]}, v_k, \lambda_k \gets \iproxoraclehat[r] (x_k, \widehat{\lambda}_k) $ \Comment{Oracle satisfying \cref{eq:inexact_general_prox_oracle_properties}}
        \State $\gamma_k \gets \min\{\lambda_k / \widehat{\lambda}_k, 1\}$; $a_k \gets \gamma_k \widehat{a}_k$; $A_k \gets a_k + A_{k-1}$
        \State $y_k \gets \argmin \{f(\tilde{y}_{k}), f(y_{k-1}) \}$ \Comment{Or $y_k \gets \frac{(1-\gamma_k)A_{k-1}}{A_k} y_{k-1} + \frac{\gamma_k \widehat{A}_k}{A_k} \yktilde[k]$}\label{line:def_yk_in_adaptive}
        \State $\zk \gets \argmin_{z\in\Rd} \{\sum_{i=1}^{k}\ak[i]\innp{ v_i, z} + \breg(z, \xk[0]) \}$ %
        \State \textbf{if} $\widehat{\lambda}_t \leq \lambda_t$ \textbf{then} $\widehat{\lambda}_{t+1} \gets \alpha \widehat{\lambda}_t$ \textbf{else} $\widehat{\lambda}_{t+1} \gets \alpha^{-1} \widehat{\lambda}_t$
    \ENDFOR
    \State \textbf{return} $\yk[T]$.
\end{algorithmic}
\end{algorithm}

\begin{theorem}\linktoproof{thm:adaptive_alg_guarantee}\label{thm:adaptive_alg_guarantee}
    Let $f:\Rd \to \R$ be a convex function and let $\psi$ be a $(\muUnif, \r)$-uniformly convex regularizer with respect to $\norm{\cdot}$. Given some constants $\sigma, \sigma'$ and initial proximal parameter $\widehat{\lambda}_0 > 0$, every iterate $y_t$ of \cref{alg:adaptive_non_euclidean_accelerated_proximal_point_unif_convex} satisfies, for any $u \in \Rd$:
    \[
        f(y_t)-f(u) = \bigopl{\r}{\frac{\breg[\psi](u, x_0)}{A_t}}.
    \] 
    In particular, it holds for a minimizer $u = \xast$ of $f$, if it exists, in which case we also have:
    \[
        \breg[\psi](\xast, x_0) \geq \sum_{k=1}^{t} \widehat{A}_k \norm{\tilde{y}_r - x_k}^{\r} \frac{1-\sigma-\sigma'}{2\max\{\widehat{\lambda}_k, \lambda_k\}}, \text{\ \  and \ \ } A_t^{1/\r}  \geq \frac{C^{1/\r}}{2r} \sum_{i\in \Lambda} (\alpha^{r_i - 2}\widehat{\lambda}_{i})^{1/\r}, 
    \] 
    for some set of indices $\Lambda$ and some numbers $r_i \geq 0$ satisfying $\sum_{i \in \Lambda} r_i = \frac{t-1}{2}$.
\end{theorem}
The second statement allows for lower bounding $A_k$ in different contexts, in order to characterize the convergence of the method. 
We also note that above we could have used inexact uniformly-convex regularizers instead of exact ones but we used the latter for simplicity. 

\section{High-Order Smooth Convex or Structured Optimization} \label{sec:application_to_high_order_methods}
In this section, we use \cref{alg:non_euclidean_accelerated_proximal_point_unif_convex,alg:adaptive_non_euclidean_accelerated_proximal_point_unif_convex} in order to optimize high-order H\"older smooth convex functions with respect to $\p$-norms by using a $\q$-th order oracle. The main result of this section is the following theorem. We also show convergence for structured functions for which we can implement an inexact $\p$-norm ball optimization oracle.

\begin{theorem}\label{thm:guarantee_high_order_smooth_cvx_methods}\linktoproof{thm:guarantee_high_order_smooth_cvx_methods}
    Let $f:\Rd \to \R$ be a $\q$-times differentiable convex function with a minimizer at $\xast$ whose $\q$-th derivative is $(\L,\nu)$-H\"older continuous with respect to $\norm{\cdot}_{\p}$, $\p\in(1, \infty)$. By making use of \cref{alg:non_euclidean_accelerated_proximal_point_unif_convex} or its generalization \cref{alg:adaptive_non_euclidean_accelerated_proximal_point_unif_convex}, initialized at $x_0$ and defining $R_{\p} \defi \norm{\xast - x_0}_{\p}$, $\m \defi \max\{2,\p\}$, we obtain a point $y_T$ after $T$ iterations, satisfying
    \[
        f(y_T)  - f(\xast) = \bigopl{\q+\nu, \p}{\L R_{\p}^{\q+\nu} T^{-\frac{(\m+1)(\q+\nu)-\m}{\m}}}, 
    \] 
    Each iteration of the algorithm makes $1$ query to a $\q$-th order oracle of $f$.
\end{theorem}
In \cref{sec:lower_bounds}, we show that our bounds are nearly optimal for any algorithm that accesses $f$ via a local oracle, even in randomized and parallel settings. We note that if we use the alternative definition of $y_k$ in Line \ref{line:def_yk_in_adaptive} of \cref{alg:adaptive_non_euclidean_accelerated_proximal_point_unif_convex}, our algorithms do not require the knowledge of the function's $0$-th order information. We also note that from our proof one can derive an analogous statement of the above theorem for any norm, if a uniformly-convex regularizer with respect to this norm is provided.
Below, we show how we can implement an inexact proximal oracle using one call of the $\q$-th order oracle by building the $\q$-th order Taylor expansion $f_q(y;x)$ of $f(y)$ at a point $x$. Note that below, $\nabla \taylorf[\q][f](y_k;x_k) - \hat{v}_k \in \partial F(y_k)$, so the condition below requires an approximate critical point of $F$.
\begin{lemma}[Taylor subproblems]\label{lemma:inexact_criterion_w_taylor_expansion}\linktoproof{lemma:inexact_criterion_w_taylor_expansion}
    Under the conditions of \cref{thm:guarantee_high_order_smooth_cvx_methods}, and consider
    $F(y)\defi \taylorf[\q][f](y; x_k) + \frac{1}{\hat{\lambda}(\q+\nu)}\norm{y-x_k}^{\q+\nu}$, for $\hat{\lambda} \defi \frac{\sigma(\q-1)!}{2L}$ and any $\sigma \in (0, 1)$. The tuple $(y_k, v_k, \lambda_k) \gets (y_k, \nabla f(y_k), \hat{\lambda} \norm{y_k-x_k}^{\r-\q-\nu})$ implements the oracle $\iproxoraclehat[r](x_k, \cdot)$ for  $\oldepsilon_k = 0$, if $y_k$ satisfies %
    \[
        \norm{\nabla \taylorf[\q][f](y_k;x_k) - \hat{v}_k} \leq  \frac{\L}{(\q-1)!} \norm{y_k - x_k}^{\q+\nu-1},
    \] 
    for some $\hat{v}_k \in  \partial_y(-\frac{1}{\hat{\lambda}(\q+\nu)}\norm{y-x_k}^{\q+\nu})(y_k) = \partial_y(-\frac{1}{\lambda_k \r}\norm{y-x_k}^{\r})(y_k)$.
\end{lemma}

\begin{remark}\label{remark:critical_points_exist_and_sometimes_only_one}
    The function $F$ has a global minimizer, and thus at least one critical point, which is what we need to approximate in order to implement the inexact proximal oracle. This is due to $F$ being a polynomial of degree $\q$ plus the $(\q+\nu)$-homogeneous term  $\frac{1}{\hat{\lambda}(\q+1)}\norm{y-x_k}^{\q+\nu}$ and so it is continuous and tends to $+\infty$ in every direction. Finding an approximate critical point does not require any further interaction with any oracle from $f$. Convexity of $F$ is not required in order to find an approximate critical point in a tractable way in several contexts, but it usually enables to solve such a problem faster, cf.  \citep{carmon2021lower}. Note that since $f$ is convex, in the cases $\q =1$ and $\q=2$, its Taylor expansion, and thus $F$, is also convex. We also show in \cref{prop:convexity_of_taylor_subproblems} that for $\p \in (1, 2]$ (and similarly in general for norms whose square is strongly convex with respect to itself) the Taylor subproblems of \cref{lemma:inexact_criterion_w_taylor_expansion} are in fact also convex for all other cases $\q \geq 3$, as long as $\sigma \leq \frac{\p-1}{\q-1}$. Note that in this case the right hand side of this condition is in $(0, 1)$.
    
\end{remark}

\begin{proposition}[Convexity of Taylor subproblems]\label{prop:convexity_of_taylor_subproblems}\linktoproof{prop:convexity_of_taylor_subproblems}
    Let $f$ be a convex function satisfying \cref{eq:holder_cont} for some norm $\norm{\cdot}$ such that $x\mapsto\norm{x}^2$ is $\hat{\mu}$-strongly convex, and let $\q \geq 2$. Then, the function $F(y) \defi \taylorf[\q][f](y;x) + \frac{M}{\q+\nu}\norm{y-x}^{\q+\nu}$ is convex, for $M \geq \frac{2L}{\hat{\mu}(\q-2)!} $.

    In particular for $\norm{\cdot} = \norm{\cdot}_{\p}$, with $\p \in (1, 2]$, it is enough that $M \geq \frac{\L}{(\p-1)(\q-2)!} $ and thus the Taylor subproblems of \cref{lemma:inexact_criterion_w_taylor_expansion} are convex if $\sigma \leq \frac{\p-1}{\q-1}$. 
\end{proposition}

We also show our framework applies to structured functions for which we can inexactly implement a non-Euclidean ball optimization oracle. This is the case for instance for a function $f$ that is quasi-self concordant with respect to a $\p$-norm, cf. \citep{carmon2020acceleration}. In such a case, we have that  the Hessian of $f$ is stable in a $\p$-norm ball of some radius $\rho$ and any center $x$, that is, there exists a constant $c$ such that $c^{-1} \nabla^2 f(y) \preccurlyeq \nabla^2 f(x) \preccurlyeq c \nabla^2 f(y)$, for all $y$ satisfying $\norm{x-y}_{\p} \leq \rho$. Under this assumption $f$ can be approximated fast in such $\p$-norm ball by solving some linear systems with the Hessian at the center of the ball, since by Hessian stability, if we transform the space by $x \to (\nabla^2 f(x))^{-1} x$, we obtain a smooth and strongly-convex function with $\bigo{1}$ condition number. As an example, for the $\ell_\infty$-regression problem, \citep[Section 4.2]{carmon2020acceleration} proved that a smoothed version of the objective, whose optimization is enough for approximating the solution, satisfies quasi-self-concordance with respect to the $\ell_\infty$-norm. Thus, for certain radius $\rho$, one can implement a ball optimization oracle of radius $\rho$ for any $\p \in [1, \infty]$, by using a few linear system solves, while only $\p=2$ was exploited in \citep{carmon2020acceleration}. This results in a trade-off where a $\p$-norm for greater $\p$ may give a smaller initial distance versus a slower convergence rate dependence on the problem parameters.

\begin{proposition}[Inexact Ball Optimization Oracle]\label{prop:ball_opti_oracle_acceleration}\linktoproof{prop:ball_opti_oracle_acceleration}
    If we can implement the oracle in \cref{eq:inexact_general_prox_oracle_properties} while satisfying $\norm{x_k-\yktilde[k]}_{\p} \geq \rho$ for $p \in (1, \infty)$, we achieve an $\epsilon$-minimizer in $\bigotildep{m}{(R_{\p}/\rho)^{\frac{m}{m+1}}}$, where $m = \max\{2, p\}$ and $R_p \defi \norm{\xast - x_0}_{\p}$.
\end{proposition}

\begin{remark}[$\mathbf{p=1}$ and $\mathbf{p=\infty}$]
    The convergence rates in \cref{thm:guarantee_high_order_smooth_cvx_methods,prop:ball_opti_oracle_acceleration} also hold in the case $\p= 1$ up to some $\ln(d)$ factors, by noticing that for $\hat{p} = 1+\ln^{-1}(d)$, we have $\norm{x}_{\p} = \Theta(\norm{x}_{\hat{p}})$, so we can work in the corresponding new $\hat{p}$-norm and the constants depending on $\hat{p}$ in the bound above amount to $\bigo{\ln(d)}$ factors. Moreover, for the case $\p =\infty$, by making use of an unaccelerated method specified in \cref{sec:unaccelerated_inexact_high_order}, we get the natural limit convergence rates $\bigop{\q+\nu}{\L R_{\p}^{\q+\nu} T^{-(\q+\nu-1)}}$, and $\bigotildep{r}{R_{\p}/\rho}$.
\end{remark}

\section{Lower bounds}\label{sec:lower_bounds}

In this section, we derive lower bounds for algorithms that interact with a local oracle to minimize a convex function that is \( \q \)-th order \( (\L, \nu) \)-H\"older continuous with respect to a given arbitrary norm ${\norm{\cdot}}$. We then specialize these results to \( \p \)-norms, obtaining near-optimal guarantees in the high-dimensional regime. Our analysis encompasses both deterministic and randomized algorithms, as well as sequential and parallel methods. The following theorem presents the main result for deterministic sequential algorithms.  In \cref{sec:lower_bounds_random}, we prove \cref{thm:lower-bound_random}, which extends this lower bound to potentially randomized and parallel methods. 

\begin{theorem}[Lower bound for deterministic sequential algorithms] \label{thm:lower-bound}\linktoproof{thm:lower-bound}
    Let $\norm{\cdot}$ a norm in $\Rd$ and $\mathcal{X}$ a closed convex set containing the $R$-ball $B^{\norm{\cdot}}_R$ of $(\Rd,\norm{\cdot})$ for some $R>0$. Let $T\leq d$ a positive integer, $\Theta > 0$ a real number, and $\{z_i\}_{i\in [d]}$ orthogonal vectors in $\Rd$ such that: \\
    $(i)\;$ $\norm{z_i}_\ast\leq 1$ for every $i\in [d]$, \\
    $(ii)\; $ $\min_{x\in \mathcal{X}}\max_{i\in [T]} \langle z_i,x\rangle\leq -\Theta,$ \\
    $(iii) \;$  $d\geq T/\Theta$.
        
    Then, for every $\L>0$, $\nu\in (0,1]$, $\q\geq 1$, and any deterministic algorithm $\mathcal{A}$ interacting with a local oracle $\mathcal{O}$ there exist a function $F:\mathcal{X}\mapsto \R$ that is $\q$-th order $(\L,\nu)$-H\"older continuous with respect to $\norm{\cdot}$ such that
    $$\min_{t\in [T]}F(x_t)-\inf_{x\in \mathcal{X}}F(x) \geq \bigomegatildelp{\q}{\L R^{\q+\nu}\frac{\Theta^{\q+\nu}}{T^{\q+\nu-1}}},$$
    where $\{x_t\}_{t\in [T]}$ is the sequence generated by the pair $(\mathcal{A}, \mathcal{O})$.
\end{theorem}

The lower bound results are particularly relevant for \( \p \)-norms, where the coefficient \( \Theta \) can be explicitly estimated as a function of the number of iterations \( T \). Specifically, they imply that if the dimension is sufficiently large, any algorithm interacting with a local oracle will require at least $\bigomegatildelp{\q+\nu,\p,}{\left(\frac{\L R_{\p}^{\q+\nu}}{\epsilon}\right)^{\theta_{\q,\p}}}$ queries to reach a precision $\epsilon>0$ where, for $\m = \max\{2, \p\}$, we have:
\[
    \theta_{\q,\p} =  \frac{\m}{(\m+1)(\q+\nu)-m} \text{ if } \p \in [1, \infty), \quad\quad \text{ and } \quad\quad  \theta_{\q,\p} =  \frac{1}{\q+ \nu - 1} \text{ if } \p=\infty.
\] 

We observe that for \(\p = 2\), our result recovers the bound \(\bigomegal{\epsilon^{-\frac{2}{3\q+1}}}\) from the Euclidean setting in \citep{arjevani2019oracle}, up to logarithmic factors. Additionally, for \(\q = 1\) and \(\p \geq 2\), we recover the bound \(\bigomegatilde{\epsilon^{-\frac{\p}{\p+1}}}\) established in \citep{guzman2015lower}. Also for \(\p = \infty\) and \( \q = 1 \), our result coincides with \citep{guzman2015lower}. To the best of our knowledge, this is the first work to address the general case of \( p \)-norms for \( \q \geq 2 \). 

Our construction builds upon the approach of \citet{garg2021nearoptimal}, combining randomized smoothing, similar to that proposed in \citet{agarwal2018lower}, with a modified softmax version of the classical hard instance function from \citet{Nemirovski:1983}. Randomized smoothing enables the approximation of a non-smooth function by one with Lipschitz continuous higher-order derivatives. In this work, we derive new bounds on the Lipschitz and smoothness constants of the smoothing operation through a novel application of the divergence theorem. Notably, our proof works seamlessly for all norms, and is the first to establish lower complexity bounds for the smooth \(\ell_{\p}\)-setting with \(1 \leq \p \leq 2\) without relying on high-dimensional embedding reductions \citep{guzman2015lower}, making these proofs arguably more constructive. Moreover, we generalize the results of \citet{garg2021nearoptimal} to accommodate general norms, noting that the properties of the partial softmax operator hold in a broader context than previously established. Interestingly, this approach yields polynomial improvements upon the state of the art lower bounds for first-order smooth parallel convex optimization; namely, for $p=1$, $q=1$, \cite{diakonikolas2020lower} established a $\tilde\Omega(\epsilon^{-2/5})$ lower bound, whereas we are able to obtain a nearly optimal $\tilde\Omega(\epsilon^{-1/2})$, establishing the impossibility of polynomial acceleration by parallelization in this case. 
The rest of this section is dedicated to give some proof details of the results announced. The full proofs are in \cref{app:lower_bound_proofs}.

\subsection{Randomized smoothing}\label{sec:randomized_smoothing}

Let $\nu_S$ be the uniform distribution on a set $S\subseteq \Rd$.
Let $\newtarget{def:ball_norm}{\ballnorm[\beta]}\subset \Rd$ be the ball of center $0$ and radius $\beta$ with respect to norm $\norm{\cdot}$. Given a function $f:\Rd\mapsto \R$, we define its randomized and its sequential randomized smoothing, respectively:
$$
    \newtarget{def:1_fold_smoothing}{\smoothing}_\beta[f](x) \defi \mathbb{E}_{v\sim \nu_{\ballnorm[\beta]}}[f(x+v)],  \quad \text{ and } \quad \newtarget{def:q_fold_smoothing}{\S{\q}{\beta}}[f] \defi (\smoothing_{\beta/2^{\q}} \circ \smoothing_{\beta/2^{\q-1}} \circ \cdots \circ \smoothing_{\beta/2}) (f).
$$
The following lemma briefs the main properties of our smoothing.

\begin{lemma} \label{lemma:randomized_smoothing}\linktoproof{lemma:randomized_smoothing}
    Assume that $f:\Rd \to \R$ is $G$-Lipschitz with respect to a norm $\norm{\cdot}$. Then,
    \begin{enumerate}
        \item $\smoothing_\beta[f]$ is $G$-Lipschitz and $\beta^{-1}dG$-smooth with respect to $\norm{\cdot}$.
        \item $\S{\q}{\beta}[f]$ is $\q$-times differentiable and $\nabla^i\S{\q}{\beta}[f]$ is $L_i$-Lipschitz in an $\norm{\cdot}$-ball of radius $\beta/2^{\q}$ with $L_i \leq \frac{d^i2^{i(i+1)/2}}{\beta^i}G$, for $i \in \{ 0, 1, \dots, \q\}$.
        \item $|\S{\q}{\beta}[f](x)-f(x)|\leq \beta G$.
        \item \label{item:cvxty_randomized_smoothing} If \(f\) is a convex function, then \(\S{\q}{\beta}[f]\) is also a convex function.  
        \item The value \(\S{\q}{\beta}[f](x)\) only depends on the values of \( f \) within the $\norm{\cdot}$-ball of radius \( (1 - 2^{-\q}) \beta \) and center \( x \).
    \end{enumerate}
\end{lemma}

\subsection{Hard instance construction} \label{sec:hard_instance}

Given $\mu>0$ and $n<d$, define the softmax and the partial softmax functions as
$$\newtarget{def:softmax}{\smax{\mu}}(x)\defi \mu\ln\left(\sum_{j=1}^d\exp(x_i/\mu)\right), \qquad 
\newtarget{def:partial_softmax}{\smaxn{n}{\mu}}(x) \defi \mu\ln\left(\sum_{j=1}^n\exp(x_i/\mu)\right).$$
The following lemma generalizes the results in \citep{garg2021nearoptimal} to arbitrary norms. 
\begin{lemma} \label{lemma:softmax}\linktoproof{lemma:softmax}
    Let $A:\Rd\mapsto \Rd$ a linear map $A(x) = (\langle a^{1},x\rangle,\dots,\langle a^{d},x\rangle)$ such that $\norm{a^{\ell}}_\ast \leq 1$ for every $\ell\in [d]$. The following properties hold. 
\begin{itemize}
    \item[$(a)$] The composition $\smax{\mu}(Ax)$ is 1-Lipschitz with respect to $\norm{\cdot}$.
    \item[$(b)$] $\nabla^{\q}\smax{\mu}(Ax)$ is $L_{\q}$-Lipschitz with respect to $\norm{\cdot}$ with $L_{\q} := \left(\frac{\q+1}{\ln(\q+2)}\right)^{\q+1}\frac{\q!}{\mu^{\q}}$.
    \item[$(c)$] Let $x\in \Rd$ and $n<d$. If $\frac{1}{\mu}[\smax{\mu}(Ax)-\smaxn{n}{\mu} (Ax)] = \delta <1$ then
        $$\norm{\nabla \smax{\mu}(Ax)-\nabla \smaxn{n}{\mu}(Ax)}_\ast\leq 4\delta. $$
\end{itemize}
\end{lemma}

Our hard instance construction is as follows. Let $\gamma,\mu, \beta$ and $\alpha$ positive parameters. For $i\in [T]$ define the functions $f_i,h,g:\Rd\mapsto \R$ by
\begin{align*}
    f_i(x) &\defi \smaxn{i}{\mu}((\langle z_j, x\rangle + (T-j)\gamma)_{j\in[d]}) + \mu(T+1-i)d^{-\alpha}, \\
    h(x) &\defi \max_{i\in[T]}f_i(x), \qquad g(x)\defi\S{\q}{\beta}[h](x).
    \label{def_fgh}
\end{align*}
The functions \( f_i \) are translated partial softmax functions, which are \(1\)-Lipschitz by \cref{lemma:softmax}. The function \( h \) is the maximum of \(1\)-Lipschitz functions and is thus also \(1\)-Lipschitz. Finally, the function \( g \) is the sequentially randomized version of \( h \), making it smooth with Lipschitz derivatives, as established in \cref{lemma:randomized_smoothing}. The following lemma formalizes the high-order smoothness of \( g \).
\begin{lemma} \label{lemma:Lq}\linktoproof{lemma:Lq}
    For any choice of vectors $\{z_j\}_{j\in[T]}$ with $\|z_j\|_{\ast}\leq 1$, the function $g$ is convex, $\q$-times differentiable and $\nabla^{\q} g(x)$ is $L_{\q}$-Lipschitz with $L_{\q} = \bigop{q}{\left(\frac{T\ln d}{\Theta}\right)^{\q}}$.
\end{lemma}

\subsection{Overview of the proof}
Given the constructions above, the rest of the proof of \cref{thm:lower-bound} follows from a standard argument \citep{Nemirovski:1983,guzman2015lower}. Assuming that we work with a set \( \mathcal{X} \) containing the unit ball of \( (\mathbb{R}^d, \|\cdot\|) \),  we first establish an upper bound on the minimum value of \( g \) over \( \mathcal{X} \) by leveraging a uniform bound on the functions \( f_i \), consequence of condition \((ii)\) of the theorem. Then, for any sequence of points \( \{x_t\}_{t \in [T]} \), we construct an instance of the function $g$ such that \( g(x_t) \) remains uniformly lower bounded for all \( t \in [T] \). This construction uses vectors of the form \( z_t = \xi_t v_t \), where \( \{v_t\}_{t \in [T]} \) is a sequence of orthogonal vectors, and \( \{\xi_t\}_{t \in [T]} \) is a sequence of signs chosen adaptively based on the points \( \{x_t\} \). By setting the parameters $\gamma = \frac{\Theta}{4T}$, $\mu = \frac{\gamma}{4\alpha\ln d}$, $\beta = \frac{\gamma}{\ln d}$ and $\alpha\ge q+1$, we establish a gap between the upper and lower bounds of order \( \bigomegatildelp{\q}{\Theta} \).  

Next, we rescale \( g \) by a factor \( \L / L_{\q} \) to ensure it is a \( q \)-th order \( \L \)-Lipschitz function while preserving an optimality gap of order  $\bigomegatildelp{\q}{\L\frac{\Theta}{L_{\q}}} =  \bigomegatildelp{\q}{\L\frac{\Theta^{\q+1}}{T^{\q}}},$ where we apply \cref{lemma:Lq} to estimate \( L_{\q} \). Finally, we extend the result to general \( R \)-balls and \( (\L, \nu) \)-H\"older continuous functions via standard rescaling and interpolation techniques.

\acks{
D. Martínez-Rubio was partially funded by the project IDEA-CM (TEC-2024/COM-89).  C. Guzmán’s research was partially supported by INRIA Associate Teams project, ANID FONDECYT 1210362 grant, ANID Anillo ACT210005 grant, and National Center for Artificial Intelligence CENIA FB210017, Basal ANID. J.P. Contreras was supported by Postdoctoral Fondecyt Grant No. 3240505.
}

\printbibliography[heading=bibintoc] %

\clearpage
\appendix

\section{Convergence of the adaptive algorithm}\label{app:adaptive_analysis}

This section proves convergence of the generalized version of our \cref{alg:non_euclidean_accelerated_proximal_point_unif_convex} that is, \cref{alg:adaptive_non_euclidean_accelerated_proximal_point_unif_convex}. Recall that for any $\r \in [1, \infty]$, we define $\dualnumber{r}$ as in \cref{def:youngs_conjugate_number}. We repeat the pseudocode for convenience.

\begin{algorithm}[h!]
        \caption{Non-Euclidean Adaptive Accelerated Proximal Point with Uniformly Convex Regularizer}
\begin{algorithmic}[1] 
    \REQUIRE Convex function $f:\Rd\to\R$. Regularizer $\psi$ that is $(1,\r)$-uniformly convex function wrt a norm $\norm{\cdot}$. Initial $\widehat{\lambda}_0$. Adjustment constant factor $\alpha > 1$. Inexactness constants $\sigma, \sigma'$.
    \vspace{0.1cm}
    \hrule
    \vspace{0.1cm}
    \State $\zk[0] \gets \yk[0] \gets \xk[0]$; \quad $\Ak[0] \gets 0$; \quad $C \gets \frac{1}{2}\left(\frac{\dualnumber{r}(1-\sigma-\sigma')}{1+\sigma^{\dualnumber{r}}}\right)^{\r-1}$ 
    \State $\tilde{y}_1, v_1, \lambda_1 \gets \iproxoraclehat[r] (x_0, \widehat{\lambda}_0) $; \quad $\widehat{\lambda}_1 \gets \lambda_1$ 
    \FOR {$k = 1 \textbf{ to } T$}
        \State $\widehat{A}_k = \widehat{a}_k + \Ak[k-1]$; \ \  $\widehat{a}_k = (C\widehat{A}_k^{\r-1} \widehat{\lambda}_k)^{1/\r}$ \Comment{$\r$-degree equation on $\widehat{a}_k > 0$.}
        \State $\xk \gets \frac{A_{k-1}}{\widehat{A}_{k}} y_{k-1} + \frac{\widehat{a}_k}{\widehat{A}_{k}} z_{k-1}$
        \State \textbf{if } $k>1$ \textbf{then} $\yktilde[k], v_k, \lambda_k \gets \iproxoraclehat[r] (x_k, \widehat{\lambda}_k) $ \Comment{Oracle satisfying \cref{eq:inexact_general_prox_oracle_properties}}
        \State $\gamma_k \gets \min\{\lambda_k / \widehat{\lambda}_k, 1\}$; $a_k \gets \gamma_k \widehat{a}_k$; $A_k \gets a_k + A_{k-1}$
        \State $y_k \gets \argmin \{f(\tilde{y}_{k}), f(y_{k-1}) \}$ \Comment{Or $y_k \gets \frac{(1-\gamma_k)A_{k-1}}{A_k} y_{k-1} + \frac{\gamma_k \widehat{A}_k}{A_k} \yktilde[k]$}
        \State $\zk \gets \argmin_{z\in\Rd} \{\sum_{i=1}^{k}\ak[i]\innp{ v_i, z} + \breg(z, \xk[0]) \}$ %
        \State \textbf{if} $\widehat{\lambda}_t \leq \lambda_t$ \textbf{then} $\widehat{\lambda}_{t+1} \gets \alpha \widehat{\lambda}_t$ \textbf{else} $\widehat{\lambda}_{t+1} \gets \alpha^{-1} \widehat{\lambda}_t$
    \ENDFOR
    \State \textbf{return} $\yk[T]$.
\end{algorithmic}
\end{algorithm}

\begin{proof}\linkofproof{thm:adaptive_alg_guarantee}
    We use an adaptive scheme inspired by \citep{carmon2022optimal} used to guess the proximal parameter $\lambda_k$. In some cases of high-order smooth convex optimization, we can implement the inexact proximal oracle of \cref{alg:non_euclidean_accelerated_proximal_point_unif_convex}, but with a parameter of $\lambda_k$ that depends on $y_k$. Because $y_k$ also depends on $\lambda_k$, this double dependence leads to problems that can be solved by using a binary search at each iteration. However, the adaptive scheme removes the need for the binary search.

    We use the same lower bound as in \cref{eq:lower_bound_ms} but this time for simplicity we only use $\r$-uniformly convex regularizers, $\r \geq 2$, instead of inexact ones. As opposed to \cref{alg:non_euclidean_accelerated_proximal_point_unif_convex}, this time we denote by $\yktilde[k]$ the points that the inexact proximal oracle returns. Therefore, $v_k \in \partial^{\oldepsilon_k} f(\yktilde[k])$. We define a different convex combination for the point where we compute the gradient $x_k = \frac{A_{k-1}}{\widehat{A}_{k}} y_{k-1} + \frac{\widehat{a}_k}{\widehat{A}_{k}} z_{k-1}$ for some $\widehat{a}_{k}$ to be determined later, that satisfies $a_{k} = \gamma_k \widehat{a}_k$, where $\gamma_k \defi \min\{\lambda_{k}/\widehat{\lambda}_{k}, 1 \}$, and where $\widehat{\lambda}_{k}$ is a guess on the proximal parameter of our next oracle and $\lambda_k$ is the proximal parameter that the oracle actually returns. We also have $\widehat{A}_k \defi A_{k-1} + \widehat{a}_k$. 

    We define the upper bound $U_k \defi f(y_{k})$ and the primal-dual gap $G_k \defi U_k - L_k$ but this time we want $U_k \leq \frac{(1-\gamma_{k})A_{k-1}}{A_{k}} f(y_{k-1}) + \frac{\gamma_{k} \widehat{A}_{k}}{A_k} f(\yktilde[k])$.  Therefore, we can define the combination $y_k \defi \frac{(1-\gamma_{k})A_{k-1}}{A_{k}} y_{k-1} + \frac{\gamma_{k} \widehat{A}_{k}}{A_k} \yktilde[k]$, which note it is a convex combination, or we can simply define it as $y_k \in \argmin\{ f(y_{k-1}), f(\yktilde[k])\}$. With these definitions, we have
        \begin{align*}
         \begin{aligned}
             A_{k} &G_{k} - A_{k-1}G_{k-1} - \eventindicator{k=1} \breg(u, \xk[0])
             \circled{1}[\leq](\Ccancel[red]{1}-\gamma_{k})A_{k-1}f(\yk[k-1]) + \gamma_{k}\widehat{A}_{k} f(\yktilde[k]) -\Ccancel[red]{A_{k-1}f(\yk[k-1])} \\
             & \quad - \ak[k] f(\yktilde[k]) \Ccancel[blue]{- \sum_{i=1}^{k-1} \ak[i] f(\yktilde[i])} - \left(\sum_{i=1}^{k-1} \ak[i](\innp{ v_i, \zk[k] - \yktilde[i]}-\oldepsilon_i)  + \breg(\zk[k], \xk[0]) \right) - \ak[k]\innp{ v_{k}, \zk[k] - \yktilde[k]} + \ak[k]\oldepsilon_k \\
             & \quad + \Ccancel[blue]{\sum_{i=1}^{k-1} \ak[i] f(\yktilde[i])} + \left(\sum_{i=1}^{k-1} \ak[i](\innp{v_i, \zk[k-1] - \yktilde[i]}-\oldepsilon_i) + \breg(\zk[k-1], \xk[0]) \right)  \\
             &\circled{2}[\leq] \innp{v_{k}, \gamma_k A_{k-1}(\yktilde[k]-\yk[k-1]) - \ak[k] (\pm \zk[k-1] + \zk[k]-\yktilde[k])} -\frac{1}{\r}\norm{\zk[k-1] - \zk[k]}^{\r} + \gamma_k \widehat{A}_{k}\oldepsilon_k  \\
             & \circled{3}[=] \innp{v_{k}, \gamma_k \widehat{A}_{k} (\yktilde[k] - x_{k}) + \ak[k](\zk[k-1] - \zk[k]) }  - \frac{1}{\r} \norm{\zk[k-1] - \zk[k]}^{\r} + \gamma_k \widehat{A}_{k}\oldepsilon_k \\
             & \circled{4}[\leq] \gamma_k \widehat{A}_{k}\innp{v_{k}, \yktilde[k] - \xk[k]} + \frac{a_{k}^{\dualnumber{r}}}{2}\norm{v_{k}}_\ast^{\dualnumber{r}} + \gamma_k \widehat{A}_{k}\oldepsilon_k \\
             & \circled{5}[\leq] \gamma_k \widehat{A}_{k}\innp{\widehat{v}_{k}, \yktilde[k] - \xk[k]} + \gamma_k \widehat{A}_{k}\norm{v_{k}-\widehat{v}_{k}}_\ast \cdot \norm{\yktilde[k] - \xk[k]} + \frac{2^{1/(\r-1)}}{\dualnumber{r}}a_{k}^{\dualnumber{r}}(\norm{\widehat{v}_{k}}_\ast^{\dualnumber{r}} + \norm{v_{k}-\widehat{v}_{k}}_\ast^{\dualnumber{r}}) + \gamma_k \widehat{A}_{k}\oldepsilon_k \\
             & \circled{6}[\leq] \left( \frac{\gamma_k \widehat{A}_k(-1+\sigma)}{\lambda_k} + \frac{a_{k}^{\r/(\r-1)}}{\lambda_{k}^{\r/(\r-1)}}\frac{2^{1/(\r-1)}}{\dualnumber{r}}(1+\sigma^{\dualnumber{r}}) \right) \norm{\yktilde[k]-\xk[k]}^{\r} + \gamma_k \widehat{A}_{k}\oldepsilon_k\\
             & \circled{7}[\leq] \left( -\widehat{A}_k(1-\sigma-\sigma') + \frac{\widehat{a}_{k}^{\r/(\r-1)}}{\widehat{\lambda}_{k}^{1/(\r-1)}}\frac{2^{1/(\r-1)}}{\dualnumber{r}}(1+\sigma^{\dualnumber{r}})\right) \frac{\gamma_k}{\lambda_k}\norm{\yktilde[k]-\xk[k]}^{\r}\\
             & \circled{8}[\leq] -\frac{\widehat{A}_k(1-\sigma-\sigma')}{2}\min\left\{\widehat{\lambda}_k^{-1}, \lambda_k^{-1}\right\} \norm{\yktilde[k]-\xk[k]}^{\r} \defi E_k.
         \end{aligned}
        \end{align*}

            Above, we wrote the definition of the gaps in $\circled{1}$, we canceled some terms and we used the indicator on the left hand side to handle the cases $k=1$ and $k> 1$ at the same time. We also used the bound $A_{k}U_{k} \leq (1-\gamma_k)A_{k-1}f(y_{k-1}) + \gamma_k \hat{A}_k\yktilde[k]$.
            In $\circled{2}$, we applied the enlarged subgradient property on the remaining terms with $f(\cdot)$, namely $\gamma_k A_{k-1}(f(y_k)-f(\yktilde))$ and used $a_k = \gamma_k \widehat{a}_k$,  $\widehat{A}_k = A_{k-1}+\widehat{a}_k$, yielding error $\gamma_kA_{k-1}\oldepsilon_k$ which gives $\gamma_k\widehat{A}_k\oldepsilon_k$ after merging it with the other error. We grouped the resulting expression with another term, and we used that the terms in parentheses are $\elll{k-1}(z_{k-1}) - \elll{k-1}(z_{k})$. The $(1, \r)$-uniform convexity of $\elll{k-1}(\cdot)$ and the fact that $z_{k-1}$ is its minimizer implies the bound. In $\circled{3}$, we used that by definition of $\xk[k]$ it is $\widehat{A}_{k} x_{k} = A_{k-1} \yk[k-1] + \widehat{a}_{k}\zk[k-1]$, along with $a_k = \gamma_k \widehat{a}_k$. We had added and subtracted $\zk[k-1]$ to apply H\"older's and Young's inequalities in $\circled{4}$, namely $\innp{v, u} \leq \norm{v}_\ast \norm{u}  \leq \frac{1}{\dualnumber{r}}\norm{v}_\ast^{\dualnumber{r}} + \frac{1}{\r}\norm{u}^{\r}$. In $\circled{5}$, we added and subtracted some $\widehat{v}_{k}$ terms and use simple bounds to make $\norm{v_{k}-\widehat{v}_{k}}_\ast$ appear. We do this because the first and third resulting terms are proportional to $\norm{y_k-x_k}^{\r}$ and with our criterion we can make the rest to be as well. So indeed, in $\circled{6}$ we applied the properties of the oracle \cref{eq:inexact_prox_oracle_properties} for the second and fourth terms and used \cref{eq:property_subgrads_of_norm_to_the_r} which also holds in this algorithm, and this application yields equality for the first and third terms. We obtain $\circled{7}$ by substituting $a_k = \gamma_k\widehat{a}_k$, and using that by definition of $\gamma_k \defi \min\{ \lambda_k / \widehat{\lambda}_k, 1\}$, it is $\gamma_k /\lambda_k  = \min\left\{\widehat{\lambda}_k^{-1}, \lambda_k^{-1}\right\} \leq \widehat{\lambda}_k^{-1}$. We also used the assumption $\oldepsilon_k \leq \frac{\sigma'}{\lambda_k} \norm{\yktilde[k]-x_k}^{\r}$.

         Finally, for $\circled{8}$, we find $\widehat{a}_k > 0$ so the second summand is half of the absolute value of the first summand. This only changes the value of $\widehat{a}_k$ slightly with respect to making the bound $0$, and at the same time, it provides a good negative term that can be used to guarantee fast growth of $A_k$ when $\norm{y_k - x_k}$ is large enough. Let $C \defi \frac{1}{2}\left(\frac{\dualnumber{r}(1-\sigma-\sigma')}{2(1+\sigma^{\dualnumber{r}})}\right)^{\r-1}$. It is enough to solve the equation $\widehat{a}_{k}^{\r} = C \widehat{A}_{k}^{\r-1} \widehat{\lambda}_{k}$. And this does not require to know the value of $\lambda_k$, which is only revealed after we choose $x_k$ and receive the answer from the oracle $\iproxoraclehat[r]$. The first part of the second statement now holds by \cref{eq:convergence_rate_inexact} and the definition of our $E_k$.

    From now on, we assume a minimizer $\xast$ exists and set $u = \xast$. Borrowing from \citep{carmon2022optimal} (note that our convention for the proximal parameter $\lambda$ being in the denominator of the Moreau's envelope definition reverses the order), we define $S_T^{\geq} \defi \setsuch{k\in[T]}{\lambda_k \geq \widehat{\lambda}_k} = \setsuch{k\in[T]}{\gamma_k=1}$ the set of \emph{up} iterates, and recall that after any iterate $k$ of them we have that $\widehat{\lambda}_k$ is increased to $\widehat{\lambda}_{k+1} \defi \alpha \widehat{\lambda}_{k}$. Similarly, the set of down iterates is defined as $S_T^{<} \defi \setsuch{k\in[T]}{\lambda_k < \widehat{\lambda}_k}$ and after any of these iterates $k$, we have $\widehat{\lambda}_{k+1} = \alpha^{-1}\widehat{\lambda}_k$. The first iterate is an up iterate by construction, see Line \ref{line:initializing_widehat_lambda} of \cref{alg:adaptive_non_euclidean_accelerated_proximal_point_unif_convex}.

The sequence of iterates up to $T$ can be split into subsequences of maximal length with only up or only down iterates. We denote the last iterate of the $i$-th subsequence of down iterates as $d_{i+1}$. And for convenience, even if the first and last iterates are not down iterates we denote them by $d_1 = 1$ and $d_S = T$, where $S-1$ is the number of up subsequences. We denote the last iterate of the $i$-th of these $S-1$ up subsequences as $u_{i}$. As an example:
\[
    \underbrace{U}_{d_1}\ U\ \underbrace{U}_{u_1} D\ D\underbrace{D}_{d_2}U\ U\underbrace{U}_{u_2}D\ D\ D\underbrace{D}_{d_3}\underbrace{U}_{u_3}\underbrace{D}_{d_4}U\ U\underbrace{U}_{u_4}D\ \underbrace{D}_{d_5}
\] 
Because of how we update $\widehat{\lambda}_k$, and the indices definitions, we have for $i\in[S-1]$ that $\widehat{\lambda}_{u_i} \geq \alpha^{d_{i+1}-u_i-2}\widehat{\lambda}_{d_{i+1}}$, where the inequality is an equality for $i = S$ in case that the last iterate is an up iterate in which case $u_{S-1} =d_{S}$ and  $\widehat{\lambda}_{u_{S-1}}  = \widehat{\lambda}_{d_{S}} \geq \alpha^{d_S - u_{S-1}-2}\widehat{\lambda}_{d_{S}}$. We also have for all $i\in[S]$ that $\widehat{\lambda}_{u_i} \geq \alpha^{u_i - d_i-2}\widehat{\lambda}_{d_i}$, where the inequality is also an equality except for $i =1$ in case that the first subsequence of up iterates is of length one in which case $d_1 = u_1$ and so $\widehat{\lambda}_{u_1} \geq \alpha^{u_1-d_1 -2}\widehat{\lambda}_{d_1}$.

Now, given the relation $\widehat{a}_{k}^{\r} = \frac{1}{2}C \widehat{A}_{k}^{\r-1} \widehat{\lambda}_{k}$, and $a_k = \widehat{a}_k$ and $A_k = \widehat{A}_k$ for all , and $A_k \geq A_{k-1}$, we have $\circled{1}$ below by \eqref{eq:recursion_lower_bounding_Ak_unif_cvx}:
\begin{align}\label{eq:lower_bounding_Ak_in_adaptive}
     \begin{aligned}
         A_T^{1/\r} & \circled{1}[\geq] A_{T-1}^{1/\r} + \eventindicator{T \in S_T^{\geq}} \frac{C^{1/\r}}{\r} \widehat{\lambda}_T^{1/\r} \circled{2}[\geq]  \sum_{i \in S_T^{\geq}} \frac{C^{1/\r}}{\r} \widehat{\lambda}_i^{1/\r}. \\
         &\geq \frac{C^{1/\r}}{2r}\left( \sum_{i\in [S-1]} \widehat{\lambda}_{u_i}^{1/\r} + \sum_{i\in [S-1]} \widehat{\lambda}_{u_i}^{1/\r}  \right)\\
         &\circled{3}[\geq] \frac{C^{1/\r}}{2r} \left( \sum_{i = 2}^{S} (\alpha^{d_{i} - u_{i-1} - 2}\widehat{\lambda}_{d_{i}})^{1/\r}  + \sum_{i = 1}^{S-1} (\alpha^{u_i-d_i-2}\widehat{\lambda}_{d_{i}})^{1/\r} \right)\\
         & \circled{4}[\geq] \frac{C^{1/\r}}{2r} \left( (\alpha^{\frac{1}{2}(u_{1} - d_1) - 2}\widehat{\lambda}_{d_{1}})^{1/\r} + \sum_{i = 2}^{S-1} (\alpha^{\frac{1}{2}(u_{i} - u_{i-1}) - 2}\widehat{\lambda}_{d_{i}})^{1/\r}  + (\alpha^{\frac{1}{2}(d_S - u_{S-1})-2}\widehat{\lambda}_{d_{S}})^{1/\r} \right)\\
         &\circled{5}[\geq] \frac{C^{1/\r}}{2r} \sum_{i\in Q} (\alpha^{r_i - 2}\widehat{\lambda}_{d_{i}})^{1/\r},
     \end{aligned}
\end{align}
where $\circled{2}$ applied the same as $\circled{1}$ recursively. In $\circled{3}$ we applied the bounds on $\widehat{\lambda}_{u_i}$ that we computed above. In $\circled{4}$, for the indices $i=2, 3, \dots, S-1$, we used the $\r$- and geometric mean inequality: $\frac{1}{2}(\alpha^{a/\r} + \alpha^{b/\r}) \geq \alpha^{(a+b)/(2r)} \geq \frac{1}{2}\alpha^{(a+b)/(2r)}$ for any $\r > 0$, where losing a factor of $2$ is done just for convenience. In the first and third summands, we just reduce the value of the exponent in order to have a unified structure in $\circled{5}$, where we just used the numbers $r_i \geq 0$ defined as $r_1 \defi \frac{1}{2}(n_1 - d_1) = \frac{1}{2}(n_1 - 1)$, $r_S = \frac{1}{2}(d_S - n_{S-1}) \defi \frac{1}{2}(T - n_{S-1})$, and for $i = 2, 3, \dots, S-1$, it is $r_i \defi \frac{1}{2}(n_i - n_{i-1})$. And note $\sum_{i=1}^S r_i = \frac{T-1}{2}$.

Finally, we note that $T$ was arbitrary, and also that the numbers defined by the subsequence are compatible with a longer subsequence, except for the last one. The theorem statement holds, after some indices relabeling and using a set $\Lambda$.
\end{proof}

\section{Proofs from \nameref{sec:preliminaries_and_groundwork}}\label{app:proofs_of_preliminaries}
 
\begin{proof}\linkofproof{prop:properties_of_M}

    Proof of \cref{properties_of_M:1}. For the norm $\norm{\cdot}_{\p}$ and $\p=1$, we can consider the function $f(x) = \frac{1}{2}\norm{x}_1^2$, then for instance for $x = (1, 1, \dots, 1)$ and $\lambda = \frac{1}{2}$, there is not a unique minimizer. Similarly, if $\p = \infty$ and $f(x) = \frac{1}{2}\norm{x}_\infty^2$ then for instance for $x = e_1$ and $\lambda = \frac{1}{2}$ there is not a unique minimizer. For $\p \in (1, \infty)$ the minimizer $\prox(x)$ is unique since $\frac{1}{2\lambda}\norm{x-\cdot}^2$ is strictly convex.

    Proof of \cref{properties_of_M:2}. The function to be optimized in the definition of $\M(x)$ is jointly convex on $x, y$. Consequently, the epigraph of $(x, y)\mapsto f(y) + \frac{1}{2\lambda}\norm{x-y}^2 $ is convex and so the epigraph of $\M(x)$ is the projection of a convex set and therefore convex. The joint convexity is derived from the joint convexity of $(x, y) \mapsto \norm{x-y}^2$ which holds since for points $x, x', y, y' \in \Rd$, we have
    \[
        \norm{(x+x')/2 - (y+y')/2}^2 \circled{1}[\leq] \left(\frac{1}{2}\norm{x-y} + \frac{1}{2}\norm{x'-y'}\right)^2 \leq \frac{1}{2}\norm{x-y}^2 + \frac{1}{2}\norm{x'-y'}^2,
    \] 
    where we used the triangular inequality in $\circled{1}$ and $(a+b)^2 \leq 2a^2 + 2b^2$ in $\circled{2}$.

   Proof of \cref{properties_of_M:3}. By definition of $M$, we have
   \[
       f(\prox(x)) \leq f(\prox(x)) + \frac{1}{2\lambda} \norm{x-\prox(x)}^2 = \M(x) \leq f(x) + \frac{1}{2\lambda} \norm{x-x}^2 = f(x).
   \] 
    In particular, since $f(\xast) = \min_{x\in \Rd} f(x)$, it must be $f(\prox(\xast)) = \M(\xast) = f(\xast)$.

    Proof of \cref{properties_of_M:4}. By the generalized Danskin's theorem \citep{bertsekas2003convex}, we have $\partial M (x) = \operatorname{conv}\{\subdiffsqnorm{x}(\prox(x)) \ | \ \prox(x) \in \Prox(x) \}$. Moreover, by the first order optimality condition of any $\prox(x) \in \Prox(x)$ in the optimization problem defining $\M(x)$, we have $0 \in \partial f(\prox(x)) + \partial_y \left.\frac{\norm{x-y}^2}{2\lambda}\right|_{y=\prox(x)}$ and so there is $g \in \subdiffsqnorm{x}(\prox(x))$ such that $g \in \partial f(\prox(x))$. Note that our proof relies on the symmetry of the function that we use to convolve with $f$, or more in particular, on $\subdiffsqnorm{x}(y) = - \subdiffsqnorm{y}(x)$ for all $x, y$. (compare to Bregman proximal point, in which one uses the Moreau envelope $\M(x) \defi \min_{y\in\Rd}\{ f(y) + \breg[\psi](x, y)\}$ where $\breg[\psi]$ is not symmetric in general).

    Proof of \cref{properties_of_M:5}.
    Let $f(x) = \frac{1}{2}\norm{x}^2$ and let $g \in \partial f(x)$, for some $x \in \Rd$. We have 
    \begin{align*}
     \begin{aligned}
         \frac{1}{2}\norm{x}^2 &= f(x) \circled{1}[=] \innp{g, x} - f^\ast(g) \circled{2}[\leq]  \norm{g}_\ast\cdot\norm{x} - f^\ast(g)  \circled{3}[\leq] \frac{1}{2}\norm{g}_\ast^2 + \frac{1}{2}\norm{x}^2  - f^\ast(g) \circled{4}[=] \frac{1}{2}\norm{x}^2.
     \end{aligned}
    \end{align*}
    where $\circled{1}$ uses Fenchel duality, %
    $\circled{2}$ uses Cauchy-Schwarz, $\circled{3}$ is due to Young's inequality and $\circled{4}$ uses the duality between norms. Because we arrived to an equality, then $\circled{2}$ and $\circled{3}$ must be equalities, which only holds if $\innp{g, x} = \norm{x}^2 = \norm{g}_\ast^2$. By shifting, scaling, and \cref{properties_of_M:4}, defining any $g \in \subdiffsqnorm{x}(\prox(x))$ and $g_M \in \partial \M(x)$, we have $\lambda\innp{g, \prox(x) - x} = \norm{x-\prox(x)}^2 = \norm{\lambda g}_\ast^2 \stackrel{?}{=} \lambda^2\norm{g_M}_\ast^2 \stackrel{?}{=} \lambda\innp{g_M, \prox(x) - x}$, as desired.

    Proof of \cref{properties_of_M:6}. We have
    \[
        \M[\lambda_1](x) - \frac{1}{2\lambda_1} \norm{x-\prox[\lambda_1](x)}^2 \circled{1}[=] f(\prox[\lambda_1](x)) \circled{2}[\geq] \M[\lambda_2](\prox[\lambda_1](x)),
    \] 
    where $\circled{1}$ holds by definition of $\M(x)$ and $\prox(x)$, and $\circled{2}$ uses \cref{properties_of_M:3}.

\end{proof}

\begin{proof}\linkofproof{lemma:inexact_unif_convex_from_unif_convex}
    By using Young's inequality with conjugate exponents $\sigma / s > 1 $ and $\sigma / (\sigma -s) > 1$:
    \[
        \norm{x-y}^s \leq \frac{1}{a^{\frac{\sigma}{s}} \sigma / s}\norm{x-y}^\sigma + a ^{\frac{\sigma}{\sigma-s}}  \frac{\sigma-s}{\sigma},
    \]
    or equivalently we have $\circled{1}$ below
    \[
        \frac{a^{\frac{\sigma}{s}}}{s} \norm{x-y}^s - a^{(\frac{\sigma}{\sigma-s} + \frac{\sigma}{s})}\frac{\sigma-s}{s\sigma} \circled{1}[\leq] \frac{1}{\sigma}\norm{x-y}^\sigma \circled{2}[\leq] \breg[\psi](x, y),
    \] 
    where $\circled{2}$ holds by $(1,\sigma)$-uniform convexity of $\psi$. Simplifying the left hand side yields the statement.
    
\end{proof}

\begin{proof}\linkofproof{lemma:regularity_of_regularizers}
In the first case $\psi(x) = \frac{1}{\p}\norm{x-x_0}_{\p}^{\p}$ and $\p \geq 2$, we note that a proof of $(2^{-\frac{\p(\p-2)}{\p-1}}, \m)$-uniform convexity is provided in \citep[Proposition 3.2]{zalinescu1983}. We show a proof of uniform convexity with a slightly better constant.  Note that $\norm{x}_{\p}^{\p}$ is a separable function. Thus, it is enough to show the uniform convexity of the one-dimensional case and add up all of the corresponding inequalities in order to obtain the result. In \citep[Lemma 4.2.3]{nesterov2018lectures}, it is established that $\frac{1}{\p}\norm{x}_2^{\p}$ is $(2^{2-\p}, \p)$-uniformly convex with respect to the Euclidean norm $\norm{\cdot}_2$. Since in one dimension, all of the $\p$-norms are the same, the result is proven.

    The second statement was shown in \citep{ball2002sharp,shalev2007online}. We reproduce the argument of the latter for completeness. We now have $\psi(x) \defi \frac{1}{2(\p-1)}\norm{x-x_0}_{\p}^2$ and $\p \in (1, 2]$, and we write $\psi(x) \defi \Psi(\sum_{i=1}^d \phi(x_i))$ for $\Psi(a) \defi \frac{a^{2/\p}}{2(\p-1)}$ and $\phi(a)=\abs{a}^{\p}$ with derivatives:
    \[
        \Psi'(a)=\frac{1}{\p(\p-1)}a^{\frac{2}{\p}-1}; \quad  \Psi''(a) =\frac{1}{\p(\p-1)}\pa{\frac{2}{\p}-1}a^{\frac{2}{\p}-2} \geq 0,
    \]

    \[
    \phi'(a) = \p \text{ sign} (a) \abs{a}^{\p-1}; \quad \phi''(a) = \p(\p-1)\abs{a}^{\p-2}.
    \] 
    We used  $\abs{a}^{\q}$ is differentiable everywhere for $\q>1$. Thus, 
    \[ 
    \nabla^2_{i,j}f(x) = \Psi''\pa{\sum_{k=1}^d\phi(x_k)}\phi'(x_i)\phi'(x_j) + \eventindicator{i=j} \Psi'\pa{\sum_{k=1}^d\phi(x_k)}\phi''(x_i),
    \]

    Let us denote $y_i=\abs{x_i}^{(2-\p)\frac{\p}{2}}$. We have
    \begin{align*}
    \begin{aligned}
        \nabla^2& f(x)[v, v]=\Psi''\pa{\sum_{\r=1}^n\phi(x_r)}\pa{\sum_i\phi'(x_i)v_i}^2+\Psi'\pa{\sum_{i=1}^d\phi(x_i)}\sum_i\phi''(x_i)v_i^2 \\
        &\circled{1}[\geq] \frac{\norm{x}_{\p}^{\p\pa{\frac{2}{\p}-1}}}{\p(\p-1)}\sum_i \p(\p-1)\abs{x_i}^{\p-2}v_i^2 = \pa{\sum_{i=1}^d \abs{x_i}^{\p}}^{\frac{2-\p}{\p}}\sum_i \abs{x_i}^{\p-2}v_i^2 \\
        &= \pa{\pa{\sum_i y_i^\frac{2}{2-\p}}^\frac{2-\p}{2}\pa{\sum_i\frac{v_i^2}{y_i^{2/\p}}}^\frac{\p}{2}}^\frac{2}{\p} \\
        &\circled{2}[\geq] \pa{\sum_i y_i \frac{v_i^{\p}}{y_i}}^\frac{2}{\p} = \pa{\sum_i v_i^{\p}}^\frac{2}{\p}  = \norm{v}_{\p}^2.
    \end{aligned}
    \end{align*}
    In $\circled{1}$ we dropped the first summand which is $\geq 0$, and wrote the expression for the second one. In $\circled{2}$ we used H\"older's inequality $\innp{w, z} \leq \norm{w}_{\q} \norm{z}_{\q^\ast}$ with the norm $\q=\frac{2}{2-\p}$ and its dual $\q^\ast =\frac{2}{\p}$.
\end{proof}

\section{Other proofs from \nameref{sec:acc_inexact_PP_unif_cvx}: Algorithms}\label{app:proofs_of_algorithms}

\begin{proof}\linkofproof{thm:convergence_of_accelerated_not_adaptive_algorithm}
    Our algorithm makes use of a $\deltainexact$-inexact $(\muUnif,\r)$-uniformly convex regularizer with respect to a norm $\norm{\cdot}$, i.e. $\breg(x, y) \geq \frac{\muUnif}{\r}\norm{x-y}^{\r} - \deltainexact$. Note that for convex $h$ we have that $\ell(x) \defi \psi(x) + h(x)$ is also $\deltainexact$-inexact $(\muUnif, \r)$-uniformly convex and if $z$ is a global minimizer of $\ell$, then by the first-order optimality condition, we have  $\ell(x) - \ell(z) \geq \breg[\ell](x, z) \geq \frac{\muUnif}{\r}\norm{x-z}^{\r} - \deltainexact$.

    We use a primal-dual technique in the spirit of Nesterov's estimate sequences \citep{nesterov2004introductory} and the approximate duality gap technique of \citet{diakonikolas2019approximate} in order to naturally define a Lyapunov function that allows to prove convergence. Given $a_i > 0$, for $i \geq 1$ and $A_k\defi \sum_{i=1}^k \ak[i]$ to be chosen later, we define the following lower bound $L_k$ on $f(u)$, for all $k \geq 1$:
        \begin{align}\label{eq:lower_bound_ms}
         \begin{aligned}
             A_k f(u) &\circled{1}[\geq] \sum_{i=1}^k \ak[i] f(\yk[i]) + \sum_{i=1}^k \ak[i]\innp{v_i, u - \yk[i]} - a_i\oldepsilon_i \\
             &\circled{2}[\geq] \sum_{i=1}^k \ak[i] f(\yk[i]) + \min_{z\in\Rd}\left\{\sum_{i=1}^k ( \ak[i]\innp{ v_i, z - \yk[i]}-\ak[i]\oldepsilon_i) + \breg(z, \xk[0])\right\} - \breg(u, \xk[0]) \\
             &\circled{3}[=] \sum_{i=1}^k \ak[i] f(\yk[i]) + \sum_{i=1}^k ( \ak[i]\innp{v_i, \zk[k] - \yk[i]} -\ak[i]\oldepsilon_i) + \breg(\zk[k], \xk[0]) - \breg(u, \xk[0]) \\
             &\defi A_k L_k,
         \end{aligned}
        \end{align}
        where $\circled{1}$ holds because $v_i \in \subdiffeps[\oldepsilon_i] f(\yk[i])$. In $\circled{2}$, we added and subtracted the regularizer $\breg[\psi](u, x_0)$ and took a minimum to remove the dependence of $u$ in the lower bound (except for the term $-\breg(u, x_0)$ that is irrelevant for defining the algorithm, as it will become evident in a moment). Equality $\circled{3}$ simply uses that $\zk[k]$ was defined as the $\argmin$ of that minimization problem. Since $A_0 = 0$, we define $A_0L_0 \defi 0$. We define the $\deltainexact$-inexact $(\muUnif, \r)$-uniformly convex function 
        \[
        \newtarget{def:thing_minimized_in_the_lower_bound}{\elll{k}}(z) \defi \sum_{i=1}^k ( \ak[i]\innp{ v_i, z - \xk[i]} - \ak[i]\oldepsilon_i)  + \breg(z, \xk[0]),
        \] 
        which is part of the bound above, and recall that its minimizer is $\zk$. Now, if we define an upper bound $U_k \geq f(\yk)$ and we show that for some numbers $E_k$, the duality gap $G_k \defi U_k -L_k$ satisfies 
        \begin{equation}\label{eq:lyapunov_property_inexact_prox}
            A_{k} G_{k} - A_{k-1}G_{k-1} \leq E_k \text{ for all }k > 1\text{, and }A_1G_1 - A_0 G_0 = A_1 G_1 \leq \breg(u, x_0) + E_1,
        \end{equation}
        then telescoping the inequalities above, we obtain the following convergence rate after $T$ steps:
        \begin{equation}\label{eq:convergence_rate_inexact}
            f(\yk[T]) - f(u) \leq U_T - L_T = G_T \leq \frac{A_1G_1 + \sum_{i=2}^T E_i}{A_T} \leq \frac{\breg(u,x_0)+ \sum_{i=1}^T E_i}{A_T},
        \end{equation}

    We choose the upper bound $U_k =f(y_k)$, so $G_k = f(y_k) - L_k$. Thus, we have, for all $k \geq 1$:
        \begin{align*}
         \begin{aligned}
             A_{k} &G_{k} - A_{k-1}G_{k-1} - \eventindicator{k=1} \breg(u, \xk[0])
             \circled{1}[=]A_{k-1}(f(\yk[k])-f(\yk[k-1])) + \cancel{\ak[k] f(\yk[k])} \\
             & \quad \cancel{- \sum_{i=1}^{k} \ak[i] f(\yk[i])} - \left(\sum_{i=1}^{k-1} (\ak[i]\innp{ v_i, \zk[k] - \yk[i]} - \ak[i]\oldepsilon_i ) + \breg(\zk[k], \xk[0]) \right) - \ak[k]\innp{ v_{k}, \zk[k] - \yk[k]} +\ak[k]\oldepsilon_k\\
             & \quad \cancel{+ \sum_{i=1}^{k-1} \ak[i] f(\yk[i])} + \left(\sum_{i=1}^{k-1} (\ak[i]\innp{v_i, \zk[k-1] - \yk[i]} - \ak[i]\oldepsilon_i ) + \breg(\zk[k-1], \xk[0]) \right)  \\
             &\circled{2}[\leq] \innp{v_{k}, A_{k-1}(\yk[k]-\yk[k-1]) - \ak[k] (\pm \zk[k-1] + \zk[k]-\yk[k])} -\frac{\muUnif}{\r}\norm{\zk[k-1] - \zk[k]}^{\r}  + \deltainexact + \Ak\oldepsilon_k  \\
             & \circled{3}[=] \innp{v_{k}, A_{k} (y_{k} - x_{k}) + a_{k}(\zk[k-1] - \zk[k]) }  - \frac{\muUnif}{\r}\norm{\zk[k-1] - \zk[k]}^{\r}  + \deltainexact + \Ak\oldepsilon_k\\
             & \circled{4}[\leq] A_{k}\innp{v_{k}, \yk[k] - \xk[k]} + \frac{a_{k}^{\dualnumber{r}}}{\muUnif^{1/(\r-1)}\dualnumber{r}}\norm{v_{k}}_\ast^{\dualnumber{r}}  + \deltainexact + \Ak\oldepsilon_k \\
             & \circled{5}[\leq] A_{k}\innp{\hat{v}_{k}, \yk[k] - \xk[k]} + A_{k}\norm{v_{k}-\hat{v}_{k}}_\ast \cdot \norm{\yk[k] - \xk[k]} + \frac{2^\frac{1}{\r-1}a_{k}^{\dualnumber{r}}}{\dualnumber{r}\muUnif^\frac{1}{\r-1}}\left(\norm{\hat{v}_{k}}_\ast^{\dualnumber{r}} + \norm{v_{k}-\hat{v}_{k}}_\ast^{\dualnumber{r}} \right)  + \deltainexact + \Ak[k]\oldepsilon_k  \\
             & \circled{6}[\leq] \left(  -\frac{A_{k}}{\lambda_{k}} + \frac{\sigma A_{k}}{\lambda_{k}} + \frac{a_{k}^{\r/(\r-1)}}{\lambda_{k}^{\r/(\r-1)}}\left(\frac{2}{\muUnif}\right)^{\frac{1}{\r-1}}\frac{1+\sigma^{\dualnumber{r}}}{\dualnumber{r}} + \frac{\sigma' A_k}{\lambda_k}\right) \norm{\yk[k]-\xk[k]}^{\r} + \deltainexact \\
             & \circled{7}[\leq] \deltainexact \defi E_k.
         \end{aligned}
        \end{align*}
        Above, we wrote the definition of the gaps in $\circled{1}$, we canceled some terms and we used the indicator on the left hand side to handle the cases $k=1$ and $k> 1$ at the same time. 
        In $\circled{2}$, we applied the enlarged subgradient property on the first term, which gives an error of $A_{k-1}\oldepsilon_k$ that we group with the other $a_k \oldepsilon_k$ error, and we grouped the resulting expression with another term, and we used that the terms in parentheses are $\elll{k-1}(z_{k-1}) - \elll{k-1}(z_{k})$. The inexact uniform convexity of $\elll{k-1}(\cdot)$ and the fact that $z_{k-1}$ is its minimizer implies the bound. In $\circled{3}$, we used that by definition of $\xk[k]$ it is $A_{k} x_{k} = A_{k-1} \yk[k-1] + a_{k}\zk[k-1]$. We had added and subtracted $\zk[k-1]$ to apply H\"older's and Young's inequalities in $\circled{4}$, namely $\innp{v, u} \leq \norm{v}_\ast \norm{u}  \leq \frac{c}{\dualnumber{r}}\norm{v}_\ast^{\dualnumber{r}} + \frac{1}{cp}\norm{u}^{\r}$, with $c = a_{k} $, and where $\dualnumber{r} \defi (1-1/\r)^{-1}$. In $\circled{5}$, we added and subtracted some $\hat{v}_{k}$ terms and use bounds to make $\norm{v_{k}-\hat{v}_{k}}_\ast$ appear, and other terms that we can bound  with something proportional to $\norm{y_k-x_k}^{\r}$. For the second summand, after applying the triangular inequality we used the means inequality $\frac{a+b}{2} \leq (\frac{a^{\dualnumber{r}}+b^{\dualnumber{r}}}{2})^{1/\dualnumber{r}}$, for $\dualnumber{r} > 1$. In $\circled{6}$ we applied the inequalities of our oracle $\iproxoracle[r]$ criterion for the second and fourth terms and used \cref{eq:property_subgrads_of_norm_to_the_r} that yields equality for the first terms and $\norm{\hat{v}_k}_\ast^{\dualnumber{r}}$. 

Let $C \defi \frac{\muUnif}{2}\left(\frac{\dualnumber{r}(1-\sigma-\sigma')}{1+\sigma^{\dualnumber{r}}}\right)^{\r-1}$. It is enough to satisfy $a_{k}^{\r} \leq C A_{k}^{\r-1}\lambda_k$ to make $\circled{7}$ hold, and then we define $E_k$ as $\deltainexact$. We choose $\ak[k] > 0$ as large as possible, that is, $a_{k}^{\r} = C A_{k}^{\r-1}\lambda_k$. For notational simplicity, let $D_k \defi C \lambda_k$. Then, since $A_k = a_k +A_{k-1}$, we can express the equation as $\hat{a}_k^{\r/(\r-1)} = \hat{a}_k + \hat{A}_{k-1}$, where $\hat{a}_k \defi a_k D_k^{-1}$ and $\hat{A}_{k-1} \defi A_{k-1} D_k^{-1}$. Now, using this expression and Young's inequality, we obtain
     \[
         \hat{A}_{k-1}^{1/\r} = \hat{a}_k^{1/\r} (\hat{a}_k^{1/(\r-1)} - 1)^{1/\r}  \leq \frac{\hat{a}_k^{1/(\r-1)}}{\dualnumber{r}} + \frac{\hat{a}_k^{1/(\r-1)}-1}{\r} = \hat{a}_k^{1/(\r-1)} - \frac{1}{\r},
     \] 
     which implies $\circled{1}$ below
     \[
         \hat{a}_k + \hat{A}_{k-1} \circled{1}[\geq] \left(\hat{A}_{k-1}^{1/\r}+\frac{1}{\r}\right)^{\r-1} + \hat{A}_{k-1} \circled{2}[\geq] \left(\hat{A}_{k-1}^{1/\r} + \frac{1}{\r}\right)^{\r}. 
     \] 
     Above, $\circled{2}$ holds by Bernoulli's inequality $(1-1/x)^{\r} \geq 1- \r/ x$ for $x,\r >1$, since dividing by the right hand side and simplifying gives $\frac{\r}{\hat{A}_{k-1}^{1/\r} \r+1} + \left(1-\frac{1}{\hat{A}_{k-1}^{1/\r}\r +1}\right)^{\r} \geq 1$, where here $x = \hat{A}_{k-1}^{1/\r} \r +1 > 1$. Multiplying by $D_k$ and taking an $\r$-th root, we obtain 
     \begin{equation}\label{eq:recursion_lower_bounding_Ak_unif_cvx}
         A_k^{1/\r} = (a_k+A_{k-1})^{1/\r} \geq A_{k-1}^{1/\r} + \frac{1}{\r}D_k^{1/\r} = A_{k-1}^{1/\r} + \frac{1}{\r}C^{\frac{1}{\r}}\lambda_k^{1/\r},
     \end{equation}
     and thus, $A_k^{1/\r} \geq \frac{1}{\r}C^{\frac{1}{\r}}\sum_{i=1}^k \lambda_i^{1/\r}$. Hence, we conclude by \cref{eq:convergence_rate_inexact} that for any $T \geq 1$, we have:
    \[
        f(y_T)-f(u) \leq \frac{\breg(u, x_0) + \deltainexact T}{A_T} \leq \frac{\r^{\r}( \breg(u, x_0)+ \deltainexact T)}{C\left(\sum_{i=1}^{T} \lambda_i^{1/\r}\right)^{\r}} =\bigopl{\r}{\frac{\breg[\psi](u, x_0) + \deltainexact T}{\muUnif\left(\sum_{i=1}^{T} \lambda_i^{1/\r}\right)^{\r}}}.
    \] 
\end{proof}

We note that in the proof above, if we had set $C \defi \frac{\muUnif}{2}\left(\frac{\dualnumber{r}(1-\sigma-\sigma')}{2(1+\sigma^{\dualnumber{r}})}\right)^{\r-1}$ instead, then we would get $E_k$ is $\delta$ and a negative term, which after concluding and using $f(y_T) - f(x^\ast) \geq 0$, yields a similar statement to the second property in \cref{thm:adaptive_alg_guarantee}.

We now proceed to prove how finding an approximate critical point of the regularized Taylor subproblems satisfies the oracle criteria.

\begin{proof}\linkofproof{lemma:inexact_criterion_w_taylor_expansion}
    Firstly, we have $\norm{\nabla f(y)-\nabla \taylorf[\q][f](y; x)}_\ast \leq \frac{\L}{(\q-1)!} \norm{y - x}^{\q+\nu-1}$, see \cref{lemma:taylor_and_derivaties_are_close_to_the_function}.
    Then, for $v_k = \nabla f(y_k)$ and $\lambda_k \defi \hat{\lambda}\norm{y_k-x_k}^{\r-\q-\nu} = \frac{\sigma(\q-1)!}{2L} \norm{y_k-x_k}^{\r-\q-\nu} $ we have
    \begin{align*}
         \begin{aligned}
             \lambda_k \norm{v_k - \hat{v}_k}_\ast &\leq \lambda_k \left( \norm{\nabla f(y_k) - \nabla \taylorf[\q][f](y_k;x_k)}_\ast +  \norm{\nabla \taylorf[\q][f](y_k;x_k) - \hat{v}_k}_\ast \right) \\
             &\circled{1}[\leq] \lambda_k\frac{\L}{(\q-1)!}\left(\norm{y_k-x_k}^{\q+\nu-1} + \norm{y_k-x_k}^{\q+\nu-1} \right) = \sigma\norm{y_k-x_k}^{\r-1}.
         \end{aligned}
        \end{align*}
        where $\circled{1}$ uses the bound above in \cref{lemma:taylor_and_derivaties_are_close_to_the_function} and the guarantee on $y_k$.
\end{proof}

We now present the following two lemmas, which develop the key ideas to show the convexity of some of our Taylor subproblems in \cref{lemma:inexact_criterion_w_taylor_expansion}.

\begin{lemma}[Hessian property of powers of some norms]\label{lemma:hessian_property_of_power_of_p_norm}
    Let $\norm{\cdot}$ be a norm such that $\psi(x) = \norm{x}^2$ is twice differentiable and $\muUnif$-strongly convex. Then, the function $g_{\q}(x) \defi \frac{1}{\q} \norm{x}_{\p}^{\q}$ satisfies 
    \[
        \nabla^2 h(x)[v, v]  \geq \frac{\muUnif}{2} \norm{x}_{\p}^{\q-2} \norm{v}_{\p}^2, \text{ for all } x, v \in \Rd.
    \] 
\end{lemma}

\begin{proof}
    Let $x \in \Rd$. Writing $g_{\q}(x) = \frac{1}{\q}(\psi(x))^{\q / 2}$, we differentiate $g_{\q}$ using the chain rule:
    \[
        \nabla g_{\q}(x) = \frac{1}{2}\psi(x)^{\frac{\q-2}{2}} \nabla\psi(x),
    \] 
    and thus, for any $v \in \Rd$:
    \begin{align*}
    \begin{aligned}
        \nabla^2 g_{\q}(x)[v, v] &= \frac{\q-2}{4}\psi(x)^{\frac{\q-4}{2}} ( v^T \nabla\psi(x)\nabla\psi(x)^T v) + \frac{1}{2}\psi(x)^{\frac{\q-2}{2}}\nabla^2\psi(x)[v, v] \\
        & \geq \frac{\muUnif}{2}\norm{x}^{\q-2}\norm{v}^2.
    \end{aligned}
    \end{align*}
    In the inequality, we dropped the first summand, which is nonnegative, and we substituted the value of $\psi(x)$ and used the strong convexity of $\psi$.

\end{proof}

\begin{lemma}\label{lemma:taylor_and_derivaties_are_close_to_the_function}
    Let $\q \in \mathbb{Z}_+$ and let $\norm{\cdot}$ be an arbitrary norm.  If $\norm{\nabla^{\q} f(x) - \nabla^{\q} f(y)}_\ast \leq \L \norm{x-y}^{\nu}$ then, for all $\ell \in \{0, 1, \dots, \q-1 \}$, we have
    \[
        \norm{\nabla^{\ell} f(x) - \nabla^{\ell} \taylorf[\q][f](x;y) }_\ast \leq \frac{\L}{(\q-\ell)!} \norm{x-y}^{\q - \ell + \nu}.
    \] 
\end{lemma}
We note that \citep[Lemma 2.5]{song2019unified} claimed this fact for $\ell = 0$ and $\ell =1$, but the proof for $\ell = 1$ was not correct since the chain rule was not used in their equation (A.12). We provide a complete proof and of a more general statement, namely for all $\ell$. Also note that above we followed the convention $\nabla^0 f \equiv f$.

\begin{proof}
    Define the quantity
    \[
        C_{i,j} \defi \frac{1}{j!} \int_0^1(1-\tau)^j \nabla^{i+1} f(y+\tau(x-y))[x-y]^{j+1} \mathrm{~d} \tau.
    \] 
    that for $i > 1$ satisfies, by integrating by parts:
    \begin{align*}
    \begin{aligned}
        C_{i, i}&{=}\frac{1}{i!} \Big[ (1-\tau)^i \nabla^{i} f(y{+}\tau(x{-}y))[x{-}y]^{i}\Big]_{\tau=0}^1 {+} \frac{1}{(i-1)!}\int_0^1 (1-\tau)^{i-1} \nabla^{i} f(y{+}\tau(x{-}y))[x{-}y]^{i} \mathrm{~d} \tau \\
        & = -\frac{1}{i!}\nabla^i f(y)[x-y]^i + C_{i-1, i-1}.
    \end{aligned}
    \end{align*}
    And also $C_{0, 0} = f(x) - f(y)$ by simple integration. In turn, these facts imply:
\begin{align*}
    \begin{aligned}
        C_{\q-1, \q-1} &= C_{0, 0} + \sum_{i=1}^{\q-1} \big(C_{i, i} - C_{i-1, i-1}\big) = f(x) - \taylorf[\q][f](x;y) +\frac{1}{\q!} \nabla^{\q} f(y)[x-y]^{\q} .
    \end{aligned}
    \end{align*}
    Taking derivatives with respect to $x$, we obtain
    \begin{align}\label{eq:aux:easy_derivative_of_cq_1}
    \begin{aligned}
        \nabla^\ell C_{\q-1, \q-1} &= \nabla^\ell f(x)-\nabla^\ell \taylorf[\q][f](x ; y)+\frac{1}{(\q-\ell)!} \nabla^{\q} f(y)[x-y]^{\q-\ell} \\
        & = \nabla^\ell f(x)-\nabla^\ell \taylorf[\q][f](x ; y)+\frac{1}{(\q-\ell-1)!} \nabla^{\q} f(y)[x-y]^{\q-\ell} \int_0^1(1-\tau)^{\q-\ell-1} \mathrm{~d} \tau.
    \end{aligned}
    \end{align}
    Now, if we differentiate the definition of $C_{i, j}$ with respect to $x$, we obtain, for $j > 1$:
\begin{align}\label{eq:aux:differentiating_cq_1}
    \begin{aligned}
        \nabla C_{i,j} &= \frac{j+1}{j!} \int_0^1(1-\tau)^{j} \nabla^{i+1} f(y+\tau(x-y))[x-y]^{j} \mathrm{~d} \tau \\
        &\quad + \frac{1}{j!} \int_0^1 \tau (1-\tau)^{j} \nabla^{i+2} f(y+\tau(x-y))[x-y]^{j+1} \mathrm{~d} \tau\\
        &\circled{1}[=] \frac{j+1}{j!} \int_0^1(1-\tau)^{j} \nabla^{i+1} f(y+\tau(x-y))[x-y]^{j} \mathrm{~d} \tau \\
        & \quad + \cancel{\frac{1}{j!}\Big[ \tau(1-\tau)^{j} \nabla^{i+1} f(y+\tau(x-y)) [x-y]^{j} \Big]_{\tau=0}^1}  \\ 
        &\quad  - \frac{1}{j!}\int_0^1 \nabla^{i+1} f(y+\tau(x-y))[x-y]^{j} \Big( (1-\tau)^{j} - j\tau(1-\tau)^{j-1}  \Big) \mathrm{~d}\tau \\
        &\circled{2}[=] \frac{1}{(j-1)!} \int_0^1 \nabla^{i+1} f(y+\tau(x-y))[x-y]^{j}\Big((1-\tau)^j + \tau(1-\tau)^{j-1} \Big) \mathrm{~d}\tau \\
        &= \frac{1}{(j-1)!} \int_0^1 \nabla^{i+1} f(y+\tau(x-y))[x-y]^{j}(1-\tau)^{j-1} \mathrm{~d}\tau \\
        &= C_{i, j-1}.
    \end{aligned}
    \end{align}
    Above, $\circled{1}$ holds by integrating the second summand by parts and canceling one term by using $j \neq 0$, and $\circled{2}$ groups and simplifies some terms, using $j > 0$. Thus, $\nabla^{\ell}C_{\q-1, \q-1}$ is also equal to $C_{\q-1, \q-1-\ell}$, as long as $\ell \leq \q-1 $. Note that we also have $\nabla C_{0, 0} = \nabla f(x)$ since $C_{0, 0} = f(x) - f(y)$.

    Combining \cref{eq:aux:easy_derivative_of_cq_1} and \cref{eq:aux:differentiating_cq_1}, we obtain, for any $\ell \in \{0, 1, \dots, \q-1\}$:
\begin{align*}
    \begin{aligned}
        (\q-\ell-1)!&\norm{\nabla^\ell f(x) - \nabla^\ell \taylorf[\q][f](x; y)}_\ast \\
        &= \norml{ \int_0^1\Big(\nabla^{\q} f(y)  - \nabla^{\q} f(y+\tau(x-y)) \Big)[x-y]^{\q-\ell} (1-\tau)^{\q-\ell-1} \mathrm{~d} \tau }_\ast \\
        &\circled{1}[=] \max_{v: \norm{v} \leq 1} \int_0^1\Big(\nabla^{\q} f(y)  - \nabla^{\q} f(y+\tau(x-y)) \Big)[x-y]^{\q-\ell}[v]^\ell (1-\tau)^{\q-\ell-1} \mathrm{~d} \tau   \\
        & \circled{2}[\leq]  \int_0^1 (1-\tau)^{\q-\ell-1} \mathrm{~d} \tau \cdot \max_{\tilde{\tau} \in [0,1], \norm{v}\leq 1} \Big(\nabla^{\q} f(y)  - \nabla^{\q} f(y+\tilde{\tau}(x-y)) \Big)[x-y]^{\q-\ell}[v]^\ell    \\
        & \circled{3}[\leq]  \frac{1}{\q-\ell} \max_{\tilde{\tau} \in [0,1]} \norm{\nabla^{\q} f(y)  - \nabla^{\q} f(y+\tilde{\tau}(x-y))}_\ast \norm{x-y}^{\q-\ell}  \\
        & \circled{4}[\leq]  \frac{\L}{\q-\ell} \norm{x-y}^{\q+\nu-1}.
    \end{aligned}
\end{align*}
    We used the definition of the dual norm in $\circled{1}$ for symmetric operators, and in $\circled{2}$ we bounded the expression by moving the $\max$ inside and we bounded part of the integrand by its maximum. In $\circled{3}$ we used the definition of the operator norm on a symmetric operator and used $\norm{v} \leq 1$. Finally, in $\circled{4}$ we used the H\"older continuity property \cref{eq:holder_cont} and $\tilde{\tau} \leq 1$.
\end{proof}

Now we have all of the ingredients to prove \cref{prop:convexity_of_taylor_subproblems}.

\begin{proof}\linkofproof{prop:convexity_of_taylor_subproblems}
    Let $v$ such that $\norm{v}_{\p} = 1$ and define $g_s^x(y) \defi \frac{1}{s}\norm{y-x}_{\p}^s$ as in \cref{lemma:hessian_property_of_power_of_p_norm} but with a shift. We have the following:
\begin{align*}
    \begin{aligned}
        0 &\circled{1}[\leq]  \nabla^2 f(y)[v, v] \circled{2}[\leq] \nabla^2 \taylorf[\q][f](y;x)[v, v] + \frac{\L}{(\q-2)!}\norm{x-y}_{\p}^{\q-2+\nu} \\
        &\circled{3}[\leq] \nabla^2 \taylorf[\q][f](y;x)[v, v] + \frac{2L}{\hat{\mu}(\q-2)!} \nabla^2 g_{\q+\nu}^x(y)[v, v] \leq \nabla^2 F(y)[v, v],
    \end{aligned}
    \end{align*}
    where $\circled{1}$ holds by convexity of $f$ while $\circled{2}$ is by \cref{lemma:taylor_and_derivaties_are_close_to_the_function}, and $\circled{3}$ uses \cref{lemma:hessian_property_of_power_of_p_norm} which also holds true for the shifted function we defined above, without loss of generality by shifting the domain so $x$ is $0$. Thus, $F(y)$ is convex.

    For the second part of the proposition, fix $\p \in (1, 2]$ and use that by \cref{lemma:regularity_of_regularizers}, it is $\hat{\mu} = 2(\p-1)$, by rescaling, which translates to the requirement for the subproblem $\frac{\L}{(\p-1)(\q-2)!} \leq M = \frac{1}{\hat{\lambda}} = \frac{\L}{\sigma(\q-1)!} $ in \cref{lemma:inexact_criterion_w_taylor_expansion}, equivalent to $\sigma \leq \frac{\p-1}{\q-1}$.
\end{proof}

We are now ready to prove the convergence rates for high-order smooth convex functions.

\begin{proof}\linkofproof{thm:guarantee_high_order_smooth_cvx_methods}
\paragraph{Solving the case $q+\oldnu \leq \max\{2, p\}$.}
    Recall that we defined $\m \defi \max\{2, \p\}$. We use the regularizers in \cref{lemma:regularity_of_regularizers}. Depending on whether $\p > 2$, one or the other of these two regularizers is $( \bigop{\p}{1}, \m)$-uniformly convex with respect to $\norm{\cdot}_{\p}$ and therefore that regularizer is, by \cref{lemma:inexact_unif_convex_from_unif_convex}, $\deltainexact$-inexact $(\muUnif, \q +\nu)$-uniformly convex regularizer with respect to $\norm{\cdot}_{\p}$, for some $\deltainexact$, $\muUnif$ that are a function of a constant $a$, that we will determine later. We use such regularizer. Note that if $\p \leq 2 $, the restriction $\q+\nu \leq \m  = \max\{2,\p\} = 2$ along with $\q \geq 1$, $\nu \in (0, 1]$ implies $\q=1$. But for $\p > 2$ we may still be working in greater order $\q>1$. 

    As established in \cref{lemma:inexact_criterion_w_taylor_expansion}, we can solve the inexact proximal problems in \cref{alg:non_euclidean_accelerated_proximal_point_unif_convex} with a single call of the $\q$-th order oracle if we set $\r = \q + \nu$ for the proximal parameter $\lambda_k = \frac{\sigma(\q-1)!}{2L}$. This parameter $\lambda_k$, unlike for other cases, does not depend on $y_k$. This fact avoids having to perform a binary search or an adaptive guess on the value of the proximal parameter, so we can use \cref{alg:non_euclidean_accelerated_proximal_point_unif_convex} instead of \cref{alg:adaptive_non_euclidean_accelerated_proximal_point_unif_convex}. Set $\sigma = \sigma' = 1/4$ for simplicity. Applying the results from the previous section, we obtain a convergence rate of
\[
    f(y_T) - f(\xast) \leq \bigopl{\p, \r}{\frac{\L(R_{\p}^{\m} + \deltainexact T)}{\muUnif T^{\r}}} = \bigopl{\p, \r}{\frac{\L\norm{x-x_0}_{\p}^{\m}}{a^{\frac{\m}{\r}} T^{\r}} + \L a^{\frac{\m}{\m-\r}} T^{1-\r} },
\] 
    where $R_{\p}= \Theta_{\p}(\breg[\psi](\xast, x_0)^{1/\r})$ is the initial distance $ \norm{\xast-x_0}_{\p}$ measured with the $\p$-norm. But we could also set it to an upper bound.
The bound above is convex on $a>0$. By taking derivatives and finding a zero, the bound is found to be optimized at a value $a = \bigopl{\p,\r}{R_{\p}^{\r\frac{\m-\r}{\m}}T^{-\frac{\r(\m-\r)}{\m^2}}}$. Thus, if we make this choice of $a$, the convergence rate becomes:
\[
    f(y_T) - f(\xast) =  \bigopl{\p, \r}{\frac{\L R_{\p}^{\r}}{T^{\frac{mr +\r -\m}{\m}}}} = \bigopl{\p, \r}{\frac{\L R_{\p}^{\q+\nu}}{T^{\frac{(\m+1)(\q+\nu) -\m}{\m}}}}.
\] 
Note that the step sizes $a_k$ depend on the constant $a$ via $\muUnif$ via the constant $C$.

\paragraph{Solving the case $q+\oldnu > \max\{2, p\}$.}

    We run \cref{alg:adaptive_non_euclidean_accelerated_proximal_point_unif_convex} with $\r = \m =  \max\{2, \p \}$ with $\sigma = \sigma' = \frac{1}{4}$ for simplicity. One may want to run it with $\sigma = \frac{p-1}{q-1}$ when $p \in (1, 2]$ and $q \geq 3$, according to \cref{remark:critical_points_exist_and_sometimes_only_one}. Note that this only changes constants $\bigop{q+\nu, r}{1}$ in our analysis. From \cref{eq:lower_bounding_Ak_in_adaptive} in the analysis of \cref{alg:adaptive_non_euclidean_accelerated_proximal_point_unif_convex}, we have that there is a set of iterates $Q_T\subseteq [T]$ and some numbers $r_k \geq 0$ such that
\begin{equation}\label{eq:summary_guarantee_adaptive_algorithm}
        A_{T}^{1/\r} \geq \widehat{C} \sum_{k \in Q_{T}} \widehat{\lambda}_k^{1/\r} (\alpha^{1/\r})^{r_k -2}.
\end{equation}
    for the constant $\widehat{C} \defi C^{1/\r} / (2r)$ where $C$ is defined in \cref{alg:adaptive_non_euclidean_accelerated_proximal_point_unif_convex}, and such that $\sum_{k \in Q_T} r_k = (T-1) / 2$.
    For notational convenience, we use $\hat{q} \defi \q +\nu$. By \cref{lemma:inexact_criterion_w_taylor_expansion}, in the case of high-order methods, we can implement the oracle with one call to the $\q$-th order oracle for $\lambda_k^{\frac{\r}{\r-\hat{q}}} \defi \hat{\lambda}^{\frac{\r}{\r-\hat{q}}} \norm{\yktilde[k]-x_k}^{\r}$ for $\hat{\lambda} \defi \frac{\sigma (\q-1)!}{2L}$. Thus, the analysis in \cref{thm:adaptive_alg_guarantee} yields 
    \[
        \breg[\psi](\xast, x_0) \geq \frac{1-\sigma - \sigma'}{2}\sum_{k \in Q_T} A_k \norm{\yktilde[k]-x_k}^{\r}\widehat{\lambda}_k^{-1} =  \frac{1-\sigma - \sigma'}{2}\hat{\lambda}^{\frac{\r}{\r-\hat{q}}}\sum_{k \in Q_T} A_k \widehat{\lambda}_k^{\frac{\hat{q}}{\r-\hat{q}}}.
    \] 
    We will make use of the reverse H\"older inequality, with which is a common tool in analysis of Monteiro-Svaiter acceleration. For $s> 1$ and positive numbers $\alpha_i, \beta_i$, we have
    \[
        \sum_{i} \alpha_i \beta_i \geq \left(\sum_i \alpha_i^{1/s}\right)^s \left(\sum_i \beta_i^{1 / (1-s)}\right)^{1-s}. 
    \] 
    We apply this inequality in $\circled{2}$ below, for $s = \frac{\hat{q} + \r\hat{q} - \r}{\r\hat{q}} > 1$ where the inequality for $s$ holds by the assumption of this section $\hat{q} = \q + \nu > \max\{2, \p\} = \r$. Also take into account that $\frac{1}{1-s} = \frac{\r\hat{q}}{\r-\hat{q}}$. Thus, we obtain the following estimate    
    \begin{align}\label{eq:estimate_adaptive_method}
         \begin{aligned}
             \widehat{C}^{-1}  A_t^{1/\r} &\circled{1}[\geq] \sum_{k\in Q_t} \widehat{\lambda}_k^{1/\r} (\alpha^{1/\r})^{r_k -2} = \sum_{k\in Q_t}  \left(A_k^{s-1} (\alpha^{1/\r})^{r_k -2} \right) (A_k^{1-s}\widehat{\lambda}_k^{1/\r} ) \\
             &\circled{2}[\geq] \left(\sum_{k\in Q_t} A_k^{1 - 1/s}(\alpha^{1/(rs)})^{r_k -2}\right)^s  \left(\sum_{k\in Q_t} A_t \widehat{\lambda}_k^{\frac{\hat{q}}{\r - \hat{q}}}\right)^{1-s} \\
             &\circled{3}[\geq] \left(\sum_{k\in Q_t} A_k^{\frac{\hat{q}-\r}{\hat{q} + \r\hat{q}-\r}} r_k c_{\alpha, s} \right)^s\left(\frac{2\breg[\psi](\xast, x_0)}{1-\sigma-\sigma'}\hat{\lambda}^{-\frac{\r}{\r-\hat{q}}}\right)^{1-s}
         \end{aligned}
    \end{align}
    where $\circled{1}$ uses $\cref{eq:summary_guarantee_adaptive_algorithm}$. In $\circled{3}$ we used \cref{lemma:numerical_aux_bound_MS} with $c_{\alpha,s} = \alpha^{-2/(rs)}\min\{1, \frac{1}{rs}\ln(\alpha)\}$. Now by using the notation $B_k \defi A_k^{\frac{ \hat{q}- \r}{\hat{q} + \r\hat{q} - \r}}$ and 
    \[
        \Gamma \defi \widehat{C}^{1/s} c_{\alpha, s} \left(\frac{2\breg[\psi](\xast, x_0)}{1-\sigma-\sigma'}\hat{\lambda}^{-\frac{\r}{\r-\hat{q}}}\right)^{\frac{1}{s}-1}
    \] 
    we have
    \[
        B_t^{\frac{\hat{q}}{\hat{q}-\r}} = A_t^{\frac{1}{rs}}  \geq \Gamma\sum_{k \in Q_t \cap [t]} B_k r_k \text{ for all } t.
    \] 
    and note that the exponent above on the left hand side is $\frac{\hat{q}}{\hat{q}-\r} >1$. Thus, we can use \citep[Lemma 3]{carmon2022optimal} which yields
    \[
        B_T \geq \left( \frac{\hat{q}-\r+\r^2}{\hat{q} + \r\hat{q} - \r}\Gamma \sum_{k \in Q_T} r_t\right)^{\frac{\hat{q}- \r}{\r}},
    \] 
    or equivalently
    \[
        A_T \geq \left(\frac{\hat{q}-\r+\r^2}{\hat{q} + \r\hat{q} - \r}\Gamma \frac{T-1}{2}\right)^{\frac{\hat{q}+\r\hat{q}-\r}{\r}}  = \bigomegapl{\hat{q}, \r}{\breg[\psi](\xast, x_0)^{\frac{\r-\hat{q}}{\r}}\hat{\lambda}^{-1} T^{\frac{\hat{q}+\r\hat{q}-\r}{\r}}}.
    \] 
    Note that above we took into account that $\alpha$ is a constant. The lower bound on $A_T$ and the same reasoning as in \cref{eq:convergence_rate_inexact} yield the convergence rate
    \[
        f(y_T) - f(\xast) =  \bigopl{\hat{q}, \r}{\frac{\L R_{\p}^{\hat{q}}}{T^{\frac{(\r+1)\hat{q}-\r}{\r}}}} = \bigopl{\hat{q}, \r}{\frac{\L R_{\p}^{\q+\nu}}{T^{\frac{(\m+1)(\q+\nu)-\m}{\m}}}}, 
    \] 
where $R_{\p} = \Theta_{\p}(\breg[\psi](\xast, x_0)^{1/\r})$ is the initial distance to a minimizer measured with the $\p$-norm, up to constants, due to our choice of regularizer. 
\end{proof}

\begin{lemma} \label{lemma:numerical_aux_bound_MS}
    For $a > 1$ and $b \geq 0$, we have $a^{b-2} \geq a^{-2}\min\{1, \ln(a)\} b$.
\end{lemma}
\begin{proof}
    It holds at $b = 0$. Taking derivatives with respect to $b$, it is clear that the derivative of the left hand side is greater than the one of the right hand side for all $b \geq 0$.
\end{proof}

\begin{proof}\linkofproof{prop:ball_opti_oracle_acceleration}
    Analogously to the case $\q+\nu > \max\{2, \p\}$ in the proof of \cref{thm:guarantee_high_order_smooth_cvx_methods}, we have for $r = m = \max\{2, \p\}$, that by \cref{thm:adaptive_alg_guarantee}:
    \[
        \breg[\psi](\xast, x_0)  = \bigomegal{\frac{1-\sigma - \sigma'}{2}\sum_{k \in Q_t} A_k \rho^{\r}\widehat{\lambda}_k^{-1} },
    \] 
    and thus, using the same as \cref{eq:estimate_adaptive_method} where the reverse H\"older's inequality is applied for $s = \frac{1+\r}{\r} > 1$, we obtain
    \[
        \widehat{C} A_t^{1 / \r}  = \bigomegapl{\r}{ \left(\sum_{k\in Q_t} A_k^{\frac{1}{\r+1}} r_k \right)^{\frac{\r+1}{\r}} \rho \breg[\psi](\xast, x_0)^{-1 / \r} }.
    \] 
    Taking a power of $\frac{\r}{\r+1}$ and using $B_k \defi A_k^{\frac{1}{\r+1}}$ we obtain
    \[
        B_t  = \bigomegapl{\r}{\rho^{\frac{\r}{\r+1}}\breg[\psi](\xast, x_0)^{-\frac{1}{\r+1} } \sum_{k \in Q_t} B_k r_k} \text{ for all } t.
    \] 
    Thus, by \citep[Lemma 3]{carmon2022optimal}, and the fact that for our regularizers it is $R_{\p} = \Theta_{\p}(\breg[\psi](\xast, x_0)^{1/\r})$, we obtain $B_T  \geq \exp\left(\bigomegapl{\r}{T (\rho / R_{\p})^{\frac{\r}{\r+1}} + \ln(A_1)} \right)$. Note that by \cref{eq:lyapunov_property_inexact_prox} and the fact that $E_i \leq 0$, it is enough to obtain $A_T \geq \frac{\breg[\psi](\xast, x_0)}{\epsilon}$ in order to reach an $\epsilon$-minimizer. Hence, there is a  $T = \widetilde{\Theta}_{\r}\left(\left(\frac{R_{\p}}{\rho}\right)^{\frac{\r}{\r+1}} \right)$ such that after at most that number of iterations, we find an $\epsilon$-minimizer. 
\end{proof}

    \section{Unaccelerated Proximal Point Algorithm Analysis}\label{sec:unaccelerated_inexact_high_order}
    In this section, we analyze an algorithm for an unaccelerated method for high-order smooth convex optimization. In particular, this method matches the lower bound when smoothness is measured with respect to $\norm{\cdot}_\infty$, a case that was not covered by the accelerated method in \cref{thm:guarantee_high_order_smooth_cvx_methods}. We also analyze an unaccelerated non-Euclidean ball-optimization-oracle algorithm.

    The algorithm is simple. Sequentially iterate 
    \begin{equation}\label{eq:unacc_proximal_point_alg}
        x_{k+1}, v_k \gets \iproxoracle[r](x_k, \lambda),
    \end{equation}
    where $\iproxoracle[r]$ is the inexact proximal oracle in \cref{eq:inexact_prox_oracle_properties} and $r = q + \nu$, $\lambda_k = 1 / \L$ if the function $f$ to be optimized is convex and $q$-th order $(\L, \nu)$-H\"older smooth with respect to a norm $\norm{\cdot}$. This is the setting explained in \cref{lemma:inexact_criterion_w_taylor_expansion}, that requires a single call to the $q$-th order oracle.
    The convergence of the algorithm after $T+1$ iterations depends on $R \defi \max_{k\in[T]} \norm{x_i-\xast}$ although any upper bound works as well. For instance, if $f$ is the sum of a high-order smooth function and the indicator function of a compact set $\X$, we can use its diameter, or we can add the constraint $\X = B^{\norm{\cdot}_{\p}}(x_0, C\norm{x_0- \xast}_{\p})$ for some $C \geq 1$.
    \begin{theorem}\label{thm:unaccelerated_alg}
        After $T+1$ iterations, the algorithm described in \cref{eq:unacc_proximal_point_alg} satisfies.
          \[
              f(\xk[T+1]) - f(\xast) = \bigopl{q+\nu}{\frac{\L R^{q+\nu}}{T^{q+\nu-1}}}.
          \] 
    \end{theorem}
    \begin{proof}
        Recall that the oracle requires $v_k \in \subdiffeps[\oldepsilon_k] f(y_k)$. In this algorithm, we assume $3\sigma + \frac{2A_k}{A_{k-1}}\sigma'  \in (0, 1)$ for all $k \in [T]$    We define $U_k \defi f(x_{k+1})$ and $A_k = A_{k-1} + a_k = \sum_{i=1}^k a_k$, for $a_k>0$ to be determined later, and $G_k \defi U_k - L_k$. The lower bound $L_k$ on $f(\xast)$ is defined via
\[
A_k L_k \defi \sum_{i=1}^k \ak[i] f(\xk[i+1]) + \sum_{i=1}^k \ak[i]\innp{ v_i, \xast - \xk[i+1]} - a_i\oldepsilon_i \leq f(\xast),
\] 
        using the inexact subgradient property in the definition of the inexact proximal oracle. Note that in particular $A_0L_0 \defi 0$. We have, for all $k \geq 1$:
\begin{align}\label{eq:bound_unacc_inexact_ppa}
         \begin{aligned}
             A_{k} &G_{k} - A_{k-1}G_{k-1} \circled{1}[=]A_{k-1}(f(x_{k+1})-f(\xk[k])) + \Ccancel[red]{ \ak f(\xk[k+1])} \\ \\
             & \quad -\left( \Ccancel[red]{ \sum_{i=1}^{k} \ak[i] f(\xk[i+1])} + \Ccancel[blue]{\sum_{i=1}^{k-1} \ak[i]\innp{ v_i, \xast - \xk[i+1]} - a_i\oldepsilon_i} \right) - \ak[k]\innp{ v_k, \xast - \xk[k+1]} + a_k \oldepsilon_k \\
             & \quad + \left(\Ccancel[red]{\sum_{i=1}^{k-1} \ak[i] f(\xk[i+1])} + \Ccancel[blue]{\sum_{i=1}^{k-1} \ak[i]\innp{ v_i, \xast - \xk[i+1]} - a_i\oldepsilon_i}\right)  \\
             &\circled{2}[\leq] A_{k-1}\innp{v_k, x_{k+1} -x_k} - a_k R\norm{v_k}_\ast  + A_k \oldepsilon_k \\
             &\circled{3}[\leq] A_{k-1} \innp{\hat{v}_{k}, x_{k+1} - x_k} + A_{k-1}\norm{v_k - \hat{v}_k}_\ast\norm{x_{k+1} - x_k} \\
             & \quad + \frac{A_{k-1}\lambda^{1 /(r-1)}\norm{v_k}_\ast^{\dualnumber{r}}}{2 \cdot 2^{1/(r-1)}} + \frac{2 R^r a_k^r}{r\lambda A_{k-1}^{r-1}}\left(\frac{2}{\dualnumber{r}}\right)^{r-1} + A_{k} \oldepsilon_k \\
             &\circled{4}[\leq] - \frac{A_{k-1}}{\lambda}\left(\frac{1}{2}-\frac{3\sigma}{2}-\frac{A_k}{A_{k-1}}\sigma'\right)\norm{x_{k+1}-x_{k}}^r  + \bigopl{r}{\frac{R^r a_k^r}{\lambda A_{k-1}^{r-1}}} \\
             &\circled{5}[=] \bigopl{r}{\frac{R^r a_k^r}{\lambda A_{k-1}^{r-1}}}.
         \end{aligned}
    \end{align}
        Above, $\circled{1}$, substitutes the definition and cancels some terms, $\circled{2}$ uses the enlarged subgradient property of $v_k$ for the first term between $x_{k+1}$ and $x_k$, and uses Cauchy-Schwarz and the definition of $R$ on the second term. Then $\circled{3}$ adds and subtracts some terms to the first summand and applies Cauchy-Schwarz to make terms that we can bound by the oracle condition, appear, and we apply Young's inequality to the second summand so we will be able to cancel the term depending on $\norm{v_k}_\ast$ in which we add and subtract $\hat{v}_k$, apply the triangular inequality and the means inequality $(a+b)^{\dualnumber{r}} \leq 2^{(\dualnumber{r}-1)}( a^{\dualnumber{r}}+b^{\dualnumber{r}})$, so in $\circled{4}$ we use \cref{eq:inexact_prox_oracle_properties,eq:property_subgrads_of_norm_to_the_r} to these terms and also the first summands. Finally by the assumption on $\sigma, \sigma'$, in $\circled{5}$ we drop the first term.
        Adding \cref{eq:bound_unacc_inexact_ppa} up for $k \in [T]$, using $A_0 = 0$, reorganizing and recalling that $G_T$ is a primal-dual gap, we obtain, when we choose $a_k = \Theta_{r}(k^{r-1})$, and thus $A_k = \Theta_{r}(k^{r})$:
          \[
              f(\xk[T+1]) - f(\xast)  \leq G_T \leq \bigopl{r}{\frac{1}{A_T}\sum_{k=1}^T  \frac{a_k^{r} R^{r}}{\lambda A_{k-1}^{r-1}}} = \bigopl{r}{\frac{R^{r}}{\lambda T^{r-1}}} = \bigopl{q+\nu}{\frac{\L R^{q+\nu}}{T^{q+\nu-1}}}.
          \] 
          Note that the term appearing in the condition regarding $\sigma'$ is $\frac{A_k}{A_{k-1}} = \Theta((1+ \frac{1}{k})^r)  = \bigop{r}{1}$.
    \end{proof}

    \begin{remark}[Unaccelerated Ball Optimization Oracle Analysis]
        Let $\newtarget{def:indicator_function}{\indicator{X}}(x)$ be the indicator function of a set $\X$, that is $0$ if $x \in \X$ and $+\infty$ otherwise.
    We note that for a closed convex set $\X$ and a function $f$ with minimizer $\xast$ when constrained to $\X$, we converge with linear rates if we implement
    \begin{equation*}
        x_{k+1} \in \argmin_{x\in\X}\left\{f(x) + \frac{1}{2\lambda_k}\norm{x_k-x}^2\right\},
    \end{equation*}
    provided that we have the guarantee that at each iteration either $\norm{x_{k+1} - x_k} \geq \rho$ or we find a minimizer. Indeed, let $v_k \in \partial (f+\indicator{X})(x_{k+1})$ such that $\norm{g_k}_\ast = \frac{1}{\lambda}\norm{x_k-x_{k+1}}$, cf. \cref{def:moreau_env_and_prox}, and \cref{properties_of_M:5}. As above define $R \defi \max_{k\in[T]} \norm{x_i-\xast}$ or as an upper bound of it. For instance, if $\X$ is compact, we can use its diameter, or we can choose $\X = B_{\norm{\cdot}_{\p}}(x_0, \bigo{\norm{x_0- \xast}_{\p}})$ or the diameter of the sublevel set of the function at $x_0$, since this method decreases the function value. 

    Denote $\M[k] \defi \M[\lambda_k]$ the non-Euclidean Moreau envelope with parameter $\lambda_k$, and define $U_k \defi \M[k+1](x_{k+1})$ and the lower bound $L_k$ on $f(\xast)$ as $A_k L_k \defi \sum_{i=1}^k \ak[i] \M[i](\xk[i]) + \sum_{i=1}^k \ak[i]\innp{ g_i, \xast - \xk[i]} \leq A_k f(\xast)$. Recall $A_k = A_{k-1} + a_k = \sum_{i=1}^k a_k$, for $a_k>0$ to be determined later, and let $G_k \defi U_k - L_k$. If we choose $a_k = A_k \norm{x_k-x_{k+1}}/(2R)$, we have, for all $k \geq 1$ (note $A_0 = 0$):
    \begin{align}\label{eq:bounding_the_discr_error_in_unacc_ppa_BOO}
         \begin{aligned}
             A_{k} &G_{k} - A_{k-1}G_{k-1} \circled{1}[=]A_{k-1}(\M[k+1](\xk[k+1])-\M[k](\xk[k])) + \ak \M[k+1](\xk[k+1]) \\ \\
             & \quad -\ak[k] \M[k](\xk[k]) \Ccancel[red]{- \sum_{i=1}^{k-1} \ak[i] \M[i](\xk[i])} - \Ccancel[blue]{\sum_{i=1}^{k-1} \ak[i]\innp{ g_i, \xast - \xk[i]}} - \ak[k]\innp{ g_k, \xast - \xk[k]} \\
             & \quad \Ccancel[red]{+ \sum_{i=1}^{k-1} \ak[i] \M[i](\xk[i])} + \Ccancel[blue]{\sum_{i=1}^{k-1} \ak[i]\innp{ g_i, \xast - \xk[i]}}  \\
             &\circled{2}[\leq] -\frac{A_k}{2\lambda}\norm{x_k - x_{k+1}}^2 + \frac{a_k}{\lambda} \norm{x_k - x_{k+1}}R.  \\
             &\circled{3}[\leq]  0.
         \end{aligned}
    \end{align}
        Above, $\circled{1}$ just uses the definitions and cancels some terms, and $\circled{2}$ groups some terms, uses the descent \cref{def:moreau_env_and_prox}, \cref{properties_of_M:6}, H\"older's inequality along with the definition of $R$, and $\norm{g_k}_\ast = \frac{1}{\lambda}\norm{x_k-x_{k+1}}$. In $\circled{3}$ we used the value of $a_k$.

If we solve the equation $a_k = (A_{k-1}+a_k) \norm{x_k-x_{k+1}}/(4R)$, we obtain $a_k = A_{k-1} (\frac{4R}{\norm{x_k-x_{k+1}}}-1)^{-1}$ and so $A_k = A_{k-1}+a_k = A_{k-1}\left( \frac{1}{1-\norm{x_k-x_{k-1}}/(4R)}\right) \geq A_{k-1} (\frac{1}{1-\rho/(4R)}) \geq A_1 (\frac{1}{1-\rho/(4R)})^{k-1} \geq A_1\exp((k-1)\frac{\rho}{4R})$, where we used the lower bound that is guaranteed on the distance traveled from one point to the next one. %
    Hence, adding up we conclude:
        \[
            f(\xk[T+2]) - f(\xast) \leq \M[T+1](\xk[T+1]) - f(\xast) \leq G_T \leq \frac{A_1G_1}{A_T} \leq G_1 \exp\left(-(k-1)\frac{\rho}{4R} \right). 
        \] 
        So we obtain an $\epsilon$-minimizer is $\bigotilde{\frac{R}{\rho}\ln(\frac{G_1}{\epsilon})}$ iterations.
    \end{remark}

\section{Proofs from \nameref{sec:lower_bounds}: Lower Bounds}\label{app:lower_bound_proofs}

\begin{proof}\linkofproof{lemma:randomized_smoothing}
\begin{enumerate}
    \item Let $\X= \ballnorm[\beta]$. 
        The Lipschitzness of $\smoothing_\beta[f]$ is a direct consequence of the smoothing as an averaging and $f$ being $G$-Lipschitz. For the smoothness, we first note that we have $\nabla \smoothing_\beta[f](x) = \frac{\vol(\partial\X)}{\vol(\X)}\E[f(x+ v) w_v][v\sim\nu_{\partial \X}]$, where $w_v$ is defined as an outward $\|\cdot\|_2$ unit vector normal to $\partial \X$, that is, $w_v \in \partial (\norm{\cdot})(v)$ is a subgradient of the norm at $v$. 
        By \cref{properties_of_M:5} of \cref{prop:properties_of_M} with $x \gets 0$, $y \gets v$, $\lambda \gets 1$, and taking into account that since $w_v$ is normal to $\partial X$, we have  $ w_v \propto g \in \subdiffsqnorm{0}(v) = \partial(\frac{1}{2}\norm{\cdot}^2)(v)$, and $ \innp{v, w_v} = \norm{v}\norm{w_v}_\ast = \beta \norm{w_v}_\ast$.
Thus, by the divergence theorem  on the identity function $\phi(v) = v$ and on $\X$:
    \begin{align}\label{eq:conclusion_div_thm}
     \begin{aligned}
         d \vol(\X) &= \int_{\X} \sum_{i=1}^d \frac{\partial \phi(v)}{\partial v_i}\d \nu_{\X}(v) =  \int_{\partial\X} \innp{v, w_v}\d \nu_{\partial X}(v)
         = \vol(\partial\X) \beta \E[\norm{w_v}_{\ast}][v\sim\nu_{\partial \X}].
     \end{aligned}
\end{align}

        Finally, using that $f$ is $G$-Lipschitz with respect to $\norm{\cdot}$, we obtain
\begin{align*}
     \begin{aligned}
         \norm{\nabla \smoothing_\beta[f](x) - \nabla \smoothing_\beta[f](y)}_{\ast} &= \frac{\vol(\partial\X)}{\vol(\X)}\norm{\E[f(x+ v)w_v- f(y+v)w_v][v\sim\nu_{\partial \X}]}_{\ast} \\ 
         &\leq  \frac{\vol(\partial\X)}{\vol(\X)}\E[\abs{f(x+v)- f(y+v)}\norm{w_v}_{\ast}][v\sim\nu_{\partial \X}] \\
         & \leq G\norm{x-y}\frac{\vol(\partial\X)}{\vol(\X)}\E[\norm{w_v}_{\ast}][v\sim\nu_{\partial \X}] \\
         & = \frac{Gd}{\beta}\norm{x-y},
     \end{aligned}
\end{align*}
        where the last equality is due to \cref{eq:conclusion_div_thm}.

\item It can be argued by induction on $\q$ similarly to \citep[Corollary 2.4]{agarwal2018lower} but using the previous part.
    We have the statement for $\q=0$ since the Lipschitzness of a function is preserved after smoothing. Let $v_1, \dots, v_{\q}$  be arbitrary unit vectors with respect to $\norm{\cdot}$, and let $G_{i} =\frac{d^{i} 2^{i(i+1)/2}}{\beta^{i}} G $. If the result holds for $\q-1$, we have that $\smoothing_{\beta/2^{\q}}\nabla^{\q-1} \S{\q-1}{\beta} [f](x)[v_1, \dots, v_{\q-1}]$ is differentiable and its differential is \\ 
    $\nabla^{\q}\S{\q}{\beta}[f](x)[v_1, \dots, v_{\q-1}]$, by commutativity of the smoothing and differential operator, hence by the first part it is Lipschitz w.r.t $\norm{\cdot}$ with constant $\frac{d 2^{\q}}{\beta}G_{\q-1}=G_{\q}$. Similarly, for $i < \q$, by the commutativity of the operators again, we have that $\nabla^i \S{\q}{\beta}[f](x)[v_1, \dots, v_i] = \smoothing_{\beta/2^{\q}} \nabla^i \S{\q-1}{\beta}[f](x)[v_1, \dots, v_i]$, and we know that the right hand side is $G_{i}$ Lipschitz by induction hypothesis and the fact that $\smoothing_{\beta/2^{\q}}$ preserves the Lipschitzness.
\item By Lipschitzness of $f$, %
$\abs{\S{\q}{\beta}[f](x) - f(x)} \leq \max_{x\in\X} \norm{x} G=\beta G$.

\item This is a direct consequence of the convexity of $f$ and the smoothing as an averaging. 

\item By expanding the expectations in the definition of \(\S{\q}{\beta}\), we get that \( \S{\q}{\beta}[f](x) = \mathbb{E}_{y \sim \mu_x}[f](y) \) where \( \mu_x \) is a distribution supported in \( B^{\norm{\cdot}}_{(1 - 2^{-\q}) \beta}(x) \).
\end{enumerate}
\end{proof}

\begin{remark}
    For \(\p\)-norm balls \(\X \defi \ballnorm[\beta][\norm{\cdot}_{\p}]\) with \(\p \in [1, \infty)\), we previously established in part \(1\) of \cref{lemma:randomized_smoothing} that $\frac{\vol(\partial\X)}{\vol(\X)} \E[\norm{w_v}_{\dualnumber{p}}][v \sim \nu(\partial \X)] = \beta^{-1} d$. However, the ratio \(\frac{\vol(\partial\X)}{\vol(\X)}\) behaves differently depending on \(\p\). Specifically, it holds that $\frac{\vol(\partial\X)}{\vol(\X)} = \bigop{\p}{\beta^{-1} d^{1/2 + 1/\p}}$, as shown in \citep[Lemma 22]{wang2019selective}. In contrast, for \(\p = \infty\), we find \(\frac{\vol(\partial\X)}{\vol(\X)} = \beta^{-1} d\). This discrepancy reveals a phase transition in the behavior of \(\frac{\vol(\partial\X)}{\vol(\X)}\) across \(\p\), even though the product \(\frac{\vol(\partial\X)}{\vol(\X)} \E[\norm{w_v}_{\dualnumber{p}}][v \sim \nu(\partial \X)]\) remains constant at \(\beta^{-1} d\) for all \(\p\).
\end{remark}

\begin{proof}\linkofproof{lemma:softmax}
Let us denote $r_j = \exp(\langle a^j, x\rangle / \mu)$ for simplicity. \\
$(a)$ We follow the idea in \citep[Example 5.15]{Beck:2017}. The derivatives of $\smax{\mu}(Ax)$ is
    $$\frac{\partial \smax{\mu}(Ax)}{\partial x_i} = \frac{1}{\sum_{j=1}^d r_j}\sum_{\ell =1}^d r_\ell a^\ell_i $$
    We can conclude the 1-Lipschitzness by looking at the norm of the gradient:
\begin{align*}
\|\nabla \smax{\mu}(Ax)&\|_\ast = \sup_{\norm{h}\leq 1} \langle \nabla \smax{\mu}(Ax), h\rangle \\
&= \frac{1}{\sum_{j=1}^d r_j}\sup_{\norm{h}\leq 1} \left\vert \sum_{i=1}^d\sum_{\ell=1}^d r_\ell  a^\ell_i h_i\right\vert \\ 
&\leq \frac{1}{\sum_{j=1}^d r_j} \sum_{\ell=1}^d r_\ell \sup_{\norm{h}\leq 1} \norm{a^\ell}_\ast \norm{h} \leq 1.
\end{align*}
$(b)$ Here we generalize the ideas in \citep[Theorem 5]{bullins2020highly}. Let $f(x) = \mu\log(x)$ and $Z_\mu(x) = \sum_{j=1}^d r_j$. Then, $\smax{\mu}(x) = f(Z_\mu(Ax))$. Moreover, it is easy to check that for every $k\geq 1$
 $$f^{(k)}(x) = \mu\frac{(-1)^{k-1}(k-1)!}{x^k}, \qquad \nabla^k Z_\mu(Ax)[h_1,\dots,h_{k}] = \frac{1}{\mu^k}\sum_{j=1}^d r_j \prod_{\ell = 1}^{k} \langle a^j,h_\ell\rangle.$$
Fix unitary vectors $h_1,\dots,h_{\q+1}\in \Rd$. For any subset $B=\{i_1,\dots,i_{|B|}\}\subseteq [\q+1]$, let us denote $\mathbf{h}_B = [h_{i_1},\dots,h_{i_{|B|}}]$. Then, because of the chain rule and Faà di Bruno's formula we have
\begin{align*}
    |\nabla^{(\q+1)}\smax{\mu}(x)[\mathbf{h}_{[\q+1]}]| &= \left\vert \sum_{\pi\in \Pi_{(\q+1)}}f^{|\pi|}(Z_\mu(Ax))\cdot \prod_{B\in \pi}\nabla^{|B|}Z_\mu(Ax)[\mathbf{h}_B]\right\vert \\
    &= \left\vert \sum_{\pi\in \Pi_{(\q+1)}} \mu\frac{(-1)^{|\pi|-1}(|\pi|-1)!}{Z_\mu(Ax)^{|\pi|}} \cdot \prod_{B\in \pi}\frac{1}{\mu^{|B|}}\sum_{j=1}^d r_j \prod_{\ell\in B} \langle a^j,h_\ell\rangle\right\vert \\
    &\leq  \sum_{\pi\in \Pi_{(\q+1)}} \left\vert\mu\frac{(-1)^{|\pi|-1}(|\pi|-1)!}{Z_\mu(Ax)^{|\pi|}}\right\vert \cdot \prod_{B\in \pi}\frac{1}{\mu^{|B|}}\sum_{j=1}^d r_j \prod_{\ell\in B}\norm{a^j}_\ast \norm{h_\ell} \\
    &\leq  \sum_{\pi\in \Pi_{(\q+1)}} \mu\frac{(|\pi|-1)!}{Z_\mu(Ax)^{|\pi|}} \cdot \prod_{B\in \pi}\frac{1}{\mu^{|B|}}Z_\mu(Ax) \\
    &=  \sum_{\pi\in \Pi_{(\q+1)}} \mu\frac{(|\pi|-1)!}{\cancel{Z_\mu(Ax)^{|\pi|}}} \cdot \frac{1}{\mu^{(\q+1)}}\cancel{Z_\mu(Ax)^{|\pi|}} \norm{h}^{(\q+1)} \\
    &=  \sum_{\pi\in \Pi_(\q+1)} \mu^{-\q}(|\pi|-1)!\norm{h}^{(\q+1)} \\
    &\leq \left\vert \Pi_{(\q+1)}\right\vert \mu^{-\q}\q!\norm{h}^{(\q+1)}.
\end{align*}
Therefore, $\norm{\nabla^{(\q+1)}\smax{\mu}(x)}_\ast \leq L_{\q}\:=\left\vert \Pi_{(\q+1)}\right\vert \mu^{-\q}\q!$. In particular, $\left\vert \Pi_{(\q+1)}\right\vert$ is the $(\q+1)$-th Bell number that can be bounded as $\left\vert \Pi_{(\q+1)}\right\vert\leq \left(\frac{\q+1}{\ln(\q+2)}\right)^{(\q+1)}$. Finally, the Lipschitzness of $\nabla^{\q}\smax{\mu}(x)$ comes from a standard mean value argument. \\
$(c)$ We now generalize the result in \citep[Lemma 3]{garg2021nearoptimal}.  Let $c = \frac{\sum_{j=n+1}^d  r_j}{\sum_{j = 1}^n r_j}$. We have
\begin{align*}
   \norm{\nabla & \smax{\mu}(Ax) - \nabla \smaxn{n}{\mu}(Ax)}_\ast \\
    &= \sup_{\norm{h}\leq 1}\frac{1}{\sum_{j=1}^d r_j}\left\vert \sum_{\ell = 1}^d r_\ell\langle a^\ell,h\rangle - \frac{1}{\sum_{j=1}^n r_j} \sum_{\ell = 1}^n r_\ell\langle a^\ell,h\rangle \right\vert\\
    &= \sup_{\norm{h}\leq 1}\frac{1}{\sum_{j=1}^d r_j}\left\vert \sum_{\ell = 1}^d r_\ell\langle a^\ell,h\rangle - \frac{1+c}{\sum_{j=1}^d r_j} \sum_{\ell = 1}^n r_\ell\langle a^\ell,h\rangle \right\vert\\
    &= \sup_{\norm{h}\leq 1}\frac{1}{\sum_{j=1}^d r_j} \left\vert\sum_{\ell = n+1}^d r_\ell\langle a^\ell,h\rangle -  c\sum_{\ell = 1}^n r_\ell\langle a^\ell,h\rangle \right\vert\\
    &\leq \sup_{\norm{h}\leq 1}\frac{1}{\sum_{j=1}^d r_j} \sum_{\ell = n+1}^d r_\ell\norm{a^\ell}_\ast\norm{h} +  c\sum_{\ell = 1}^n r_\ell\norm{a^\ell}_\ast\norm{h} \\
    & \leq \frac{1}{\sum_{j=1}^d r_j}\left(\sum_{\ell=n+1}^d r_\ell + c\sum_{\ell=1}^n r_\ell\right) = \frac{2\sum_{\ell=n+1}^d r_\ell}{\sum_{j=1}^d r_j}.
\end{align*}
On the other hand, $\smax{\mu}(Ax) - \smaxn{n}{\mu}(Ax) = \delta$ implies 
$$\delta = \ln\left(
\frac{\sum_{j=1}^d r_j}
{\sum_{j=1}^n r_j}
\right) 
= \ln\left(1 + \frac{\sum_{j=n+1}^d r_j}{\sum_{j=1}^n r_j}\right)=\ln(1+c)\geq \frac{c}{2}.$$
Hence, $\sum_{j=n+1}^d r_j \leq 2\delta\sum_{j=1}^d r_j$ and the conclusion follows.

\end{proof}

\begin{proof}\linkofproof{lemma:Lq}
    Each $f_i$ is an instance of partial softmax composed with a linear map and a translation. Therefore, the high-order Lipschitzness and convexity properties of $\smax{\mu}$ in \cref{lemma:softmax} also apply to $f_i$, and thus $f_i$ is convex, $\q$-times differentiable with $\bigop{\q}{\mu^{-\q}}$-Lipschitz $\q$-th derivatives. The function $h$ is also convex, since it is a maximum of convex functions. Because of \cref{lemma:randomized_smoothing}.\ref{item:cvxty_randomized_smoothing} the function $g$ is also convex.

    Let $x \in \Rd$. Let $j \in[T]$ be the minimum number such that there is a point $\omega \in B^{\norm{\cdot}}_\beta(x)$ for which $h(\omega)=f_j(\omega)$. For every $z\in B^{\norm{\cdot}}_\beta(x)$, $h(z)=f_j(z)+\max _{i\geq j}\left\{f_i(z)-f_j(z)\right\}$. The term $f_j(z)$ is smooth in the ball whereas the term $\max _{i\geq j}\left\{f_i(z)-f_j(z)\right\}$ may not be smooth. If all points $z \in B^{\norm{\cdot}}_\beta(x)$ satisfy $h(z)=f_j(z)$, then the nonsmooth term is $0$ and so $h$ is as smooth as $f_j$, and the $i$-th derivative of $g=\S{\q}{\beta}[h]$ enjoys the same Lipschitzness as the $i$-th derivative of $f_j$ by \cref{lemma:randomized_smoothing}.

    We show that the nonsmooth term has a small Lipschitz constant in $B^{\norm{\cdot}}_\beta(x)$, which will be later used in conjunction with \cref{lemma:randomized_smoothing} to conclude. We can now assume that the non-smooth term is nonzero at some point in $B^{\norm{\cdot}}_\beta(x)$. %
    Towards this let $x'\in B^{\norm{\cdot}}_\beta(x)$, and $I(x')=\left\{i \in [T] \mid h\left(x'\right)=f_i\left(x'\right)\right\}$. The set of subgradients of the nonsmooth term at $x'$ is the convex hull of $\left\{\nabla\left(f_i-f_j\right)\left(x'\right)\right\}_{i \in I\left(x'\right)}$. So if we show that for an arbitrary $i \in I\left(x'\right),\norm{\nabla\left(f_i-f_j\right)\left(x'\right)}_{\ast} \leq \L$, then we know that the nonsmooth part is $\L$-Lipschitz at $x'$. If $i=j$, then the gradient is zero. Let us take an $i \neq j$ (since $j$ is the smallest, in fact $i>j$). 
    By convexity of the ball and the continuity of $f_i$ and $f_j$, there must be a point $y$ in $B^{\norm{\cdot}}_\beta(x)$ for which $h(y)=f_i(y)=f_j(y)$. Note that $x' \in B^{\norm{\cdot}}_{2 \beta}(y)$.
    The statement $f_i(y)=f_j(y)$ translates to $\circled{1}$ below
\begin{align*}
     \begin{aligned}
         (i-j) d^{-\alpha} &\circled{1}[=] \frac{\smaxn{i}{\mu}((\langle z_\ell, y\rangle + (T-\ell)\gamma)_{\ell\in [d]})-\smaxn{j}{\mu}((\langle z_\ell, y\rangle + (T-\ell)\gamma)_{\ell\in [d]})}{\mu} \\
         & = \ln \left(\frac{\sum_{\ell=1}^i \exp \left(\frac{\langle z_\ell, y\rangle + (T-\ell) \gamma}{\mu}\right)}{\sum_{\ell=1}^j \exp \left(\frac{\langle z_\ell, y\rangle +(T-\ell) \gamma}{\mu}\right)}\right)  =\ln \left(1+\frac{\sum_{\ell=j+1}^i \exp \left(\frac{\langle z_\ell, y\rangle+(T-\ell) \gamma}{\mu}\right)}{\sum_{\ell=1}^j \exp \left(\frac{ \langle z_\ell, y\rangle+(T-\ell) \gamma}{\mu}\right)}\right) \\
         &\circled{2}[\geq] e^{-4\beta/\mu}\ln \left(1+\frac{e^{2 \beta / \mu} \sum_{\ell=j+1}^i \exp \left(\frac{\langle z_\ell, y\rangle+(T-\ell) \gamma}{\mu}\right)}{e^{-2 \beta / \mu} \sum_{\ell=1}^j \exp \left(\frac{\langle z_\ell, y\rangle+(T-\ell) \gamma}{\mu}\right)}\right) \\
         & \circled{3}[\geq] e^{-4\beta/\mu}\frac{\smaxn{i}{\mu}((\langle z_\ell, x'\rangle + (T-\ell)\gamma)_{\ell\in [d]})-\smaxn{j}{\mu}\left((\langle z_\ell, x'\rangle + (T-\ell)\gamma)_{\ell\in [d]}\right)}{\mu}\\
         & \circled{4}[=] e^{-4\beta/\mu}\left(f_i(x') -f_j(x') + (i-j)d^{-\alpha} \right),
     \end{aligned}
\end{align*}
    where  $\circled{2}$ holds since for all $c>0$, it is $\ln(1+c) \geq e^{-4\beta/\mu}\ln(1+e^{4\beta/\mu}c)$ and $\circled{3}$ is due to $\abs{x'_\ell-y_\ell} \leq 2 \beta$ which for any $\ell$ is implied by the fact that $\norm{x'-y}_{\p} \leq 2 \beta$. Finally $\circled{4}$ holds by the definition of $f_i$ and $f_j$. Therefore, by \cref{lemma:softmax} (c) 
    $$\norm{\nabla\left(f_i-f_j\right)\left(x'\right)}_{\ast} \leq 4(i- j) d^{-\alpha} e^{4 \beta / \mu} \leq 4T d^{-\alpha} e^{4 \beta / \mu}.$$ 
     
     The $\q$-th derivatives of $g=\S{\q}{\beta}[h]=\S{\q}{\beta}\left[f_j\right]+\S{\q}{\beta}\left[\max _{i\geq j}\left\{f_i-f_j\right\}\right]$ are thus the sum of two Lipschitz functions with constants $\bigop{\q}{\mu^{-\q}}$ and $\bigop{\q}{\beta^{-\q} T d^{\q-\alpha} \exp (4 \beta / \mu)}$ respectively, where the last is a consequence of \cref{lemma:randomized_smoothing}. Finally, we use the values of the parameters $\gamma = \frac{\Theta}{4T}$, $\mu = \frac{\gamma}{4\alpha\ln d}$, and $\beta = \frac{\gamma}{\ln d}$ to bound both quantities:
     $$\mu^{\q} = \left(\frac{\gamma}{4\alpha \ln d}\right)^{-\q} = \left(\frac{\Theta}{16T\alpha \ln d}\right)^{-\q} \leq\bigopl{\q}{\left(\frac{T\ln d}{\Theta}\right)^{\q}},$$
     and 
     $$\beta^{-\q} T d^{\q-\alpha} \exp (4 \beta / \mu) = \left(\frac{\gamma}{\ln d}\right)^{-\q} \frac{T}{d^{\alpha-\q}}\exp(16\alpha) \leq \bigopl{\q}{\left(\frac{T\ln d}{\Theta}\right)^{\q}},$$
     where the last inequality holds because $\alpha\ge q+1$ and $T\le d$. 
\end{proof}

\begin{proof}\linkofproof{thm:lower-bound}
We start by estimating the optimality gap of the function $g$.  We start by establishing an upper bound for $\inf_{x\in \mathcal{X}}g(x)$. Initially, we assume that \( \mathcal{X} \subseteq \mathbb{R}^d \) is a closed convex set containing the unit ball \( B^{\norm{\cdot}} \) of \( (\mathbb{R}^d, \norm{\cdot}) \). 

For every $i\in [T]$ we have the upper bound
\begin{align}
          f_i(x) &\leq \min_{x\in \mathcal{X}} \mu \ln\left(\sum_{j=1}^i \exp\left(\frac{\langle z_j,x\rangle + T\gamma}{\mu}\right)\right) + \mu (T+1) d^{-\alpha} \nonumber \\
         &\leq \mu \ln\left(T\exp\left(\frac{\max_{j\in [T]}\langle z_j,x\rangle + T\gamma}{\mu}\right)\right) + \mu (T+1) d^{-\alpha}  \nonumber \\
         &\leq \mu\ln T + \max_{j\in [T]}\langle z_j,x\rangle + T\gamma + \mu (T+1) d^{-\alpha} \label{eq:bound_fi},
\end{align}

Therefore $h(x) \leq \mu\ln T + \max_{j\in [T]}\langle z_j,x\rangle + T\gamma + \mu (T+1) d^{-\alpha}$ for every $x\in \mathcal{X}$. Moreover, using the properties of the randomized smoothing we have that for every $x$, $g(x)\leq h(x)+2\beta$. In particular, 
$$\inf_{x\in \mathcal{X}}g(x) \leq \inf_{x\in \mathcal{X}}h(x)+2\beta \leq  \mu\ln T -\Theta + T\gamma + \mu (T+1) d^{-\alpha} + 2\beta,$$
where we have used the hypothesis $(ii)$ of \cref{thm:lower-bound}. 

Now, we compute a lower bound for the algorithm's output. For this, we consider the construction of the hard instance functions $g$ with vectors of the form $z_i = \xi_iv_i$ where $\xi\in \{-1,1\}^{d}$ is a vector of signs and $\{v_i\}_{i\in [d]}$ are orthogonal vectors in $\Rd$. Given an algorithm $\mathcal{A}$ interacting with a local oracle $\mathcal{O}$, denote $x_0, x_1, \dots, x_{T-1}$ the first $T$ query points. The key of the construction is to choice $\xi_i$ such that they only depend on $\{x_0, \dots, x_i\}$. In particular, our sign choices are based on  inductively defined sets $I_i=\{i_j\}_{j=0}^{i}\subseteq [d]$, as follows. First, $I_{-1}=\emptyset$, and given $I_{i-1}$, let $I_i=I_{i-1}\cup\{\sigma(i)\}$, where $\sigma(i) \in \argmax_{j\in[d]\setminus I_i} \abs{\langle v_j, x_{i}\rangle}$, and we let $\xi_i = \sign(\langle v_{\sigma(i)},x_{i}\rangle)$. Hence for every $t\in [T]$
$$g(x_t) \circled{1}[\ge] h(x_t) - 2\beta \geq f_t(x_t)-2\beta \circled{2}[\ge] \xi_t\langle v_{\sigma(t)},x_t\rangle -2\beta \circled{3}[\ge] -2\beta.$$
Here, $\circled{1}$ uses the properties of the randomized smoothing, in $\circled{2}$ we drop every term in the softmax except the last one, and $\circled{3}$ is because of the choice of $\xi_t$. 

    Function $g$ is $\q$-th order smooth with constant $L_{\q}\leq \bigotildep{\q}{(T/\Theta)^{\q}}$. By rescaling, we can construct the function $F = (\L/L_{\q}) g$ that is $\q$-th order smooth with constant $\L$ and the optimality gap for every $t\in [T]$ is
$$F(x_t) - \inf_{x\in \mathcal{X}}F(x) \geq \frac{\L}{L_{\q}}\left(-\mu\ln T +\Theta - T\gamma - \mu (T+1) d^{-\alpha}-4\beta\right) \geq \bigomegatildelp{q}{\L\frac{\Theta^{\q+1}}{T^{\q}(\ln d)^{\q}}}.$$
It remains to prove that for any $y_t \in \ballnorm[\beta](x_t)$, we have that $h(y_t)$ does not depend on $\xi_i$, for $i>t$. That is, $f_{t+1}(y_t) \geq f_{i+1}(y_t)$. Assume for simplicity from now on by relabeling the coordinates without loss of generality that the index at the $\ell$-th step is $\ell$, that is $i_{\ell-1} = \ell$. Inequality $f_{t+1}(y_t) \geq f_{i+1}(y_t)$ holds if the following expression is $\leq (i-t)d^{-\alpha}$
\begin{align*}
     \begin{aligned}
         \ln\left( \frac{\sum_{j=1}^{i+1} \exp(\frac{\xi_j\langle v_j,  y_{t}\rangle +  (T-j)\gamma}{\mu})}{\sum_{j=1}^{t+1} \exp(\frac{\xi_j\langle v_j,  y_{t}\rangle +  (T-j)\gamma}{\mu})} \right) & \circled{1}[\leq]  \frac{\sum_{j=t+2}^{i+1} \exp(\frac{\xi_j\langle v_j,  y_{t}\rangle +  (T-j)\gamma}{\mu})}{\sum_{j=1}^{t+1} \exp(\frac{\xi_j\langle v_j,  y_{t}\rangle +  (T-j)\gamma}{\mu})} \\
         & \circled{2}[\leq] \frac{T{\displaystyle\max_{t+2 \leq j \leq i+1}}\exp(\frac{\xi_j\langle v_j,  y_{t}\rangle +  (T-j)\gamma}{\mu})}{\exp(\frac{\xi_{t+1}\langle v_{t+1},  y_{t}\rangle +  (T-t-1)\gamma}{\mu})}
     \end{aligned}
\end{align*}
where $\circled{1}$ uses $\ln(1+c) \leq c$, while in $\circled{2}$ we drop all summands in the denominator but the last one and bounded the sum by a $\max$ and we bound $j$ by $t+2$ in the $\exp$ in the numerator.
It suffices to prove that the right hand side is upper bounded by $d^{-\alpha} \leq (i-t)d^{-\alpha}$. Equivalently, it suffices to show
\[
    \mu \ln T +  \max_{t+2 \leq j \leq i+1} \xi_j\langle v_j,  y_{t}\rangle - \xi_{t+1}\langle v_{t+1},  y_{t}\rangle + \gamma \leq -\mu\alpha \ln d.
 \] 
By the definition of $i_t = t+1$, we have $\xi_{i+1}\langle v_{i+1}, x_{t}\rangle - \xi_{t+1} \langle v_t, x_{t} \rangle \leq 0$ for any $i > t$. Thus, we have $\xi_{i+1}\langle v_{i+1}, y_{t}\rangle - \xi_{t+1} \langle v_t,y_{t} \rangle \leq 2\beta$. So it suffices that
\[
    \mu (\ln T + \alpha \ln d) + 2\beta \leq  \gamma,
\] 
which holds by construction.

\paragraph{Extension to $R$-balls} \label{remark:ballR}
To extend the results for a set containing a ball of radius $R>0$, it is enough to use the construction above with the function $\hat F(x)\defi R^{\q+1}F(x/R)$ acting over the set $\hat{\mathcal{X}}\defi R\mathcal{X}$. Clearly, 
if $B^{\norm{\cdot}} \subseteq \mathcal{X}$, then $B^{\norm{\cdot}}_R \subseteq \hat{\mathcal{X}}$. Moreover, using the chain rule and the fact that $F$ is $\q$-th order $\L$-Lipschitz, it is easy to verify that $\hat{F}$ is also $\q$-th order $\L$-Lipschitz, and for every $t\in [T]$
$$\hat{F}(Rx_t)-\inf_{x\in \hat{\mathcal{X}}}\hat{F}(x) = R^{q+1}\left(F(x_t)-\inf_{x\in \mathcal{X}}F(x)\right) \geq \bigomegatildelp{q}{\L R^{\q+1}\frac{\Theta^{\q+1}}{T^{\q}(\ln d)^{\q}}}.$$
Moreover, by applying a simple translation, \(\mathcal{X}\) can be centered at the origin. This enables us to shift and scale any full-dimensional convex body to ensure it encloses \(B^{\|\cdot\|}\).

\paragraph{Extension to H\"older continuous functions}

Now we extend the result to H\"older continuous functions. Let $g$ constructed as in the past sections. We have that $g$ is $q$-times differentiable and its derivatives are $L_q$ Lipschitz from \cref{lemma:Lq}, i.e.
$$\norm{\nabla^{\q} g(x)-\nabla^{\q} g(y)}_{\ast} \leq L_{\q}\norm{x-y}.$$
Moreover, since $\nabla^{\q-1}g(x)$ is $L_{\q-1}$-Lipschitz, a standard mean value argument implies that the $q-$th order derivatives are bounded, i.e.
$$\norm{\nabla^{q} g(x)}_{\ast} \leq L_{\q-1} \quad \mbox{for every $x\in \mathcal{X}$}.$$ 
Hence for every $\nu\in (0,1]$
$$\norm{\nabla^{\q} g(x) - \nabla^{\q} g(y)}_{\ast} \leq (2L_{\q-1})^{1-\nu}L_{\q}^{\nu}\norm{x-y}_{\p}^{\nu}.$$
Therefore, $g$ is $(H_{\nu,\q},\nu)$-H\"older continuous with $H_{\q,\nu} \defi (2L_{\q-1})^{1-\nu}L_{\q}^{\nu}$. 

It follows from \cref{lemma:Lq}  that 
$$H_{\nu,\q} =  \bigotildep{\q}{\left(T/\Theta\right)^{(\q-1)(1-\nu)}\left(T/\Theta\right)^{\q\nu}} = \bigotildep{\q}{\left(T/\Theta\right)^{\q+\nu-1}}.$$

Given $H>0$, the rescaled function $F(x) = \frac{H}{H_{\nu,\q}}g(x)$ is $(H,\nu)$-H\"older continuous. Furthermore, we can extend the result to a set containing a $R$-ball by considering the function $\hat{F}(x) = R^{\q+\nu}F(x/R)$, which is also $(H,\nu)$-H\"older continuous leading to the optimality gap $\bigomegatildelp{\q}{HR^{\q+\nu}\Theta/H_{\nu,\q}} = \bigomegatildelp{\q}{HR^{\q+\nu}\Theta^{\q+\nu}/T^{\q+\nu-1}}$.

In particular,  recalling the specific value of $\Theta$ for $\p$-norms, i.e., $\Theta = T^{-1/\p}$ for $\p\geq 2$, $\Theta = 1$ for $p=\infty$, and $\Theta=T^{-1/2}$ for $1\leq \p<2$, we have that the number of iterations needed to reach the precision $\epsilon$ is at least 
$$\bigomegatildelp{\q,\p}{\left(\frac{HR^{\q+\nu}}{\epsilon}\right)^{\frac{\m}{(\m+1)(\q+\nu)-\m}}},$$
where $\m\defi \max\{2,\p\}$. For $\p=\infty$, we have the rate $\bigomegatildelp{\q,\p}{\left(\frac{HR^{\q+\nu}}{\epsilon}\right)^{\frac{1}{\q+\nu-1}}}$.

\end{proof}

\subsection{The case of \texorpdfstring{$p$}{p}-norms} 

In this section, we specialize \cref{thm:lower-bound} for the case of the $\p$-norms, following classical constructions of orthonormal bases from \cite{Nemirovski:1983} that we include for self-containedness. To this, we separate the cases $p\geq 2$ and $1\leq p\leq 2$. In particular, for $p\ge 2$ we prove that if $d\geq \bigomega{T^{1+1/p}}$ then we can take $\Theta = T^{-1/p}$. On the other hand, when $1\le p < 2$ we can take $\Theta = T^{-1/2}$ provided that $d\geq \bigomega{T^{3/2}}$. 

\begin{itemize}
    \item For $\p\geq 2$ we use $z_i = \xi_ie_i$ where $\xi\in \{-1,1\}^d$ is a vector of signs and $e_i$ is the $i$-th canonical vector. It is easy to check that
        $$\min_{x\in \mathcal{X}}\max_{i\in [T]}\langle z_i,x\rangle \leq \min_{\norm{x}_p \leq 1}\max_{i\in [T]}\xi_i\langle e_i,x\rangle \leq -T^{-1/\p}.$$
    Replacing $\Theta = T^{-1/\p}$ the optimality gap for $\p\geq 2$ is
    $$F(x_T) - \inf_{x\in \mathcal{X}}F(x) \geq \bigomegatildelp{q}{\L\frac{\Theta^{\q+1}}{T^{\q}(\ln d)^{\q}}} = \bigomegatildelp{\q,\p}{LT^{-\frac{p\q+\q+1}{\p}}},$$
    so at least $\bigomegatildelp{\q,\p}{\left(\frac{\L}{\epsilon}\right)^{\frac{\p}{p\q+\q+1}}}$ iterations are needed to reach the precision $\epsilon$. 

    Similarly, replacing $\Theta = 1$ for $\p=\infty$ we obtain the rate $\bigomegatildelp{\q,\p}{\left(\frac{\L}{\epsilon}\right)^{\frac{1}{\q}}}$. 

\item For $1\leq \p < 2$ we use a different construction. Assume that $d=2^{\bar s}$ with $\bar s\in \NN$ such that $2^{\bar s-1}<8T^{3/2}\leq 2^{\bar s}$, and consider the Hadamard base $\{\hat e_1,\ldots,\hat e_d\}$ formed by the columns of the matrix $H_{d}$ that is constructed recursively as $H_1 = H_{2^0}=[1]$, and
\[ H_{2^{s+1}} = \frac{1}{\sqrt 2} \left[ 
\begin{array}{cc}
H_{2^s} & H_{2^s}\\
H_{2^s} & - H_{2^s}
\end{array}
\right].
\]

It is easy to see that
\[ \|\hat e_j\|_2=1, \qquad \|\hat e_j\|_{\infty}=1/\sqrt{d}. \]
Using interpolation inequalities for $\p$ norms, we have that for all $j\in[d]$,
\[ \|\hat e_j\|_{\dualnumber{p}}\leq \|\hat e_j\|_2^{\frac{2}{\dualnumber{p}}}\|\hat e_j\|_{\infty}^{1-\frac{2}{\dualnumber{p}}}=d^{-\frac12(1-\frac{2}{\dualnumber{p}})}= d^{\frac{1}{\dualnumber{p}}-\frac12}. \]
In particular, we have that $\{v_j\}_{j\in[d]}$, where $v_j=d^{1/2-1/\dualnumber{p}}\hat e_j$, is such that these vectors are orthogonal and have unit $\ell_{\dualnumber{p}}$-norm.

Using minimax duality, for every $\xi\in\{-1,+1\}^T$
$$\min_{x\in{\cal X}}\max_{j\in[T]} \xi_j\langle v_j,x\rangle  \leq \min_{\norm{x}\leq 1}\max_{\lambda\in\Delta_T} \sum_{j\in[T]}\lambda_j\langle v_j,x\rangle 
    = -\min_{\lambda\in\Delta_T}\Big\|\sum_{j\in[T]}\lambda_j \xi_j v_j\Big\|_{\dualnumber{p}}.$$

In order to estimate this quantity. Consider first the $\|\cdot\|_2$:
\[ \Big\|\sum_{j\in[T]}\lambda_j \xi_j v_j\Big\|_2=\sqrt{\sum_{j\in[T]}\lambda_j^2 \|v_j\|_2^2}= T^{\frac12-\frac{1}{\dualnumber{p}}}\frac{1}{\sqrt{T}}=T^{-1/\dualnumber{p}}.
\]
Now, using H\"older's inequality:
\[ T^{-\frac{1}{\dualnumber{p}}}=
\Big\|\sum_{j\in[T]}\lambda_j \xi_j v_j\Big\|_{2}
\leq T^{\frac{1}{2}-\frac{1}{\dualnumber{p}}} \Big\|\sum_{j\in[T]}\lambda_j \xi_j v_j\Big\|_{\dualnumber{p}},  \]
hence $\min_{x\in{\cal X}}\max_{j\in[T]} \xi_j\langle v_j,x\rangle \leq-\frac{1}{\sqrt T}$, and $\Theta = \frac{1}{\sqrt{T}}$. The optimality gap is then
$$F(x_T) - \inf_{x\in \mathcal{X}}F(x) \geq \bigomegatildelp{q}{\L\frac{\Theta^{\q+1}}{T^{\q}(\ln d)^{\q}}} = \bigomegatildelp{\q,\p}{LT^{-\frac{3\q+1}{2}}},$$
so at least $\bigomegatildelp{\q,\p}{\left(\frac{\L}{\epsilon}\right)^{\frac{2}{3\q+1}}}$ iterations are needed to reach the precision $\epsilon$. 
\end{itemize}

\subsection{Randomized and parallel methods} \label{sec:lower_bounds_random}

In this section we prove the result in \cref{thm:lower-bound_random} for possibly randomized and parallel algorithms that interact with a local oracle. 

In the \(K\)-parallel framework for convex optimization \citep{nemirovski1994parallel}, algorithms operate iteratively across multiple rounds. During each round, the algorithm issues a batch of queries denoted by \(X_t = \{x_{t,1}, \ldots, x_{t,K}\}\). In response, the local oracle \(\mathcal{O}\) provides a batch of outputs, represented as \(\mathcal{O}_F(X_t) = (\mathcal{O}_F(x_{t,1}), \ldots, \mathcal{O}_F(x_{t,K}))\). The algorithm's behavior may adapt over successive rounds, with each new batch of queries depending on prior queries and the corresponding oracle responses:
\[
X_{t+1} = \Psi_{t+1}(X_1, \mathcal{O}_F(X_1), \ldots, X_t, \mathcal{O}_F(X_t)), \quad \forall t \geq 1,
\]
where \(\Psi_{t+1}\) defines the update (possible randomized) mechanism for generating the next batch of queries. Notably, setting \(K = 1\) recovers the standard definition of sequential oracle complexity. The following theorem is the extension of \cref{thm:lower-bound} for $K$-parallel randomized algorithms, which we prove in \cref{sec:lower_bounds_random}. 

\begin{theorem}[Lower bound for parallel randomized algorithms] \label{thm:lower-bound_random}\linktoproof{thm:lower-bound_random}
    Let $\norm{\cdot}$ a norm in $\Rd$ and $\mathcal{X}$ a closed convex set containing the $R$-ball $B_R^{\norm{\cdot}}$ of $(\Rd,\norm{\cdot})$ for some $R>0$. Let $T$ a positive integer,  $\Theta,\; \tilde{M} >0$ real numbers, 
  $0<\eta<1/2$ a probability, and $\{z_i\}_{i\in [T]}$ independent random vectors in $\Rd$ such that: \\
    $(i)\;$ $\norm{z_i}_\ast\leq 1$ for every $i\in [T]$, \\
    $(ii)\; $ $\mathbb{P}[\min_{x\in \mathcal{X}}\max_{i\in [T]} \langle z_i,x\rangle\leq -\Theta]\geq 1-\eta,$ \\
    $(iii) \;$  For every $i\in [T]$, $x\in \mathcal{X}$ and $\delta>0$, $\max\{\mathbb{P}[\langle z_i,x\rangle\geq \delta],\mathbb{P}[\langle z_i,x\rangle\leq -\delta]\}\leq \exp(-\tilde{M}\delta^2)$ \\
    $(iv) \;$ $\Theta \geq 64T\sqrt{\ln (TK/\eta)/\tilde{M}}$.
        
    Then, for every $\L>0$, $\nu\in (0,1]$, there exists a family of $\q$-th order $(\L,\nu)$-H\"older continuous functions $\mathcal{F}$ such that for any $K$-parallel algorithm $\mathcal{A}_K$ interacting with a local oracle $\mathcal{O}$ it holds
    $$\mathbb{P}_{F\sim \Delta(\mathcal{F})}\left[\min_{t\in [T],k\in [K]}F(x_{t,k}) - \min_{x\in \mathcal{X}}F(x)\geq \bigomegatildelp{\q}{\L R^{q+\nu}\frac{\Theta^{\q+\nu}}{T^{\q+\nu-1}}}\right] \geq 1-2\eta,$$
    where $\{x_{t,k}\}_{t\in [T],k\in[K]}$ is the sequence generated by the pair $(\mathcal{A}_K, \mathcal{O})$.
\end{theorem}

\begin{proof}\linkofproof{thm:lower-bound_random}
The lower bound for randomized algorithms relies on two properties. First, an upper bound on the minimal value of $F$ that holds with high probability, and second, a lower bound on the function value of the algorithm's output that also holds with high probability. Let us start by considering a set $\mathcal{X}$ containing the unit ball, and recall the construction of the hard instance function

For $i=1,\ldots,T$ define the functions $f_i:\Rd\mapsto \R$ we define
$$
    f_i(x) \defi \smaxn{i}{\mu}((\langle z_j, x\rangle + (T-j)\gamma)_{j\in[d]}) + \mu(T+1-i)d^{-\alpha},
$$
and 
$$
    h(x) \defi \max_{i\in[T]}f_i(x), \qquad g(x)\defi\S{\q}{\beta}[h](x).
$$
    Here, $z_i$ are random vectors as described in the statement of \cref{thm:lower-bound_random}, and the parameters are chosen as before 
$$
    \gamma = \frac{\Theta}{4T}, \quad \mu = \frac{\gamma}{4\alpha\ln d}, \quad \beta = \frac{\gamma}{\ln d}, \quad \alpha \geq \q+1.
$$
\cref{lemma:Lq} implies that $g$ is $\q$-th order smooth with constant $L_{\q}$. Moreover, as in \cref{eq:bound_fi} we can upper bound the minimal value of $g$ over $\mathcal{X}$ as follows:
\begin{align*}
    \min_{x\in \mathcal{X}} g(x) & \leq \min_{x\in \mathcal{X}} h(x) + 2\beta \\
    & \leq \mu \ln T + \min_{x\in \mathcal{X}}\max_{i\in [T]} \langle z_i,x\rangle+T\gamma+\mu(T+1)d^{-\alpha}+2\beta\\
    & \leq \mu \ln T - \Theta +T\gamma+\mu(T+1)d^{-\alpha}+2\beta
\end{align*}
To lower bound $g(x_T)$ the key idea is to show that, at each round $t$, w.h.p., the algorithm can only access information about $z_1, \ldots, z_t$ and has no knowledge of $z_{t+1}, \ldots, z_k$. We denote the history of the algorithm-oracle interaction until iteration $t-1$ as $\Pi^{t} = (X_s, \mathcal{O}(X_s))_{s<t}$. We also define the following events
\[
\mathcal{E}^t(x) \defi \left\{\langle z_i, x \rangle > -\frac{\gamma}{4} \right\} \cap \left\{\langle z_i, x \rangle < \frac{\gamma}{4} \, (\forall i > t) \right\}, \quad \mbox{and} \quad \mathcal{E}^{<t} \defi \bigcap_{s<t, k\in [K]} \{\mathcal{E}^s(x_{s,k})\},
\]
where $\gamma>0$ is a parameter to be determined and $\mathcal{E}^{<1}$ is such that $\mathbb{P}[\mathcal{E}^{<1}]=1$. By \cref{fact:rlb1}, we have in particular that w.p. at least $1-\eta$ for every $t\in T$
$$
    g(x_t)\geq h(x_t)-2\beta \geq f_t(x_t)-2\beta \geq \langle z_t,x_t\rangle-2\beta \geq -\frac{\gamma}{4}-2\beta.
$$
Putting the results together, rescaling $g$ by $F=\L/L_{\q}g$, and using the value of the parameters we obtain the result.
\end{proof}

Under the assumptions of \cref{thm:lower-bound_random}, the following fact\footnote{In \citep{diakonikolas2020lower}, this claim mentions the predictability of $X_t$ with respect to $\{z_s\}_{s<t}$, conditionally on $\mathcal{E}^{<t}$. This claim is incorrect, as the algorithm is randomized. Instead, the correct affirmation is that $X^t$ is conditionally independent, which suffices for the high-probability conclusion.}
directly follows from \citep{diakonikolas2020lower}, which we used in the proof of \cref{thm:lower-bound_random}. 

\begin{fact} \label{fact:rlb1}
    Let $t<T$ and assume that event $\mathcal{E}^{t}$ holds. Then, for all $k\in [K]$ and $x\in B^{\norm{\cdot}}_r(x_{t,k})$, $g(x)$ is fully determined by vectors $z_s$ with $s\leq t$. Moreover,
    $X_t$ is independent of $\{z_s\}_{s\geq t}$, conditionally on $\mathcal{E}^{<t}$, and
    $\mathbb{P}\left[\bigcap_{t\in [T]}\mathcal{E}^{t}\right]\geq 1-\eta$. 
\end{fact}

\subsection{The case of \texorpdfstring{$p$}{p}-norms in randomized and parallel methods} \label{sec:p_norms_randomized_parallel}

To specialize the result for $p$-norms we need to estimate the value of $\Theta$ in $(ii)$ of \cref{thm:lower-bound_random}. We separate the cases $p\geq 2$ and $1\leq p < 2$.

\begin{itemize}
    \item For $\p\geq 2$, the construction is as follows. Let $\{J_i\}_{i=1}^T$ be a collection 
of subsets of $\{1, \ldots, d\}$ such that $|J_i| = M$ and $J_i \cap J_{i'} = \emptyset, \, \forall i \neq i'$. Here $M$ is an integer such that $d \geq TM$. Set $I_i^M = \text{diag}(1_{J_i})$, i.e., the $(j,j)$ element of the 
diagonal matrix $I_i^M$ is 1 if $j \in J_i$ and 0 otherwise. The vector $z_i$ is defined as 
\[
    z_i \defi \frac{1}{M^{1/\dualnumber{p}}} I_i^M \xi_i,
\]
        where $(\xi_i) \in \{-1, 1\}^{d}$ is an independent Rademacher sequence. 

Using minimax duality we have 
$$
    \min_{x\in \mathcal{X}}\max_{i\in [T]} \langle z_i,x\rangle \leq \min_{\norm{x}\leq 1}\max_{\lambda\in \Delta_T} \left\langle \sum_{i\in [T]} \lambda_iz_i, x\right\rangle  = -\min_{\lambda\in 
\Delta_T}\left\| \sum_{i\in [T]} \lambda_i z_i \right\|_{\dualnumber{p}}.
$$

Let $\lambda \in \Delta_T$ be fixed. Observe that, since $z_i$'s have disjoint support 
(each $z_i$ is supported on $J_i$ such that $|J_i| = M$ and $J_i \cap J_{i'} = \emptyset$ 
for all $i \neq i'$), vector $\sum_{i \in [T]} \lambda_i z_i$ is such that its coordinates 
indexed by $j \in J_i$ ($M$ of them) are equal to $\lambda_i z_{j,i}, \, \forall i \in [T]$. 
Therefore, using the definition of $z_i$ 
\[
\left\| \sum_{i\in [T]} \lambda_i z_i \right\|^{\dualnumber{p}}_{\dualnumber{p}} =\sum_{i\in [T]} \left(M \cdot \left(\lambda_i M^{-1/\dualnumber{p}}\right)^{\dualnumber{p}}\right) = \|\lambda\|_{\dualnumber{p}}^{\dualnumber{p}}.
\]
By the relationship between $\ell_{\p}$ norms and the definition of $\lambda$, we have that $1 = \|\lambda\|_1 \leq T^{1/\p} \|\lambda\|_{\dualnumber{p}}$. Hence
$$
    \min_{\lambda\in \Delta_T}\left\|\sum_{i\in [T]}\lambda_i z_i\right\|_{\dualnumber{p}} = \norm{\lambda}_{\dualnumber{p}}\geq T^{-1/\p},
$$
and condition $(ii)$ in \cref{thm:lower-bound_random} is satisfied with $\Theta= T^{-1/\p}$ for any $\eta\geq 0$. 

On the other hand, by the definition of $z_i$'s and Hoeffding's Inequality, for all $x\in B^{\norm{\cdot}_{\p}}$, $\delta>0$
$$
    \mathbb{P}[\langle z_i, x\rangle>\delta] = \mathbb{P}[\langle z_i,x\rangle<-\delta] = \mathbb{P}\left[\sum_{j\in J_i} \xi_i[j]x_j>\delta M^{1/\dualnumber{p}}\right]\leq \exp\left(-\frac{M^{2/\dualnumber{p}}\delta^2}{2\sum_{j\in J_i}x_j^2}\right),
$$
where $\xi_i[j]$ is the $j$-th coordinate of $\xi_i$.
As $|J_i| = M$, using the relations between $\p$-norms
$$
    \norm{\{x_j\}_{j\in J_i}}_2\leq M^{1/2-1/\p}\norm{\{x_j\}_{j\in J_i}}_{\p} \leq M^{1/2-1/\p}\norm{x}_{\p} \leq  M^{1/2-1/\p}.
$$
Therefore, 
$$
    \mathbb{P}[\langle z_i, x\rangle>\delta] = \mathbb{P}[\langle z_i,x\rangle<-\delta] \leq \exp\left(-\frac{M^{2/\dualnumber{p}}\delta^2}{2M^{1-2/\p}}\right)=\exp\left(-\frac{M\delta^2}{2}\right).
$$
        Hence, condition $(iii)$ in \cref{thm:lower-bound_random} holds with $\tilde{M} \defi \lfloor d/(2T) \rfloor$. Finally, to guarantee condition $(iv)$ it is enough to take the dimension large enough, namely $d\geq \bigomegal{T^{3+2/p}\ln(TK/\eta)}$.

\item For $1\leq p < 2$, define $z_i = d^{1/\dualnumber{p}}\xi_i$ where $\xi_i\in \{-1,1\}^d$ are independent vectors with Rademacher entries. It is easy to check that $\norm{z_i}_{\dualnumber{p}} \leq 1$. Moreover, using minimax duality
$$
    \min_{x\in \mathcal{X}}\max_{i\in [T]} \langle z_i,x\rangle \leq \min_{\norm{x}\leq 1} \max_{\lambda\in \Delta_T}\langle \sum_{i\in [T]}\lambda_iz_i,x\rangle = -\min_{\lambda\in \Delta_T}\left\Vert\sum_{i\in [T]}\lambda_iz_i\right\Vert_{\dualnumber{p}}
$$

Let $\epsilon$ and $c_{\q}$ be a constant that only depends on $\q$. Using \citep[Lemma 23]{diakonikolas2020lower}, for $T\leq \min\{\frac{1}{200\epsilon^2}, \frac{c_{\q}d-\ln(1/\eta)}{\ln(3/\epsilon)}\}$ it holds
$$
    \mathbb{P}\left[\left\Vert \sum_{i\in [T]}\lambda_iz_i\right\Vert_{\dualnumber{p}} \leq 4\epsilon\right]\leq \eta.$$
Taking $\epsilon = \frac{1}{\sqrt{200 T}}$ and $d\geq \bigomegapl{\q}{T\ln(3\sqrt{200T})+\ln(1/\eta)}$ we obtain that condition $(ii)$ in \cref{thm:lower-bound_random} holds with $\Theta = \frac{\sqrt{2}}{5\sqrt{T}}$. On the other hand, by a direct application of the Hoeffding's inequality for every $x\in B^{\norm{\cdot}_p}$ 
$$
    \mathbb{P}[\langle z_i,x\rangle>\delta]=\mathbb{P}[\langle \xi_i, x\rangle>d^{1/\dualnumber{p}}\delta] \leq \exp\left(-d^{2/\dualnumber{p}}\delta^2\right).
$$
Hence, condition $(iii)$ in \cref{thm:lower-bound_random} holds with $\tilde{M} = d^{2/\dualnumber{p}}$, and condition $(iv)$ reads $d\geq \bigomegal{\left(T^{3/2}\sqrt{\ln(TK/\eta)}\right)^{\dualnumber{p}}}$.
\end{itemize}
\end{document}

%% file: algorithm_config.tex
\usepackage{algorithm}
\usepackage{algcompatible}
\algnewcommand{\lst}{\texttt{lst}}
\algnewcommand{\slst}{\texttt{slst}}
\algnewcommand{\SEND}{\textbf{send}}

\newsavebox{\algleft}
\newsavebox{\algright}

\makeatletter
\newcounter{algorithmicH}
\let\oldalgorithmic\algorithmic
\renewcommand{\algorithmic}{%
  \stepcounter{algorithmicH}
  \oldalgorithmic}
\renewcommand{\theHALG@line}{ALG@line.\thealgorithmicH.\arabic{ALG@line}}
\makeatother

%% file: todo_config.tex
\usepackage{todonotes}
\usepackage{xcolor}
\newif\ifTODO 

\TODOfalse

\ifTODO
\paperwidth=\dimexpr \paperwidth + 6cm\relax
\oddsidemargin=\dimexpr\oddsidemargin + 3cm\relax
\evensidemargin=\dimexpr\evensidemargin + 3cm\relax
\marginparwidth=\dimexpr \marginparwidth + 3cm\relax
\fi


%% file: bibliography_config.tex
\usepackage[natbib, backend=biber, maxcitenames=3, minalphanames=3, maxbibnames=99, style=alphabetic, hyperref, backref, useprefix=true, uniquename=false, doi=false,url=false,eprint=false]{biblatex}

\usepackage{csquotes}               

\bibliography{refs} 

\newbibmacro{string+doiurlisbn}[1]{%
  \iffieldundef{doi}{%
    \iffieldundef{url}{%
      \iffieldundef{isbn}{%
        \iffieldundef{issn}{%
          #1%
        }{%
          \href{http://books.google.com/books?vid=ISSN\thefield{issn}}{#1}%
        }%
      }{%
        \href{http://books.google.com/books?vid=ISBN\thefield{isbn}}{#1}%
      }%
    }{%
      \href{\thefield{url}}{#1}%
    }%
  }{%
    \href{https://doi.org/\thefield{doi}}{#1}%
  }%
}

\DeclareFieldFormat{title}{\usebibmacro{string+doiurlisbn}{\mkbibemph{#1}}}
\DeclareFieldFormat[article,incollection,inproceedings,inbook,book,phdthesis]{title}%
    {\usebibmacro{string+doiurlisbn}{\mkbibquote{#1}}}

%% file: definitions.tex
\usepackage{xparse}
\usepackage{xargs} %

\newcommand\newlink[2]{\hyperlink{#1}{\normalcolor #2}}
\makeatletter
\newcommand\newtarget[2]{\Hy@raisedlink{\hypertarget{#1}{}}#2}
\makeatother

\newcommand\linktoproof[1]{{\normalfont[{\hyperlink{proof:#1}{$\downarrow$}}]}}
\newcommand\linkofproof[1]{\textbf{of \cref{#1}. }\newtarget{proof:#1}}

\newcommand\setsuch[2]{\{ #1 \mid #2  \}}

\newcommand{\bigo}[1]{O( #1 )}

\newcommand{\bigop}[2]{\newlink{def:big_O_p}{O_{#1}}( #2 )}
\newcommand{\bigopl}[2]{\newlink{def:big_O_p}{O_{#1}}\left( #2 \right)}

\newcommand{\bigotilde}[1]{\newlink{def:big_O_tilde}{\widetilde{O}}( #1 )}
\newcommand{\bigotildep}[2]{\newlink{def:big_O_p}{\widetilde{O}_{#1}}( #2 )} %
\newcommand{\bigotildepl}[2]{\newlink{def:big_O_p}{\widetilde{O}_{#1}}\left( #2 \right)}
\newcommand{\bigomega}[1]{\Omega( #1 )}
\newcommand{\bigomegal}[1]{\Omega\left( #1 \right)}
\newcommand{\bigomegatilde}[1]{\newlink{def:big_O_tilde}{\widetilde{\Omega}}( #1 )}

\newcommand{\bigomegatildelp}[2]{\newlink{def:big_O_p}{\widetilde{\Omega}_{#1}}\left( #2 \right)}

\newcommand{\bigomegapl}[2]{\newlink{def:big_O_p}{\Omega_{#1}}\left( #2 \right)}

\let\oldepsilon\epsilon
\renewcommand\epsilon{\newlink{def:accuracy}{\oldepsilon}} 

\let\oldnu\nu
\renewcommand\nu{\newlink{def:hoelder_smoothness_exponent}{\oldnu}} %

\renewcommand\L{\newlink{def:hoelder_smoothness_constant}{L}}

\renewcommand\d{\mathrm{d}}

\newcommandx\regret[2][2= , usedefault]{\newlink{def:acc_sec_regret}{R_{#1}^{#2}}}

\newcommandx*\Gt[1][1=t, usedefault]{\newlink{def:acc_sec_differentiable_gap_bound}{G_{#1}}}
\newcommandx*\Lt[1][1=t, usedefault]{\newlink{def:acc_sec_differentiable_lower_bound}{L_{#1}}}
\newcommandx*\Ut[1][1=t, usedefault]{\newlink{def:acc_sec_differentiable_upper_bound}{U_{#1}}}

\newcommandx*\etai[2][1=t, 2=, usedefault]{\newlink{def:acc_sec_sequence_of_lr_for_accelerated_rates}{\eta_{#1}^{#2}}}

\newcommandx*\aig[2][1=t, 2=, usedefault]{a_{#1}^{#2}}

\newcommand*\circledaux[1]{\tikz[baseline=(char.base)]{
    \node[shape=circle,draw,inner sep=0.8pt] (char) {#1};}}

\NewDocumentCommand{\circled}{ m o }{%
    \IfNoValueTF{#2}{ \circledaux{#1} }{ \stackrel{\circledaux{#1}}{#2} }%
}

\newcommandx*\prox[1][1=, usedefault]{\newlink{def:proximal_operator}{\operatorname{prox}_{#1}}}
\newcommand\Prox{\newlink{def:set_of_all_prox_points}{\operatorname{Prox}}}

\newcommandx*\yktilde[1][1=k, usedefault]{\newlink{def:y_k_tilde}{\tilde{y}_{#1}}} 
\newcommandx*\xk[1][1=k, usedefault]{x_{#1}} 
\newcommandx*\yk[1][1=k, usedefault]{y_{#1}} 
\newcommandx*\zk[1][1=k, usedefault]{z_{#1}} 
\newcommandx*\Ak[2][1=k, 2=, usedefault]{A_{#1}^{#2}}
\newcommandx*\ak[2][1=k, 2=, usedefault]{a_{#1}^{#2}}

\newcommandx*\breg[1][1=\psi, usedefault]{\newlink{def:bregman}{D_{#1}}}
\newcommandx*\bregd[1][1=\psi^{\ast}, usedefault]{\newlink{def:bregman}{D_{#1}}} %
\newcommand\xast{\newlink{def:xast}{x^\ast}}

\usepackage{cancel}
\newcommand\Ccancel[2][black]{
    \let\OldcancelColor\CancelColor
    \renewcommand\CancelColor{\color{#1}}
    \cancel{#2}
    \renewcommand\CancelColor{\OldcancelColor}
}

\newcommandx*\Rp[1][1=, usedefault]{\mathbb{R}_{\geq 0}}

\newcommand\smax[1]{\newlink{def:softmax}{\operatorname{smax}_{#1}}}
\newcommand\smaxn[2]{\newlink{def:partial_softmax}{\operatorname{smax}^{\leq #1}_{#2}}}

\renewcommand\vol{\operatorname{vol}}
\newcommand\sign{\operatorname{sign}}
\newcommand\subgrad{\partial}
\newcommand\dualnumber[1]{\newlink{def:youngs_conjugate_number}{{#1}_\ast}}

\renewcommand\S[2]{\newlink{def:q_fold_smoothing}{\mathcal{S}^{(#1)}_{#2}}}
\newcommand\smoothing{\newlink{def:1_fold_smoothing}{S}}

\newcommand\elll[1]{\newlink{def:thing_minimized_in_the_lower_bound}{\ell_{#1}}}

\newcommandx*\subdiffeps[2][1=\epsilon, 2=, usedefault]{\newlink{def:epsilon_subdifferential}{\partial^{#1}_{#2}}}

\newcommand\deltainexact{\newlink{def:delta_inexact}{\delta}}
\newcommand\p{\newlink{def:p_from_p_norm}{p}}
\renewcommand\r{\newlink{def:exponent_of_regularizers_unif_convexity}{r}}
\newcommand\m{\newlink{def:m_max_of_p_and_2}{m}}
\newcommand\q{\newlink{def:order_of_derivative_for_smoothness}{q}}
\newcommand\muUnif{\newlink{def:regularizers_unif_cvx_constant}{\mu}}

\renewcommand\dim{\newlink{def:dimension}{d}} 
\newcommand\Rd{\R^{\dim}} 
\newcommand\NN{\mathbb{N}}

\newcommand\eventindicator[1]{\newlink{def:event_indicator_function}{\mathbb{I}_{\{#1\}}}}

\newcommand\indicator[1]{\newlink{def:indicator_function}{I}_{#1}}

\newcommandx*\ballnorm[2][1=\beta, 2=\norm{\cdot}, usedefault]{\newlink{def:ball_norm}{B^{#2}_{#1}}}

\newcommandx*\taylorf[2][1=q, 2=f, usedefault]{\newlink{def:qth_taylor_expansion}{#2_{#1}}}

\newcommandx*\M[1][1=, usedefault]{\newlink{def:non_euclidean_moreau_envelope}{M_{#1}}}

\newcommandx*\iproxoracle[1][1=, usedefault]{\newlink{def:inexact_proximal_oracle}{\mathcal{O}_{#1}}}
\newcommandx*\iproxoraclehat[1][1=, usedefault]{\newlink{def:generalized_inexact_proximal_oracle}{\widehat{\mathcal{O}}_{#1}}}

\newcommand\subdiffsqnorm[1]{\newlink{def:subdiff_of_squared_norm}{h_{#1}}}